\documentclass[11pt,a4paper, envcountsame]{amsart}

 \usepackage{amsmath, amssymb, xspace}

 \usepackage{graphicx}

 \usepackage[all]{xy}

\usepackage{undertilde}

\newtheorem{theorem}{Theorem}[section]

\newtheorem{thm}[theorem]{Theorem}

\newtheorem{fact}[theorem]{Fact}

\newtheorem{proposition}[theorem]{Proposition}

\newtheorem{remark}[theorem]{Remark}

\newtheorem{prop}[theorem]{Proposition}

\newtheorem{claim}[theorem]{Claim}

\newtheorem{lemma}[theorem]{Lemma}

\newtheorem{corollary}[theorem]{Corollary}

\newtheorem{cor}[theorem]{Corollary}

\newtheorem{question}[theorem]{Question}

\theoremstyle{definition}
\newtheorem{definition}[theorem]{Definition}

\newcommand{\spe}{\mathrm{Spec}}

\newcommand{\CK}{\omega_1^{\mathrm{CK}}}
\newcommand{\KO}{\mathcal{O}}
\newcommand{\DII}{\Delta^0_2}
\newcommand{\NN}{{\mathbb{N}}}
\newcommand{\RR}{{\mathbb{R}}}
\newcommand{\R}{{\mathbb{R}}}
\newcommand{\QQ}{{\mathbb{Q}}}
\newcommand{\ZZ}{{\mathbb{Z}}}
\newcommand{\sub}{\subseteq}
\newcommand{\sN}[1]{_{#1\in \NN}}
\newcommand{\uhr}[1]{\! \upharpoonright_{#1}}
\newcommand{\ML}{Martin-L{\"o}f}
\newcommand{\SI}[1]{\Sigma^0_{#1}}
\newcommand{\PI}[1]{\Pi^0_{#1}}
\newcommand{\PPI}{\PI{1}}

\renewcommand{\P}{\mathcal P}

\newcommand{\bi}{\begin{itemize}}
\newcommand{\ei}{\end{itemize}}
\newcommand{\bc}{\begin{center}}
\newcommand{\ec}{\end{center}}

\newcommand{\Halt}{{\ES'}}
\newcommand{\ES}{\emptyset}
\newcommand{\estring}{\la \ra}
\newcommand{\ria}{\rightarrow}
\newcommand{\tp}[1]{2^{#1}}
\newcommand{\ex}{\exists}
\newcommand{\fa}{\forall}

\newcommand{\la}{\langle}
\newcommand{\ra}{\rangle}

\newcommand{\seqcantor}{2^{ \NN}}

\newcommand{\cantor}{\seqcantor}
\newcommand{\strcantor}{2^{ < \omega}}

\newcommand{\strbaire}{\omega^{ < \omega}}

\newcommand{\Opcl}[1]{[#1]^\prec}
\newcommand{\leT}{\le_{\mathrm{T}}}

\renewcommand{\P}{\mathcal P}
\newcommand{\MLR}{\mbox{\rm \textsf{MLR}}}
\newcommand{\Om}{\Omega}
\newcommand{\n}{\noindent}

\newcommand{\vsp}{\vspace{6pt}}
\newcommand{\leb}{\mathbf{\lambda}}

\newcommand{\sss}{\sigma}
\newcommand{\aaa}{\alpha}

\newcommand{\s}{\sigma}

\newcommand{\rest}[1]{\! \upharpoonright_{#1}}

\newcommand{\lland}{\, \land \, }

\newcommand \seq[1]{{\left\langle{#1}\right\rangle}}

\newcommand\+[1]{\mathcal{#1}}
\newcommand{\sC}{\+ C}

\newcommand{\wt}{\widetilde}
\newcommand{\ol}{\overline}
\newcommand{\ul}{\underline}

\newcommand{\lra}{\leftrightarrow}
\newcommand{\LR}{\Leftrightarrow}
\newcommand{\RA}{\Rightarrow}
\newcommand{\LA}{\Leftarrow}

\newcommand{\rapf}{\n $\RA:$\ }
\newcommand{\lapf}{\n $\LA:$\ }

\newcommand{\UM}{\mathbb{U}}
\newcommand{\sssl}{\ensuremath{|\sigma|}}

\newcommand{\dom}{\ensuremath{\mathrm{dom}}}

\def\uh{\upharpoonright}

\newcommand{\dset}[2]{\{#1 : #2 \}}

  \newcommand{\SR}{\mbox{\rm \textsf{SR}}}
  \newcommand{\CR}{\mbox{\rm \textsf{CR}}}
    \newcommand{\rank}{\mbox{\rm \textsf{rank}}}
    
\newcommand{\RCA}{\ensuremath{\mathbf{RCA_0}}} \newcommand{\WKL}{\ensuremath{\mathbf{WKL_0}}} \newcommand{\ACA}{\ensuremath{\mathbf{ACA_0}}}

\begin{document}

\begin{abstract}  The 2013  logic blog has focussed on the following:

\n 1.  Higher randomness. Among others, the Borel complexity of $\Pi^1_1$ randomness and higher weak 2 randomness is determined.

\n 2. Reverse mathematics and its relationship to randomness.  For instance, what is the strength of Jordan's theorem in analysis?  (His  theorem states that each  function of bounded variation is the difference of two nondecreasing functions.)

\n 3. Randomness and computable analysis. This focusses on the connection of randomness of a real $z$ and Lebesgue density of effectively closed sets at $z$.

\n 4. Exploring similarity relations for Polish metric spaces, such as isometry, or having Gromov-Hausdorff distance $0$. In particular   their complexity was studied.

\n 5. Various results connecting computability theory and randomness.
 \end{abstract}

\title{Logic Blog 2013}

 \author{Editor: Andr\'e Nies}

\maketitle

Previous Logig Blogs from 2010 on can be found on Nies' web site.

  
%
%
%
%

  %
\medskip

\tableofcontents



\newpage


\newpage

\part{Higher randomness}

\section{Greenberg, Monin: An upper bound on the Borel rank of the set  of $\Pi^1_1$-random reals}
Written by Benoit Monin in August, joint work with Noam Greenberg.
\\

Recall that a set  $Z \in \cantor$ is $\Pi^1_1$-random if it is in no $\Pi^1_1$ null class. Kechris~\cite{Kechris:75}  showed  that there is a largest $\Pi^1_1$ null class, which can be seen as  a universal test for $\Pi^1_1$-randomness.  A simple direct proof of this fact is in the last section of Hjorth and Nies~\cite{Hjorth.Nies:07}. For background on higher randomness see \cite[Ch.\ 9]{Nies:book}.

We show that the set of $\Pi^1_1$-randoms is $\mathbf{\Pi^0_3}$. Together with Yu Liang's result in Section~\ref{s:GB-Monin sharpness}, this give the exact Borel rank of the $\Pi^1_1$-random reals (thus also the exact Borel rank of the Kechris' largest $\Pi^1_1$ nullset).

We will show that being $\Pi^1_1$-random is equivalent to a certain notion of genericity. The class of elements for this notion of genericity will have the Borel complexity of $\mathbf{\Pi^0_3}$. This notion of genericity is a variation of the notion of forcing with $\Pi^0_1$ class of positive measure, where we use the same idea lying in the difference between $1$-generic and weakly-$1$-generic.

Following the thesis of Kautz where forcing with closed classes of positive measure is called Solovay forcing, we introduce the two notions of \textbf{weakly-Solovay-$\Sigma^1_1$-generic} real and \textbf{Solovay-$\Sigma^1_1$-generic} reals:

\begin{definition}
We say that $X$ is weakly Solovay-$\Sigma^1_1$-generic if for any uniformly  $\Sigma^1_1$ sequence $\{F_n\}_{n \in \omega}$ of closed sets of positive measure  with $\lambda(\bigcup_n F_n)=1$ we have that $X$ is in one of the $F_n$.
\end{definition}

It is easy to see that weak Solovay-$\Sigma^1_1$-genericity is the same as  higher weak 2-randomness (sometimes  called strong $\Pi^1_1$-ML-randomness). We now give the notion of genericity that  will turn out  to be equivalent to $\Pi^1_1$-randomness:

\begin{definition}
We say that $X$ is Solovay-$\Sigma^1_1$-generic if for any  uniformly  $\Sigma^1_1$ sequence $\{F_n\}_{n \in \omega}$ of closed set of positive measure, $\Sigma^1_1$, we have either $X$ is in one of the $F_n$ or there is a $\Sigma^1_1$ closed set of positive measure $G$ such that $G \cap \bigcup F_n =\emptyset$ and $X \in G$.
\end{definition}

Clearly, the     two last definitions are related  between each other in the same way as    $1$-genericity is related to  weak  $1$-genericity. This   justifies the terms weakly-Solovay-$\Sigma^1_1$-generic and Solovay-$\Sigma^1_1$-generic. We now have to prove that this last notion of genericity coincides with  $\Pi^1_1$-randomness. The only difficult part of the demonstration is to show that if $X$ is Solovay-$\Sigma^1_1$-generic then $\omega_1^X$ is equal to $\omega_1^{ck}$. In order to prove this, we use the idea imaginated by Sacks and simplified by Greenberg, to show that the set of $X$ with $\omega_1^X > \omega_1^{ck}$ has measure $0$.\\

The idea is the following, suppose that for some $X$ we have a function $\varphi$ such that:
$$\forall n\ \ \exists \alpha < \omega_1^{ck}\ \ \varphi^X(n) \in \KO_\alpha^X$$
where $\KO^X$ is the set of Kleene's notation for ordinals computable in $X$ and $\KO^X_\alpha$ the set of Kleene's notation for ordinals computable in $X$ with order-type strictly smaller than $\alpha$. Suppose also that $X$ is Solovay-$\Sigma^1_1$-generic. Then we will show that the supremum of $\varphi^X(n)$ over $n \in \omega$ is smaller than $\omega_1^{ck}$. To show this we need two lemmas:

\begin{lemma} \label{approx}
Let $S$ be a $\Sigma^1_1$ predicate of positive measure of the form
$$S(X) \leftrightarrow \exists n\ \ \forall \alpha < \omega_1^{ck}\ \ S_{\alpha, n}(X)$$
where $S_{\alpha, n}$ is a $\Delta^1_1$ predicate uniformly in $n$ and $\alpha$. Then there is a union of uniformly $\Sigma^1_1$ closed set $\bigcup_n F_n \subseteq S$ with $\lambda(S - \bigcup_n F_n) = 0$.
\end{lemma}
\begin{proof}
Let $S_n=\{X\ | \forall \alpha < \omega_1^{ck}\ \ S_{\alpha, n}(X)\}$. So we have $S=\bigcup_n S_n$ and $S_n = \bigcap_{\alpha < \omega_1^{ck}} S_{\alpha, n}$. Let us fix $n$ and let us build a union of uniformly $\Sigma^1_1$ closed set $\bigcup_m F_{m,n} \subseteq S_n$ with $\lambda(S_n - F_{m,n}) < 2^{-m}$.\\

For each $m$, in each $S_{\alpha, n}$, find a $\Sigma^1_1$ closed set $F_{\alpha, m, n}$ with $F_{\alpha, m, n} \subseteq S_{\alpha, n}$ and $\lambda(S_{\alpha, n} - F_{\alpha, m, n}) < 2^{-p(\alpha)}2^{-m}$ where $p$ is an injection of $\omega_1^{ck}$ into $\omega$. Let us set $F_{m,n}=\bigcap_{\alpha < \omega_1^{ck}} F_{\alpha, m, n}$. As the intersection of closed set $\bigcap_\alpha F_{\alpha, m, n}$ is a closed set and as the predicate $\forall \alpha\ \ X \in F_{\alpha, m, n}$ is clearly a $\Sigma^1_1$ predicate, we have that $F_{m,n}$ is a $\Sigma^1_1$ closed set. Also as $S_n - F_{m,n}=\bigcup_\alpha S_n - F_{\alpha, m,n}$ we have:
$$
\begin{array}{rcl}
\lambda(S_n - F_{m,n})&\leq&\lambda(\bigcup_\alpha S_n - F_{\alpha, m,n}) \\
&\leq&\lambda(\bigcup_\alpha S_{n, \alpha} - F_{\alpha, m,n}) \\
&\leq&\sum_\alpha \lambda(S_{\alpha, n} - F_{\alpha, m, n}) \leq 2^{-m}
\end{array}
$$

Then, uniformly in $n$ and $m$ we have sequenes of $\Sigma^1_1$ closed set $F_{n,m} \subseteq S_n$ such that $\lambda(S_n - F_{n,m}) < 2^{-m}$. Any $\omega$-order of $\omega \times \omega$ gives us the desired sequence of $\Sigma^1_1$ closed set.
\end{proof}

\begin{lemma}
Let $P(X)$ be a $\Pi^1_1$ predicate of the form
$$P(X) \leftrightarrow \forall n\ \ \exists \alpha < \omega_1^{ck}\ \ P_{\alpha, n}(X)$$
where each $P_{\alpha, n}$ is $\Delta^1_1$ uniformly in $n$ and $\alpha$. Suppose that $X$ is Solovay $\Sigma^1_1$-generic and suppose $P(X)$. Then there exists a $\Sigma^1_1$ closed set $F$ of positive measure with $X \in F$ and $\lambda(F - P)=0$.
\end{lemma}
\begin{proof}
If the complement of $\{X\ |\ P(X)\}$ is of measure 0 then take $F=2^\omega$. Otherwise from lemma \ref{approx} we have a union of $\Sigma^1_1$ closed set of positive measure included in the complement and equal to it up to a set of measure 0. As $X$ is Solovay $\Sigma^1_1$-generic and in $P$ we have a $\Sigma^1_1$ closed set of positive measure containing $X$ which is disjoint from the complement of $P$ up to a set of measure 0.
\end{proof}

We can now prove the desired theorem:

\begin{theorem}\label{hard to find name}
If $Y$ is Solovay $\Sigma^1_1$-generic then $\omega_1^Y=\omega_1^{ck}$.
\end{theorem}
\begin{proof}
Suppose that $Y$ is Solovay $\Sigma^1_1$-generic. For any Turing functional $\varphi^X$, consider the set:
$$P=\{X\ |\ \forall n\ \ \exists \alpha < \omega_1^{ck}\ \ \varphi^X(n) \in \KO^X_{\alpha}\}$$
Let $P_n=\{X\ |\ \exists \alpha < \omega_1^{ck}\ \ \varphi^X(n) \in \KO^X_{\alpha}\}$ and $P_{\alpha, n}=\{X\ |\ \varphi^X(n) \in \KO^X_{\alpha}\}$, so $P=\bigcap_n P_n$  and $P_n=\bigcup_{\alpha < \omega_1^{ck}}P_{\alpha, n}$. Note that $P_{\alpha, n}$ is $\Delta^1_1$ uniformly in $n$ and $\alpha$.\\

Suppose that $Y$ is in $P$. As $Y$ is Solovay $\Sigma^1_1$-generic, from the previous proposition, it is contained in a closed set of positive measure $F$ with $\lambda(F-P)=0$. In particular for each $n$ we have $\lambda(F - P_n)=0$ and then $\lambda(F^c \cup P_n)=1$. Then for each pair $\langle n, m \rangle$ we can search for the smallest ordinal $\alpha_{n,m}$ such that:
$$\lambda(F^c_{\alpha_{n,m}} \cup \bigcup_{\alpha < \alpha_{n,m}} P_{\alpha, n}) > 1 - 2^{-m}$$
where $F^c_\alpha$ is the open set $F^c$ enumerated up to stage $\alpha$. Let $\alpha^*=\sup_{n, m} \alpha_{n,m}$. By admissibility we have that $\alpha^*< \omega_1^{ck}$. Then we have:

$$
\begin{array}{rlcl}
&\forall n\ \ \lambda(F^c_{\alpha^*} \cup \bigcup_{\alpha < \alpha^*} P_{\alpha, n})&=&1\\
\rightarrow&\forall n\ \ \lambda(F_{\alpha^*} \cap \bigcap_{\alpha < \alpha^*} P_{\alpha, n}^c)&=&0\\
\rightarrow&\forall n\ \ \lambda(F \cap \bigcap_{\alpha < \alpha^*} P_{\alpha, n}^c)&=&0\\
\rightarrow&\forall n\ \ \lambda(F - \bigcup_{\alpha < \alpha^*} P_{\alpha, n})&=&0\\
\rightarrow&\lambda(F - \bigcap_n \bigcup_{\alpha < \alpha^*} P_{\alpha, n})&=&0
\end{array}
$$

As $X$ if Solovay-$\Sigma^1_1$ generic it is in particular weakly-Solovay-$\Sigma^1_1$ generic and then it weakly-$\Pi^1_1$-random. In particular it belongs to no $\Sigma^1_1$ set of measure 0. Then as $F - \bigcap_n \bigcup_{\alpha < \alpha^*} P_{\alpha, n}$ is a $\Sigma^1_1$ set of measure 0 we have that $X$ belongs to $\bigcap_n \bigcup_{\alpha < \alpha_{n,m}} P_{\alpha, n}$ and then $\sup_n \varphi^X(n) \leq \alpha^* < \omega_1^{ck}$.
\end{proof}

Using the equivalence between $\Pi^1_1$-random and $\Delta^1_1$-random + $\omega_1^X=\omega_1^{ck}$, we then have that the Solovay $\Sigma^1_1$-generic are included in the $\Pi^1_1$-randoms. All we have to do is prove the reverse inclusion.

\begin{theorem}\label{theorem: monin and greenberg}
The set of Solovay $\Sigma^1_1$-generic is exactly the set of $\Pi^1_1$ randoms.
\end{theorem}
\begin{proof}
Suppose $X$ is not Solovay $\Sigma^1_1$-generic. Either $\omega^{X}_1 > \omega_1^{ck}$ and then $X$ is not $\Pi^1_1$-random. Or $\omega^{X}_1 = \omega_1^{ck}$. In this case there is a sequence of $\Sigma^1_1$ closed set $\bigcup_n F_n$ of positive measure such that $X$ is not in $\bigcup_n F_n$ and such that any $\Sigma^1_1$ closed set of positive measure which is disjoint from $\bigcup_n F_n$ does not contain $X$. The complement of $\bigcup_n F_n$ is a $\Pi^1_1$ set contaning $X$ that we can write as an uncountable union of Borel sets.
As $\omega^{X}_1 = \omega_1^{ck}$ we have that $X$ is in the first $\omega_1^{ck}$ components of the uncountable union. Then $X$ is in a $\Delta^1_1$ set disjoint from $\bigcup_n F_n$. Since we can approximate this $\Delta^1_1$ from below by a union of $\Sigma^1_1$ closed set of the same measure, then $X$ is in a $\Pi^1_1$ set of measure 0 and then not $\Pi^1_1$-random.
\end{proof}

\begin{corollary}
The set of $\Pi^1_1$-randoms is $\mathbf{\Pi^0_3}$
\end{corollary}

The notion of test is interesting. For any union of closed $\Sigma^1_1$ set $S=\bigcup_n S_n$, let us define $\tilde{S}$ as the smallest intersection of $\Pi^1_1$ open set $O=\bigcap_n O_n$ contaning it. Then a test is equal to $\tilde{S} - S$. The question of whether weakly-Solovay-$\Sigma^1_1$-generic implies Solovay-$\Sigma^1_1$-generic is now the same as the open question of whether weak-$\Pi^1_1$-randomness implies $\Pi^1_1$-randomness. And this question is still open.

\section{Yu Liang: A lower bound on the Borel rank \\ of the set of $\Pi^1_1$-random reals}
Input by Yu Liang in April.

Let $\mathcal{S}$ be the set of $\Pi^1_1$-random reals. We will show that $\mathcal S$ is not $\mathbf{\Sigma}^0_2$. (This will be improved to not $\mathbf{\Sigma}^0_3$  below.)
Since $\mathcal{S}$ is a dense meager set, it is not $\mathbf{\Pi}^0_2$. As noted by Chong, Nies and Yu \cite[proof of Thm 3.12]{Chong.Nies.Yu:08}, $\+ S$ is Borel: $\+ S$ is the intersection of the $\Delta^1_1$ randoms ($\PI 3$)  with  the sets that are low for $\omega_1^{CK}$, which is properly $\mathbf{\Pi}^0_{\omega_1^{CK}+2}$ by  a result of M.
 Steel~\cite[end of Section 2]{Steel:78}.

Since $\mathcal{S}$ is a $\Sigma^1_1$-set, there must be some recursive tree $T\subseteq 2^{\omega}\times \omega^{\omega}$ such that $$x\in \mathcal{S}\leftrightarrow \exists f \forall n(x\uh n,f\uh n)\in T.$$

Now assume for a contradiction  that $\mathcal{S}$ is a $\mathbf{\Sigma}^0_2$-set. Choose   a sequence of closed sets $\{P_n\}_{n\in \omega}$ so that $\bigcup_{n\in \omega}P_n=\mathcal{S}$.

Recall that the  Gandy topology on Cantor space is given by the $\Sigma^1_1$ sets as a countable basis. (Note this is Polish on the set of sets  that are low for $\omega_1^{CK}$.)
\begin{lemma}
For each $n$, the set $\mathcal{S}\setminus P_n$ is comeager in $\mathcal{S}$ in the sense of the Gandy topology.
\end{lemma}
\begin{proof}
Fix a $\Sigma^1_1$ uncountable set $C\subseteq \mathcal{S}$. Let  $\bar{C}$ be the closure (in the Cantor space sense) of $C$. Then $\bar{C}$ is a $\Sigma^1_1$ closed set. The leftmost real in $\bar{C}$ is either hyperarithmetic or hyperarithmetically equivalent to $\+ O$. So $\bar{C}$ is not a subset of $P_n$. So there must be some $\sigma\in 2^{<\omega}$ so that $[\sigma]\cap C\neq\emptyset$ but $[\sigma]\cap C\cap P_n=\emptyset$. Let $D=[\sigma]\cap C$. So $D\subseteq C$ is an open set (in the Gandy topology sense). Thus $\mathcal{S}\setminus P_n$ contains a dense open set and so must be comeager in $\mathcal{S}$.
\end{proof}

Since Gandy topology has Baire property,  there must be some real $x\in \bigcap_{n\in \omega}\mathcal{S}\setminus P_n=\mathcal{S}\setminus \bigcup_{n\in \omega} P_n$, contradiction.
\bigskip

%

\section{Yu: Strong $\Pi^1_1$-ML-randomness is properly   $\mathbf{\Pi}^0_3$} \label{s:GB-Monin sharpness}
Input by Yu in May.

Strong $\Pi^1_1$-ML-randomness is the higher analog of weak 2-randomness. This was mentioned in a problem in  \cite[Ch.\ 9]{Nies:book}. It is open whether this notion is the same as $\Pi^1_1$-randomness.

Obviously the collection of  strongly $\Pi^1_1$-ML-random reals is $\mathbf{\Pi}^0_3$.
\begin{proposition}
The collection of strongly $\Pi^1_1$-ML-random reals is not $\mathbf{\Sigma}^0_3$.
\end{proposition}

We use a forcing argument.

Let $\mathbb{P}=(\mathbf{P},\leq)$ where $\mathbf{P}$ is the collection of $\Sigma^1_1$ closed sets with a positive measure.

Obviously, if $g$ is sufficiently generic, then it must be  strongly $\Pi^1_1$-ML-random.

\begin{lemma}\label{lemma: tech for sml}
For any $\Sigma^1_1$ tree $T$ with $\mu([T])>0$ having only $\Pi^1_1$-ML random reals, there is a uniformly $\Pi^1_1$ sequence open sets $\{U_n\}_{n\in \omega}$ so that
\begin{itemize}
\item $\forall n\mu(U_n\cap [T])<2^{-n}$; and
\item for any $\sigma$, if $[\sigma]\cap [T]\neq\emptyset$, then $[\sigma]\cap [T] \cap (\bigcap_n\{U_n\}_{n\in \omega})\neq\emptyset$.
\end{itemize}
\end{lemma}
\begin{proof} This is like difference tests. 
Given $n$, for any $\sigma$, we enumerate strings of $2^{2\cdot |\sigma|+n+1}$ into $U_n$ from left to right which are possibly the leftmost finite string of length $2\cdot|\sigma|+n+1$ among those in $[\sigma]\cap [T]$.

The sequence $\{U_n\}_{n\in \omega}$ is precisely what we want.
\end{proof}

\begin{lemma}[Nies, see Thm. 11.7 \cite{LogicBlog:12}]\label{lemma: nies}
A $\Pi^1_1$-ML random real $x$ is $\Pi^1_1$-difference random if and only if 
the $\Pi^1_1$ version of Chaitin's halting probability   $\ul \Omega \not \leq_{\mathrm{h-T}}x$.
\end{lemma}

\begin{lemma}[Greenberg, Bienvenu, Monin]\label{lemma: GLM}
No strongly $\Pi^1_1$-ML random is $\mathrm{h-T}$-above $\ul \Omega$. So by  Lemma~\ref{lemma: nies} every strongly $\Pi^1_1$-ML random real is $\Pi^1_1$-difference random.
\end{lemma}

Now given ANY sequence open sets $\{V_n\}_{n\in \omega}$, let $\mathcal{D}_{V}=\{T\mid T\in \mathbf{P}\wedge [T]\cap \bigcap_{n\in \omega}V_n=\emptyset\}$.
\begin{lemma}\label{lemma: dense dv}
If $\bigcap_{n\in \omega}V_n$ only contains strongly $\Pi^1_1$-ML-random, then $\mathcal{D}_{V}$ is dense.
\end{lemma}
\begin{proof}
Given any condition $T\in \mathbf{P}$. By Lemma \ref{lemma: tech for sml}, there is a uniformly $\Pi^1_1$ sequence open sets $\{U_n\}_{n\in \omega}$ as described. Then there must be some $\sigma$ so that $[\sigma]\cap [T]\neq \emptyset$ but $[\sigma]\cap [T]\cap (\bigcap_{n\in \omega}V_n)=\emptyset$ (Otherwise, there must be a real $x\in [T]\cap (\bigcap_{n\in \omega}V_n) \cap (\bigcap_{n\in \omega}U_n)$. Then $x$ would be a strongly $\Pi^1_1$-ML random but not $\Pi^1_1$-difference random, a contradiction to Lemma \ref{lemma: GLM}). Let $[S]=[\sigma]\cap [T]$. Then $S\leq T$ and $S\in \mathcal{D}_V$.
\end{proof}

Given any $\mathbf{\Sigma}^0_3$-set,  if $g$ is sufficiently generic, then $g$ is a strongly $\Pi^1_1$-ML-random  real not in this set.
This concludes the proof of the  proposition.

\section[Greenberg and Monin]{Yu based on Greenberg and Monin: a  lower bound on the Borel rank of the set of $\Pi^1_1$-random reals }

Input by Yu in  August, 2013. The result is essentially due to Greenberg and Monin. It is Greenberg who told me the result.

Let $\mathbb{P}=(\mathbf{P},\leq)$ where $\mathbf{P}$ is the collection of $\Sigma^1_1$ closed sets with a positive measure.

\begin{lemma}
Every $\mathbb{P}$-generic real $g$ is $\Pi^1_1$-random.
\end{lemma}
\begin{proof}
By Theorem \ref{theorem: monin and greenberg}, it is sufficient to prove that $g$ is $\Sigma^1_1$-Solovay generic.

Given a uniformly $\Sigma^1_1$-closed sets $\{F_n\}_{n\in \omega}$ with positive measure. Let $$\mathcal{D}=\{F_n\mid n\in \omega\}\cup \{F\mbox{ is closed and }\Sigma^1_1 \mid F\cap (\bigcup_{n\in \omega}F_n)=\emptyset\}.$$

Obviously $\mathcal{D}$ is dense. So the lemma follows.
\end{proof}

Now suppose that $\{V_n\}_{n}$ is a sequence open sets so that $\bigcap_{n\in \omega}V_n$ only contains $\Pi^1_1$ random reals. Let $$\mathcal{D}_{V}=\{T\mid T\in \mathbf{P}\wedge [T]\cap \bigcap_{n\in \omega}V_n=\emptyset\}.$$

By Lemma \ref{lemma: dense dv}, $\mathcal{D}_{V}$ is dense.

In conclusion, the collection of $\Pi^1_1$-random reals is not $\mathbf{\Sigma}^0_3$.

{\bf Remark:} By the proof, for any set $\mathcal A$ of reals, if $\Pi^1_1$-random $\subseteq \mathcal A\subseteq \Pi^1_1$-difference random, then $A$ is not $\mathbf{\Sigma}^0_3$.

\section{Bienvenu, Greenberg, Monin: A $\Pi^1_1$-MLR set $X$ is not $\Pi^1_1$-random iff $X$ $hT$-computes a $\Pi^1_1$ sequence which is not $\Delta^1_1$.}

The $hT$-reductions are the most general version of Turing-reductions, as defined by Bienvenu, Greenberg and Monin in LogicBlog2012. We have that if $X$ $hT$-computes a $\Pi^1_1$ sequence which is not $\Delta^1_1$, then $X$ $h$-computes $\KO$ and is thus not $\Pi^1_1$-random, as $\omega_1^X > \CK \leftrightarrow X \geq_h \KO$. So all we need to prove is the following theorem:

\begin{theorem}
If $X$ is $\Pi^1_1$-MLR but not $\Pi^1_1$-random, then $X$ $hT$-computes a $\Pi^1_1$ sequence which is not $\Delta^1_1$.
\end{theorem}
\begin{proof}
Suppose that $X$ is $\Pi^1_1$-MLR but not $\Pi^1_1$-random. Then from theorem \ref{hard to find name} there is a uniform intersection of $\Pi^1_1$ open sets $\bigcap_n O_n$ so that $X \in \bigcap_n O_n$ and so that no $\Delta^1_1$ closed set $F \subseteq \bigcap_n O_n$ of positive measure contains $X$ (and thus no $\Delta^1_1$ closed set $F \subseteq \bigcap_n O_n$ contains $X$). Let $\{W_e\}_{e \in \omega}$ be an enumeration of the $\Pi^1_1$ subsets of $\omega$. We will construct a $\Pi^1_1$ sequence $A$ which is not $\Delta^1_1$ and so that $X$ can hT-compute $A$. We use the usual way to make $A$ not $\Delta^1_1$, by meeting each requirement 
$$R_e : W_e \mbox{ infinite } \rightarrow A \cap W_e \neq \emptyset$$
making sure in the meantime that $A$ is co-infinite.\\

In what follows, to speak of ordinal stages and finite substages in a clean way, we use the ordinal version of the euclidian division: For an ordinal $\alpha$, there is a unique pair of ordinal $\langle \beta, n \rangle$ so that $\alpha=\omega \times \beta + n$. Furthermore one can uniformly find $\beta$ and $n$ from $\alpha$ (a simple research within the ordinals smaller than $\alpha$). Then, the stage $\omega \times \beta + n$ should be understood as substage $n$ of stage $\alpha$.\\

\noindent \textbf{Construction of $A$}:\\
\\
First, for each $e$ let $b_e$ be a boolean initialized to 'false'. At stage $\gamma=\omega \times \alpha + \langle e, m, k \rangle$ (At stage $\alpha$, at substage $\langle e, m, k \rangle$), if $b_e$ is marked 'true', go to the next stage (next substage), otherwise if $m \in W_e[\alpha]$ with $m>2e$, then consider the $\Delta^1_1$ set $\bigcap_n O_n[\alpha]$ and compute an increasing union of $\Delta^1_1$ closed sets $\bigcup_n F_n$ with $\bigcup_n F_n \subseteq \bigcap_n O_n[\alpha]$ and $\lambda(\bigcup_n F_n)=\lambda(\bigcap_n O_n[\alpha])$. \\

If $\lambda(F_k^c \cap O_m[\gamma]) \leq 2^{-e}$ then enumerate $m$ into $A$ at stage $\gamma$, mark $b_e$ as 'true' and let $U_{\langle m, e \rangle}=F_k^c \cap O_m[\gamma]$ (the $U_{\langle m, e \rangle}$ are intended to form a higher Solovay test).\\

\noindent \textbf{Verification that $A$ is not $\Delta^1_1$}:\\
\\
$A$ is co-infinite because for each $e$ at most one $m$ is enumerated into $A$ and this $m$ is bigger than $2e$. Now suppose that $W_e$ is infinite. There exists then $\alpha<\CK$ so that $W_e[\alpha]$ is infinite (otherwise the function which to $n$ associates the first ordinal time at which the $n$-th element enters $W_e$ would have its range cofinal in $\CK$, which is not possible). Then there exists $\beta>\alpha$ so that $\lambda(\bigcap_n O_n - \bigcap_n O_n[\beta]) < 2^{-e}$. Thus there is a $\Delta^1_1$ closed set $F_k \subseteq \bigcap_n O_n[\beta]$ so that $\lambda(\bigcap_n O_n - F_k) < 2^{-e}$. Then there exists $a$ for that for all $b \geq a$ we have $\lambda(O_b - F_k) < 2^{-e}$ and in particular $\lambda(O_b[\omega \times \beta + \omega] - F_k) < 2^{-e}$. But as $W_e[\alpha]$ is already infinite we have for some $m \in W_e[\beta]$ with $m>2e$ that $\lambda(O_m[\omega \times \beta + \langle e, m, k \rangle] - F_k) < 2^{-e}$ and then at stage $\omega \times \beta + \langle e, m, k\rangle$, $m$ is enumerated into $A$, if $R_e$ is not met yet.\\

\noindent \textbf{Verification that $\{U_{\langle m, e \rangle}\}_{m,e \in \omega}$ is a higher Solovay test}:\\
\\
We put an open set in the Solovay test only when $R_e$ is 'actively' met, and this open set has measure smaller than $2^{-e}$. As each $R_e$ is 'actively' met at most once, we have a Solovay test.\\

\noindent \textbf{Computation of $A$ from $X$}:\\
\\
Note that we now just describe the algorithm to compute $A$ from $X$. The verification that the algorithm works as expected is given in the next section. Let $p$ be the smallest integer so that for any $m>p$, $X$ is in no $U_{\langle m, e \rangle}$. To decide whether $m>p$ is in $A$, look for the smallest $\alpha$ such that $X \in O_m[\alpha]$. Then decide that $m$ is in $A$ iff $m$ is in $A[\alpha]$.\\

\noindent \textbf{Verification that $X$ computes $A$}:\\
\\
Let $p$ be the smallest integer so that for any $m>p$, $X$ is in no $U_{\langle m, e \rangle}$. Suppose for $m>p$ that $X \in O_m[\alpha]$ and $m \notin A[\alpha]$. Suppose also that at latter stage $\gamma=\omega \times \beta + \langle e, m, k \rangle >\alpha$, the integer $m$ is enumerated into $A$. By construction, it means we have $\lambda(O_m[\gamma] - F_k) < 2^{-e}$ for some $\Delta^1_1$ closed set $F_k \subseteq \bigcap_n O_n$ and that $U_{\langle m, e \rangle}=O_m[\gamma] - F_k$ (Note that $U_{\langle m, e \rangle}$ cannot be replaced latter because of a different $k$, as $R_e$ is now met).

As $X$ does not belong to $U_{\langle m, e \rangle}$ and does not belong to $F_k$, it does not belong to $O_m[\gamma]$ which contradicts the fact that it belongs to $O_m[\alpha] \subseteq O_m[\gamma]$.

\end{proof}

\section{Bienvenu, Greenberg, Monin: For any $n\geq 4$ we have $\Pi^1_1$-random $\leftrightarrow P^0_n(\Pi^1_1)$-random}

We say that a set is $P^0_2(\Pi^1_1)$ if it is equal to $\bigcap_n O_n$ where each $O_n$ is a $\Pi^1_1$ open set uniformly in $n$. The $P^0_3(\Pi^1_1)$ sets are those of the form $\bigcap_m \bigcup_n F_{n,m}$ where each $F_{n,m}$ is a $\Sigma^1_1$ closed set uniformly in $n$ and $m$. The $P^0_n(\Pi^1_1)$ sets and $S^0_n(\Pi^1_1)$ sets are then defined for any $n \in \omega$, following the same logic.\\

Let $\KO_1$ be a $\Pi^1_1$ set of unique computable ordinal notations. For $o \in \KO_1$ we denote by $|o|$ the corresponding ordinal. We will consider in this section an extension of the notion of functionals, which seems more adapted to work in the higher world. Some recent work of Bienvenu, Greenberg and Monin, still unpublished, says that we can have for some $X,Y$ that $X \geq_{hT} Y$, but with the impossibility of having a computation of $Y$ from $X$ with functional consistent everywhere. This is why we decide here to make of inconsistency something 'normal' by defining $\Pi^1_1$ ``relationals" which are intended to be to $\Pi^1_1$ relations what $\Sigma^0_1$ functionals are to $\Sigma^0_1$ functions.\\

A $\Pi^1_1$ relational $\varphi_P$ is given by a $\Pi^1_1$ predicate $P \subseteq 2^{<\omega} \times \omega \times \KO_1$. We write $\varphi_P^X(n) \downarrow$ the predicate $\exists o\ \exists \s\ \prec X\ P(\s, n, o)$. If  $p \in \omega$ we write $\varphi_P^X(n) \downarrow = p$ for   $\exists \s \prec X\ P(\s, n, p)$. Note that we can have distinct values $p_1, p_2 \in \KO_1$ so that $\varphi_P^X(n) \downarrow = p_1$ and $\varphi_P^X(n) \downarrow = p_2$. We write $\dom \varphi_P$ for  the set of $X$ such that any $n$ is in relation with at least one element of $p$ : $\{X\ |\ \forall n\ \exists p\ \varphi_P^X(n) \downarrow = p\}$.

\begin{fact} \label{fact}
Each $\Pi^1_1$ relational $\varphi_P$ corresponds to the higher $\Pi^0_2$ set $\dom \varphi_P$. Conversely each higher $\Pi^0_2$ set $\bigcap_n O_n$ corresponds to the $\Pi^1_1$ relational 
$\varphi_P^X(n)\downarrow=p \leftrightarrow p \in \KO_1 \wedge \exists \s \prec X\ \s \in O_n[|p|]$. 
\end{fact}

\begin{lemma} \label{relationals}
If $\omega_1^Z > \CK$ and $Z$ is $\Delta^1_1$ random then there is a $\Pi^1_1$ relational $\varphi_P$ such that $Z \in \dom \varphi_P$ and such that $\sup_n \{\min \{|o|\ |\ \varphi^Z_P(n) \downarrow = o\}\} = \CK$.
\end{lemma}
\begin{proof}
From theorem \ref{hard to find name} there is a higher $\Pi^0_2$ set $\bigcap_n O_n$ containing $Z$ so that $Z$ is in no $\Sigma^1_1$ closed set of positive measure included in $\bigcap_n O_n$ (and then in no $\Sigma^1_1$ closed set included in $\bigcap_n O_n$). Consider for each $\alpha$ computable the set $\bigcap_n O_n[\alpha]$. We can approximate $\bigcap_n O_n[\alpha]$ from below by a union of $\Delta^1_1$ closed sets of the same measure and as $Z$ is in no $\Delta^1_1$ nullset and in no $\Delta^1_1$ closed sets included in $\bigcap_n O_n[\alpha]$, $Z$ cannot be in $\bigcap_n O_n[\alpha]$. This implies that for $\varphi_P$ defined from $\bigcap_n O_n$ in fact \ref{fact}, we have $\sup_n \{\min \{o\ |\ \varphi_P(n) \downarrow = o\}\} = \CK$.
\end{proof}

\noindent We now have the following lemma:

\begin{lemma}\label{plouf}
From any $\Pi^1_1$ relational $\varphi_P$ one can obtain effectively in $\varepsilon$ a $\Pi^1_1$ relational $\varphi_Q$ so that:
\renewcommand{\arraystretch}{1.5}
$$
\begin{array}{l}
1:\ \ \dom{\varphi_P}=\dom{\varphi_Q}\\
2:\ \ \forall X\ \forall n\ (\exists! o\ \varphi_P^X(n)=o) \rightarrow \varphi_Q^X(n)=o\\
3:\ \ \forall X\ \forall n\ \min \{|o|\ |\ \varphi_P^X(n)=o\} \leq \min \{|o|\ |\ \varphi_Q^X(n)=o\}\\
4:\ \ \lambda(\{X\ |\ \exists n\ \exists o_1 \neq o_2\ \varphi_Q^X(n)\downarrow=o_1 \wedge \varphi_Q^X(n)\downarrow=o_2\}) \leq \epsilon
\end{array}
$$
\renewcommand{\arraystretch}{1}
\end{lemma}
\begin{proof}
{.}\\
\\
\textbf{The construction}:\\
\\
In what follows, to speak of ordinal stages and finite substages in a clean way, we use the ordinal version of the euclidian division: For an ordinal $\alpha$, there is a unique pair of ordinal $\langle \beta, n \rangle$ so that $\alpha=\omega \times \beta + n$. Furthermore one can uniformly find $\beta$ and $n$ from $\alpha$ (a simple research within the ordinals smaller than $\alpha$). Then, the stage $\omega \times \beta + n$ should be understood as substage $n$ of stage $\alpha$.\\

Take a computable sequence of rationals $\varepsilon_n$ so that $\sum_n \varepsilon_n \leq \varepsilon$. For each $n$ and uniformly in $n$ we do the following:\\

At stage $\gamma=\omega \times \alpha + \langle \s, o \rangle$ let $A_{\gamma}=\bigcup \{[\tau]\ |\ \exists \beta < \gamma\ \varphi_Q^\tau(n)[\beta]\downarrow\}$. If $\varphi_P^\s(n)[\alpha]\downarrow=o$, we effectively find a clopen set $B_{\gamma}=\bigcup_{i<m} [\tau_i] \subseteq [\s]$ so that $B_{\gamma} \cup A_{\gamma}$ covers $[\s]$ and such that $\lambda (B_{\gamma} \cap A_{\gamma}) < \varepsilon_n 2^{-p(\gamma)}$, where $p$ is an injection from $\CK$ to $\omega$. We then set $\varphi_Q^{\tau_i}(n)[\gamma]=o$ for any of the $\tau_i$ such that $B_{\gamma}=\bigcup_{i<m} [\tau_i]$.\\
\\
\textbf{Verifcation}:
\\
\begin{enumerate}
\item Let us prove $\dom{\varphi_Q} \subseteq \dom{\varphi_P}$. Suppose that $\varphi_Q^X(n)\downarrow$. Then by construction we have a stage $\gamma$ with $X$ in a clopen set $B_{\gamma} \subseteq [\s]$ with $\varphi_P^\s(n)\downarrow$. Then we also have $\varphi_P^X(n)\downarrow$ which gives us $\dom{\varphi_Q} \subseteq \dom{\varphi_P}$. For the other inclusion, suppose that $\varphi_P^\s(n)\downarrow$ with $\s \prec X$. Then by construction we have a stage $\gamma$ and an open set $B_\gamma \cup A_\gamma$ covering $[\s]$ with $\varphi_Q^Y(n)[\gamma]\downarrow$ for any $Y$ in $B_\gamma \cup A_\gamma$. Then we have that $\varphi_Q^X(n)\downarrow$ and then $\dom{\varphi_P} \subseteq \dom{\varphi_Q}$.

\item Suppose that $\exists! o\ \varphi_P^X(n)=o$. Consider the smallest $\gamma = \omega \times \alpha + \langle \s, o \rangle$ so that $\s \prec X$ and $\varphi_P^\s(n)[\alpha]\downarrow=o$. By hypothesis $\{\tau\ |\ \exists \beta < \gamma\ \varphi_Q^\tau(n)[\beta]\downarrow\}=\emptyset$ and then $\varphi_Q^\s(n)[\gamma]\downarrow=o$.

\item This is true because the images of $n$ via $\varphi_Q^X$ are a subset of the images of $n$ via $\varphi_P^X$.

\item For a given $n$ we have that $\{X\ |\ \exists o_1 \neq o_2\ \varphi_Q^X(n)\downarrow=o_1 \wedge \varphi_Q^X(n)\downarrow=o_2\}$ is included in $\bigcup_\gamma B_\gamma \cap A_\gamma$. Also we have $\lambda(\bigcup_\gamma B_\gamma \cap A_\gamma) \leq \sum_\gamma \varepsilon_n 2^{-p(\gamma)} \leq \varepsilon_n$.
\end{enumerate}

\end{proof}

\begin{lemma}\label{good}
If $\omega_1^Z > \CK$ and $Z$ is $\Pi^1_1$-ML-random then there is a $\Pi^1_1$ relational $\varphi_P$ such that $Z \in \dom \varphi_P$, such that $\forall n\ \exists!o\ \varphi_P^Z(n)=o$ and such that $\sup_n |\varphi_P(n)| = \CK$.
\end{lemma}
\begin{proof}
Suppose that $\omega_1^Z > \CK$ and $Z$ is $\Delta^1_1$ random, from lemma \ref{relationals} we have $\varphi_P$ such that $Z \in \dom \varphi_P$ and such that $\sup_n \{\min \{o\ |\ \varphi_P(n) \downarrow = o\}\} = \CK$. But from lemma \ref{plouf} one can obtain uniformly in $\varepsilon$ a functional $\varphi_Q$ with $\dom \varphi_Q=\dom \varphi_P$ and so that the $\Pi^1_1$ open set:
$$\{X\ |\ \exists n\ \exists o_1 \neq o_2\ \varphi_Q^X(n)\downarrow=o_1 \wedge \varphi_Q^X(n)\downarrow=o_2\}$$ 
has measure smaller than $\varepsilon$. Since $Z$ is $\Pi^1_1$-ML random there exists a relational $\varphi_Q$ with $Z$ in its domain, which is functionnal on $Z$ and (using 3 of lemma \ref{plouf}) such that $\sup_n |\varphi_Q(n)| = \CK$.
\end{proof}

We now assume that we have $Z$ $\Pi^1_1$-ML-random with $\omega_1^Z > \CK$ and that we have a relational $\varphi_P$ with the properties of lemma \ref{good}. In order to put $Z$ in a $\P^0_2(\Pi^1_1)$ nullset, we would like $\dom \varphi_P$ to have measure 0. In order to do so we would like $\dom \varphi_P$ to contain no $X$ with $\omega_1^X = \CK$. This is what we are trying to achieve now. Note that we eventually won't be able to put $Z$ in a $\P^0_2(\Pi^1_1)$ nullset, but only in a $\P^0_4(\Pi^1_1)$ nullset.\\

For any $e \in \omega$ we define $R_e \subseteq \omega \times \omega$ by $R_e(n,m) \leftrightarrow \langle n, m \rangle \in W_e$. We define then $R_e \rest k$ to be $R_e \rest k (n,m) \leftrightarrow \langle m, k \rangle \in W_e \wedge R_e(n, m)$. Note that $R_e \rest k$ is well defined for any $e$. Also in what follows, a morphisms from a relation $R_a$ to another relation $R_b$ is a function $f$ total on $\dom{R_a}$, with $f(\dom{R_a}) \subseteq \dom{R_b}$ and $R_a(x, y) \rightarrow R_b(f(x), f(y))$. Let us consider the two following predicates on $2^\omega \times \omega$:

\renewcommand{\arraystretch}{1.5}
$$
\begin{array}{rcl}
C_1(X,e)&\leftrightarrow&\exists n\ \exists o_n\ \varphi^X_P(n)\downarrow = o_n \wedge \forall f\ f \mbox{ is not a morphism from } R_{o_n} \mbox{ to } R_e\\
C_2(X,e)&\leftrightarrow&\exists m\ \forall n\ \exists o_n\ \varphi^X_P(n)\downarrow = o_n \wedge \forall f\ f \mbox{ is not a morphism from } R_e \rest m \mbox{ to } R_{o_n}
\end{array}
$$
\renewcommand{\arraystretch}{1}

We will now join them into one predicate. Let us define $G$ to be the $\Pi^0_2$ set of $e$ so that $R_e$ is a linear order of $\omega$. We then define:
$$C(X)\leftrightarrow X \in \dom\varphi_P \wedge (\forall e \in G\ \ C_1(X,e) \vee C_2(X,e))$$

Let us first make sure that $\{X\ |\ C(X)\}$ is $P^0_4(\Pi^1_1)$. We have that $\dom \varphi_P$ is $P^0_2(\Pi^1_1)$, $\{X\ |\ C_1(X,e)\}$ is $S^0_1(\Pi^1_1)$ uniformly in $e$ and the set $\{X\ |\ C_2(X,e)\}$ is $\Sigma^0_3(\Pi^1_1)$ uniformly in $e$. Then the set $\dom \varphi_P \cap (\{X\ |\ C_1(X,e)\} \cup \{X\ |\ C_2(X,e)\})$ is $S^0_3(\Pi^1_1)$ uniformly in $e$. As $G$ has a $\Pi^0_2$ description, we have that $\{X\ |\ C(X)\}$ is a $P^0_4(\Pi^1_1)$ set.

The goal is now to prove that $Z \in C$ and that for all $X \in \dom \varphi$ if $\varphi_P$ is functionnal on $X$ and $\omega_1^X=\CK$ then $X \notin C$.

Let us first prove that $Z \in C$. By hypothesis we have $Z \in \dom\varphi_P$. Take any $e \in G$. Suppose that $R_e$ is a well-founded relation. As $e$ is already in $G$ we have that $R_e$ is a well-ordered relation. Then $|R_e| < \CK$. But then there is some $n$ so that $|\varphi^Z_P(n)| > |R_e|$ and we cannot have a morphism from $R_{\varphi^Z_P(n)}$ to $R_e$. Then $C_1(Z, e)$ is true. Suppose now that $R_e$ is a ill-founded relation. There is then some $m$ so that $R_e \rest m$ is already ill-founded. But as $R_{\varphi^Z_P(n)}$ is well-founded for every $n$, then for every $n$ we cannot have a morphism from $R_e \rest m$ to $R_{\varphi^Z_P(n)}$. Then $C_2(Z, e)$ is true.

Let us first prove that $\forall X \in \dom \varphi$ if $\varphi_P$ is functional on $X$ and $\omega_1^X=\CK$ then $X \notin C$. Take any $Y \in \dom \varphi_P$ so that $\varphi_P$ is functional on $Y$ and $\omega_1^Y=\CK$. Then $\sup_n |\varphi^Y_R(n)| = \alpha < \CK$. Thus there exists some code $e \in G$ so that $R_e$ is a well-order of order-type $\alpha$. For this $e$ we certainly have for all $n$ a morphism from $R_{\varphi^Y_R(n)}$ into $R_e$. Then we do not have $C_1(Y, e)$. Let us now prove that we do not have $C_2(Y, e)$. For any $m$ we have $|R_e \rest m| < \alpha$ (even if $\alpha$ is successor). But because $\alpha=\sup_n \varphi^Y_R(n)$ there is necessarily some $n$ so that $|\varphi^Y_R(n)|>|R_e \rest m|$. Thus there is a morphism from $R_e \rest m$ into $R_{\varphi^Y_R(n)}$. Then we do not have $C_2(Y, e)$ and then we do not have $C(Y)$.

\noindent Thus the measure of $\{X\ |\ C(X)\}$ is bounded by the measure of 
$$H=\{X \in \dom \varphi_P\ |\ \varphi_P \mbox{ is not functional on } X\}$$
But we can obtain uniformly in $\varepsilon$ some predicate $C_\varepsilon(X)$ with $C_\varepsilon(Z)$ and with the measure of $H$ bounded by $\varepsilon$. As each $C_\varepsilon$ is $P^0_4(\Pi^1_1)$ uniformly in $\varepsilon$ we have that $\bigcap_\epsilon C_\varepsilon$ is a $P^0_4(\Pi^1_1)$ set of measure 0 and containing $Z$. Thus $P^0_4(\Pi^1_1)$ nullsets are enough to capture anything that a $\Pi^1_1$ nullset can capture. In particular the $P^0_n(\Pi^1_1)$ sets do not capture anything more for $n > 4$.


\newpage

\part{Reverse Mathematics}

\section{Nies: The strength of Jordan decomposition \\ for functions of bounded variation}

\label{s:GMNS reverse}
Written by Nies based on work with N.\ Greenberg, J.S.\ Miller,  and T.\ Slaman.

All real valued functions  will have domain    $[0,1]$ unless otherwise mentioned. Variables $f,g,h$ denote functions. Recall that  for a function $f$, one defines the variation function $V_f$ by
 $$V_f(x)= \sup \sum_{i=1}^{n-1}   | f(t_{i+1}) - f(t_i)| < \infty,$$ 
where the sup is taken over all collections $ t_1 \le t_2 \le  \ldots \le  t_n$ in $[0,x]$.  One says that $f$ is of bounded variation if $V_f(1)< \infty$.

 The Jordan decomposition theorem states that every function $f$ of bounded variation can be written in the form  $f = g-h$ where $g,h$ are nondecreasing. Moreover, if  $f$ is continuous, we can ensure that $g,h$ are continuous as well. This is easily proved: let $g = V_f$, and observe that $h=g-f$ is nondecreasing. So we have a  Jordan decomposition $f = g- (g-f)$.
Jordan proved this theorem in his  lectures at the Ecole Polytechnique 1882-7. Today it  is often  treated in a first course on real analysis. Its strength in the sense of reverse mathematics is not obvious. We will see that, depending on whether we want $g,h$ to be continuous or not,  gives rise to  principle   equivalent to $\mathtt{ACA}_0$,  or to $\mathtt{WKL_0}$. 

Let us say that   \bc  $f \le_{\mathtt{slope}} g$ $:\LR$ $\fa x< y  \, [ f(y) -f(x) \le g(y)- g(x)]$.  \ec 
That is, the slopes of $g$ are  at least as large as the slopes of $f$. Clearly, this is equivalent to saying that $h:= g-f$ is nondecreasing. Thus, the problem of finding a Jordan decomposition of $f$ is equivalent to finding a nondecreasing function $g$ with $f \le_{\mathtt{slope}}g $. This was already pointed out in \cite{Rettinger.Zheng:04}. (Sometimes one of the functions is partial; then we only look at slopes in the domain.)

\subsection{Jordan decomposition  via continuous functions}

We first work in the usual setting of reverse mathematics,  which only deals with continuous functions, suitably encoded, for instance, by the values at rationals, together with a modulus of uniform continuity (see e.g.\ Simpson's book II.6).
It is equivalent (over $\mathtt{RCA_0}$) to describe $f$ as the limit of a Cauchy name $\seq{p_s}$ with respect to the sup norm,  where the $p_s$ are   polygonal functions   with rational breakpoints. Thus $||p_t-  p_s||_{sup} \le \tp{-s}$ for $t> s$.

The principle $\mathtt{Jordan_{cont}}$ says that for each (continuous) function $f$ of bounded variation, there are continuous nondecreasing functions $g,h$ such that $f= g-h$.   

In Proposition~\ref{prop:ACA} and Theorem~\ref{thm:WKL} below, technique  and some  writing involving     reverse mathematics  was provided by Keita Yokoyama. 

\begin{proposition}\label{prop:ACA} Over $\mathtt{RCA_0}$, we have \bc $\mathtt{Jordan_{cont}} \lra \mathtt{ACA_0}$. \ec
\end{proposition}  

\begin{proof}  $\leftarrow$: given $f$, from the jump of a representation of $f$ as a continuous function we can compute a representation of $V_f$ as a continuous function. This works in $\mathtt{ACA_0}$.
	
	\n $\rightarrow$: Suppose we are given  a model of  $\mathtt{Jordan_{cont}}$. Let 
	
	\bc $q_n = 1-\tp{-n}$, and $q_{n,s} = q_n - \tp{-n-s-1}$. \ec  
	
  Instead of proving $\ACA$, we will show the existence of the range of any one-to-one functions on $\NN$ within $\RCA$. (See \cite[Lemma~III.1.3]{Simpson:09}). 
 
 Let $h:\NN\to\NN$ be a one-to-one function.
 Define continuous functions $f_{s}:[0,1]\to\RR$ as follows:
define $f_{s}$ on $[q_{n,k},q_{n,k+1}]$ to be a sawtooth function of height $2^{-k}$ with $2^{k-n}$ many teeth if $k\le s$ and $h(k)=n$, and $f_{s}=0$ otherwise.
 Then, the limit $f=\lim_{s\to\infty}f_{s}$ exists.

Let $f\le_{slope}g$. Take a function $\eta:\NN\to\NN$ so that $g(q_{n})-g(q_{n,\eta(n)})<2^{-n}$.
This can be done by the following argument: since $g$ is continuous and $\lim_{k\to\infty}q_{n,s}=q_{n}$, we have \bc  $\forall n \exists k \theta(n,k)\equiv\exists m\theta_{0}(n,m,k)$ \ec where $\theta(n,k)$ is a $\Sigma^{0}_{1}$-formula which expresses $g(q_{n})-g(q_{n,k})<2^{-n}$, and $\theta_{0}(n,m,k)$ is a $\Sigma^{0}_{0}$-formula.
Define $\eta_{0}$ as $\eta_{0}(n)$ to be the least $\langle m,k\rangle$ where $\langle\cdot,\cdot\rangle$ is a standard pairing function, and $\eta(n)=(\eta_{0}(n))_{1}$.

Now, if $h(k)=n$, then $g(q_{n})-g(q_{n,k})\ge g(q_{n,k+1})-g(q_{n,k})\ge 2^{-n}$, and hence $k<\eta(n)$.
Thus, $n\in \mathrm{rng}(h)\Leftrightarrow\exists k<\eta(n)\  h(k)=n$, which means that the range of $h$ exists by $\Delta^{0}_{1}$-comprehension.
\end{proof}
%
	

\begin{cor} Over $\mathtt{RCA_0}$, the statement that every (continuous) function $f$  of bounded variation has a variation function
is equivalent to $\mathtt{ACA_0}$. \end{cor}

\begin{proof} Given $X$, let $f$ be the function constructed above. If $V_f$ exists then $V_f- (V_f-f)$ is a continuous Jordan decomposition. So $V_f$ computes $X'$. \end{proof}

\subsection{Nies, Yokoyama: Jordan decomposition  without  continuity}

A weaker principle is obtained if we admit non continuity of $g,h$ in a Jordan decomposition $f = g-h$.  We only require that $g,h$ are defined in $I_\QQ:=\QQ \cap [0,1]$. An $\RR$ -valued  function $g$ defined on $I_\QQ$ is given by a path $Z_f$  through a  binary tree. Let $\la p_n, q_n \ra$ be a list of all pairs of rationals $\la p, q \ra$ with $0 \le p \le 1$.  We let $Z_f(2n)= 1$ iff $g(p_n) < q_n$. We let $Z_f(2n+1)= 1$ iff $g(p_n) > q_n$.  We often identify $f$ and $Z_f$. It is clear that the nondecreasing functions form a $\PPI$ class.

 The principle $\mathtt{Jordan_{\QQ}}$ says that for each (necessarily continuous by the encoding) function $f$ of bounded variation, there are   nondecreasing functions $g,h$ defined on $I_\QQ$  such that $f= g-h$ on $I_\QQ$. Letting $\widehat g(x) = \sup\{ g(q) \mid \, q \le x \lland q \in I_\QQ\}$, we obtain an actual Jordan decomposition because $f \le_\mathtt{slope} \hat g$.   
 
 We first prove a purely computability theoretic result. Yokoyama has given the extension to reverse mathematics- see below.
 
 \begin{theorem}[Greenberg, Miller, Nies, Slaman, 2013] \label{thm:PA} An oracle  $B$ is PA-complete $\lra$ for  each computable function $f$ on $[0,1]$ of bounded variation,   $B$ computes a function $g\colon \, I_\QQ \to \RR$ with $f \le_\mathtt{slope} g$. 
\end{theorem}  
\begin{proof} $\leftarrow$:   Given $f$, via the encoding above, the functions $g$ defined on $I_\QQ$ with  $f \le_\mathtt{slope} g$    form a  nonempty $\PPI(f)$ class.  Then $B$ computes a member of  this class.

\n $\rightarrow$:
We define a computable function $f$ of bounded variation on $[0,1]$ such that each function $g\colon \, I_\QQ \to \RR$ with $f \le_\mathtt{slope} g$ has PA degree. 

Let $\+ P \sub \cantor$ be a nonempty $\PPI$ class of  sets of PA degree, such as the (binary encoded) completions of Peano arithmetic. As usual $\+ P_s$ is a clopen set computable from $s$ approximating  $\+ P$ at stage $s$. So $\+ P = \bigcap_s \+ P_s$.   By standard methods there is a computable prefix-free sequence $\seq {\sss_s} \sN s$ of strings of length $s$ such that $[\sss_s] \cap \+ P_s \neq \ES$ for each $n$.

Given $\sss \in \strcantor$, let $I_\sss = [0.\sss, 0.\sss + \tp{-\sssl}]$ be the corresponding closed subinterval of  $[0,1]$. 
By stage $s$ we determine $f$ up to a precision of $\tp{-s}$. Suppose  $n$ enters  $\ES'$ at stage $s$. Let $\sss = \sss_s$.  
We define $f$ on $I_\sss$ to be a sawtooth function of height $\tp{-s}$ with $\tp{s-n}$ many teeth.   It is clear that this adds at most $\tp {-n+1}$ to the variation of $f$. So $f$ is of bounded variation. (It is in fact AC since it can be written as an integral.)

Now suppose $g\colon \, I_\QQ \to \RR$ is a function such that  $f \le_\mathtt{slope} g$.  As before for $x \in [0,1]$ let $\widehat g(x) = \sup\{ g(q) \mid \, q \le x \lland q \in I_\QQ\}$.

\n {\it Case 1.}  $\widehat g$ is discontinuous at the real $y= 0.Y$ for some $Y \in \+ P$. Then $Y \leT g$, so $g$ is of PA degree. (To see this, fix rational $r$ with $\widehat g( y) < r < g^+(y)$. Then $p< y  \lra g(p) < r$, and $p > y \lra g(p ) > r$. )

\vsp

\n {\it Case 2.} Otherwise. Then $\ES' \leT g$: given $n$, using $g$ compute stage $s$ such that for each $\sss$ of length $s$  with $[\sss] \cap \+ P_s \neq \ES$, we have $g(\max I_\sss)  - g(\min I_\sss ) < \tp{-n}$. This $s$ exists by case assumption using compactness of Cantor space. (This part of the argument cannot be adapted to reverse mathematics.) Then as before we have    $n \in \ES' \lra n \in \ES'_s$.
\end{proof}

%
%
%

%
%
%
%
%
%
%
%
%

 Yokoyama, starting from the proof of Theorem~\ref{thm:PA} above, has provided a proof that works in reverse mathematics.
\begin{theorem}   \label{thm:WKL} Over $\mathtt{RCA_0}$, we have \bc $\mathtt{Jordan_\QQ} \lra \mathtt{WKL_0}$. \ec
\end{theorem}

\begin{proof}[Proof of Theorem~\ref{thm:WKL}.]
 $\leftarrow$: the  original proof can be carried out within $\WKL$.
 
 $\rightarrow$: we reason within $\RCA$.
 Let $T\subseteq2^{<\NN}$ be an infinite binary tree. We will show that $T$ has a path.
Let $\tilde{T}=\{\tau\notin T\mid \tau\upharpoonright(|\tau|-1)\in T\}$. Without loss of generality, we may assume that $\tilde{T}$ is infinite.
Define $h:\NN\to2^{<\NN}$ as $h(n)$ to be (one of) the shortest leaf (dead end) of $T\setminus\{h(k)\mid k<n\}$.
(Note that $h$ can be considered as $h:\NN\to\NN$ by the usual coding.)
Then, we can easily see that $T\setminus \mathrm{rng}(h)=\mathrm{Ext}(T)=\{\sigma\in T\mid \sigma$ has infinitely many extensions in $T\}$.
Let $\langle \tilde{\sigma}_{k}\mid k\in\NN\rangle$ be an enumeration of $\tilde{T}$ such that $|\tilde{\sigma}_{i}|\le |\tilde{\sigma}_{i+1}|$. By an easy calculation, $|\tilde{\sigma}_{k}|\le l$ implies $k\le 2^{l}$.
For given $\sigma\in 2^{<\NN}$, define $I_{\sigma}=[l_{\sigma},r_{\sigma}]=[0.\sigma,0.\sigma+2^{|\sigma|}]$.
Now, define continuous functions $f_{s}:[0,1]\to\RR$ as follows:
define $f_{s}$ on $I_{\tilde{\sigma}_{k}}$ to be a sawtooth function of height $2^{-k}$ with $2^{k-n}$ many teeth if $k\le s$ and $h(k)=n$, and $f_{s}=0$ otherwise.
Then, the limit $f=\lim_{s\to\infty}f_{s}$ exists.

Now suppose $g:I_{\QQ}\to\RR$ is a function such that $f\le_{slope}g$. Define $\Delta:\NN\to\RR$ as
\[\Delta(k)=\max\{g(r_{\sigma})-g(l_{\sigma})\mid \sigma\in T\wedge |\sigma|=k\}.\]
Note that $\Delta(k)<2^{-n}$ can be expressed by a $\Sigma^{0}_{1}$-formula.

\noindent \textit{Case 1}: $\lim_{n\to\infty}\Delta(n)=0$.
In this case, using the same argument as the above proof, take $\eta:\NN\to\NN$ so that $\Delta(\eta(n))<2^{-n}$.
If $h(k)=n$, then, $g(r_{\tilde{\sigma}_{k}})-g(l_{\tilde{\sigma}_{k}})\ge 2^{-n}$, hence, $|\tilde{\sigma}_{k}|\le \eta(n)$.
Thus, $n\in \mathrm{rng}(h)\Leftrightarrow\exists k\le 2^{\eta(n)}\  h(k)=n$, thus, $T\setminus\mathrm{rng}(h)=\mathrm{Ext}(T)$ exists. Hence, we can easily find a path of $T$.

\noindent \textit{Case 2}: $\lim_{n\to\infty}\Delta(n)>0$.
In this case, take $q\in\QQ$ such that $\lim_{n\to\infty}\Delta(n)>q>0$, and define $\hat{T}$ as $\hat{T}=\{\sigma\in T\mid g(r_{\sigma})-g(l_{\sigma})>q\}$.
Then, $\hat{T}$ is an infinite subtree of $T$, and there exists $K\in\NN$ such that for any prefix-free $P\subseteq \hat{T}$, $|P|\le K$. For this, take $K$ so that $Kq>g(1)-g(0)$. Then, for any prefix-free $P\subseteq \hat{T}$,
\[|P|q<\sum_{\sigma\in P}g(r_{\sigma})-g(l_{\sigma})\le g(1)-g(0)<Kq.\]
Thus, we have $|P|\le K$.
Now, we can find a path of $\hat{T}$ by the following claim.

\noindent\textit{Claim}  ($\RCA$).
If $T$ is an infinite binary tree, and there exists $K\in\NN$ such that for any prefix-free $P\subseteq T$, $|P|\le K$, then $T$ has a path.

By $\Sigma^{0}_{1}$-induction, take \bc $k=\max\{i\le K\mid \exists P\subseteq T$ such that $P$ is prefix-free and $|P|=i\}$, \ec and let $P_{k}\subseteq T$ be its witness.
Then, any long enough $\sigma\in T$ is an extension of a member of $P_{k}$, and it has at most one successor. Thus, by $\Sigma^{0}_{1}$-induction, there exists $\tau\in P_{k}$ such that $\tau\in \mathrm{Ext}(T)$. Since each extension of $\tau$ has exactly one successor, we can easily find a path of $T$ extending $\tau$.
\end{proof}

 A    stronger variant     $\mathtt{strong \ Jordan_{\QQ}}$ would   require that   $f$ is only defined on $I_\QQ$, and of bounded variation for partitions consisting of rationals. By the proof above, this principle  is also equivalent to $\mathtt{WKL_0}$ over $\mathtt{RCA_0}$.

%

%
%
%
%
%
%
%
%
%

\section{Yokoyama : Notes on BVDiff and reverse mathematics}

The following was contributed by {Keita Yokoyama%
\footnote{
School of Information Science,
Japan Advanced Institute of Science and Technology,
1-1 Asahidai, Nomi, Ishikawa, 923-1292, JAPAN,
E-mail: y-keita@jaist.ac.jp}
}
following a talk of Nies at JAIST (Kanazawa, Japan) suggesting the principles studied below.   The principle $\mathrm{BVDiff}$  was introduced by Greenberg, Nies and Slaman in Auckland, Nov 2012.

A  function $f: \sub [0,1]\to \mathbb{R}$  with domain containing $\QQ \cap[0,1]$  is said to be pseudo-differentiable at $z\in (0,1)$ if   $f(z):=\lim_{x\in \mathbb{Q}, x\to z} f(x)$ exists  and for every $\varepsilon>0$, there exists $\delta>0$ such that for every $0<|h|,|h'|<\delta$,
\[\left| \frac{f(z+h)-f(z)}{h}- \frac{f(z+h')-f(z)}{h'}\right|<\varepsilon.\]
Note that this  is equivalent (over $\mathsf{RCA_0}$) to the definition of pseudo-differentiability  in \cite[Section 7]{Brattka.Miller.ea:nd}.
In the reverse mathematics setting, note that we don't require that the derivative exists as a real  of the model.

\begin{theorem}
 The following are equivalent over $\mathsf{RCA_0}$.
 
\begin{enumerate}
 \item $\mathsf{WWKL_0}$.
 \item {$\mathrm{BVDiff}$:} every continuous bounded variation function $f:[0,1]\to \mathbb{R}$ has a pseudo-differentiable point.
 \item {$\mathrm{aeBVDiff}$:} every continuous bounded variation function $f:[0,1]\to \mathbb{R}$ is pseudo-differentiable almost surely in the following sense:
\begin{itemize}
 \item[] for any family of open intervals $\mathcal{U}=\{(u_{i},v_{i})\}_{i\in\mathbb{N}}$, if $\mathcal{U}$ covers any pseudo-differential points of $f$, then $\sum_{i\in \mathbb{N}}(v_{i}-u_{i})\ge 1$. 
\end{itemize}
\end{enumerate}
\end{theorem}
This theorem follows from the Greenberg/Miller/Nies/Slaman results in Section~\ref{s:GMNS reverse},  Brattka/Miller/Nies \cite{Brattka.Miller.ea:nd},  and a new fact:

\begin{proposition}\label{propBVDiff}
$\mathrm{BVDiff}$ (or $\mathrm{aeBVDiff}$) is provable within $\mathsf{WWKL_0}$.
\end{proposition}

\begin{lemma}[Brattka/Miller/Nies \cite{Brattka.Miller.ea:nd}, $\mathsf{RCA_0}$]\label{lem1}
Let $f:\mathbb{Q}\cap[0,1]\to \mathbb{R}$ be a non-decreasing function, and let $z\in [0,1]$ be Martin-L\"of (computably) random relative to $f$.
Then, $f$ is pseudo-differentiable at $z$. \end{lemma}
Recall Theorem~\ref{thm:WKL} above:  The following are equivalent over $\mathsf{RCA_0}$.
\begin{enumerate}
 \item $\mathsf{WKL_0}$.
 \item For every  bounded variation function $f:\mathbb{Q}\cap[0,1]\to \mathbb{R}$, there exist non-decreasing functions $g,h:\mathbb{Q}\cap[0,1]\to\mathbb{R}$ such that $f=g-h$.
\end{enumerate}

Another crucial ingredient is the following, due to  Stephen G.~Simpson and Keita Yokoyama.
\begin{lemma}[\cite{Simpson.Yokoyama:11}, Lemma 3.6]   \label{lem3} For any countable model $(M,S)\models\mathsf{WWKL_0}$, there exists $\bar{S}\supseteq S$ such that $(M,\bar{S})\models\mathsf{WKL_0}$ and the following holds:
\begin{itemize}
 \item[$(\dag)$] for any $A\in \bar{S}$ there exists $B\in S$ such that $B$ is Martin-L\"of random relative to $A$.
\end{itemize}
\end{lemma}

\begin{proof}[Proof of Proposition~\ref{propBVDiff}.]
 We will show that BVDiff holds in any countable model of $\mathsf{WWKL_0}$.
 Let $(M,S)$ be a countable model of $\mathsf{WWKL_0}$, and let  $f:[0,1]\to \mathbb{R}$ be a continuous bounded variation function in $(M,S)$.
 By Lemma~\ref{lem3}, take an extension $\bar{S}\supseteq S$ such that $(M,\bar{S})\models\mathsf{WKL_0}$ and satisfies $(\dag)$.
 Then, by Theorem~\ref{thm:WKL}, there exist non-decreasing functions $g,h\in \bar{S}$, $g,h:\mathbb{Q}\cap[0,1]\to\mathbb{R}$ such that $f=g-h$.
 By $(\dag)$, take $z\in S$, $z\in [0,1]$ which is Martim-L\"of random relative to $g\oplus h$ (in $(M,\bar{S})$).
 Then, by Lemma~\ref{lem1}, both of $g$ and $h$ are pseudo-differentiable at $z$ (in $(M,\bar{S})$), thus, $f$ is pseudo-differentiable at $z$ (in $(M,{S})$).
 
 To show aeBVDiff, let $\mathcal{U}=\{(u_{i},v_{i})\}_{i\in\mathbb{N}}$ be a family of open intervals such that $\sum_{i\in \mathbb{N}}(v_{i}-u_{i})< 1$. Then, $[0,1]\setminus \bigcup_{i\in\mathbb{N}}(u_{i},v_{i})$ is a closed set which has a positive measure. Thus, in the above proof, we can find a Martin-L\"of random point $z$ in $[0,1]\setminus \bigcup_{i\in\mathbb{N}}(u_{i},v_{i})$.
\end{proof}

\begin{remark} {\rm 
Within $\mathsf{WWKL_0}$ one cannot assure the existence of the value of the derivative, in other words, the pseudo-differentiability in BVDiff cannot be replaced with the (usual) differentiability.
In fact, Jason Rute showed that the following are equivalent over $\mathsf{RCA_0}$.}
\begin{enumerate}
 \item Every continuous bounded variation function $f:[0,1]\to \mathbb{R}$ has a derivative somewhere, \textit{i.e.}, there exists $z\in (0,1)$ such that $f'(z)=\lim_{x\to z}{(f(x)-f(z))}/{(x-z)}$ exists.
 \item $\mathsf{ACA_0}$.
\end{enumerate}

\end{remark}

%

%

\section{Randomness notions as principles of reverse mathematics}
Written by Nies based on work with Greenberg and Slaman in Cambridge, June  2012 and Auckland, Dec 2012.

Let $\mathsf{C}$ denote a randomness notion. For instance $\mathsf{MLR}$ is ML-randomness, $\mathsf{CR}$ is computable randomness and $\mathsf{SR}$ is Schnorr randomness. We study the strength of the system 

\bc $\mathsf{C_0}  = \mathsf{RCA_0} + \fa X \ex Y \, [Y \in \mathsf{C}^X]$. \ec

Note that $\mathsf{MLR_0}$ is equivalent to   weak weak K\"onigs lemma at least for $\omega$-models. 
\begin{proposition} \label{prop: ML CR}  $\mathsf{CR_0}$ does not imply $\mathsf{MLR_0}$, as shown by a suitable $\omega$-model. \end{proposition}
This suggests   to call the   principle $\mathsf{CR_0}$ weak weak weak K\"onigs lemma. Recall that every high set is Turing above a computably random set by \cite{Nies.Stephan.ea:05} (also see \cite[Ch.\ 7]{Nies:book}).
\begin{proof} By the proof of \cite[Lemma 4.11]{Cholak.Greenberg.ea:06}, for each set $B$ of non-d.n.c.\ degree there is a set $X$ high (even LR-hard) relative to $B$ such that $B \oplus X$ is also not of d.n.c.\ degree. Iterating this in the standard way, we build an $\omega$-model of  $\mathsf{CR_0}$ without a set of d.n.c.\ degree. In particular, there is no ML-random set.  \end{proof} 

Recall from  \cite{Nies.Stephan.ea:05} that every non high Schnorr random set is already ML-random. This only requires $\Sigma_1$-induction. The following is somewhat surprising and maybe explains why these two randomness notions are harder to separate than  other pairs of notions. 
\begin{proposition} \label{prop: CR SR} $\mathsf{CR_0}$ is equivalent to  $\mathsf{SR_0}$. \end{proposition}
\begin{proof} Let $\+ M$ be a model of $\mathsf{SR_0}$. Let $X$ be a set of $\+ M$. Arguing within $\+ M$, if no set $Y$ is high in $X$, then $\SR^X = \MLR^X$, so by assumption on $\+ M$ there is a $Z$ in $\CR^X$. Otherwise, some set $Y$ is  high in $X$, i.e.,  $X'' \leT (Y\oplus X)'$, and so $Y \oplus X$ computes a set in $\CR^X$. 
\end{proof}

We note that analytical equivalents such as $\mathsf{BVDiff}$ of a randomness axiom are harder to come by in the absence of a universal test. However, we note that $\CR_0$ is equivalent,  over  $\mathsf{RCA_0} + $ sufficient  induction, to  the statement that for each $X$ there is real $z$ such that each nondecreasing function $f$ computable from $X$ is differentiable at $z$.

\begin{question} {\rm Find other pairs of randomness notions close to each other where the corresponding  principles are equivalent. For instance, consider pairs among $\mathsf{MLR_0}$, difference random reals,    ML-random density one points   Oberwolfach random reals.}\end{question} 

$\mathsf{MLR_0}$ and $\mathsf{DiffR_0}$ should be equivalent over $\mathsf{RCA_0} + $ by an argument similar to  \ref{prop: CR SR}.


\part{Randomness and computable analysis}

\section{Nies: Density and differentiability: dyadic versus full}

\label{s:diff and porous}

\newcommand{\cdf}{{\sf cdf}}

For $U,V$ measurable subsets   of a  measure space $(X, \mu)$, if $\mu(V) >0$ we let
\bc $\mu_V(U) = \mu(U \cap V) / \mu (V)$. \ec
This is the local, or conditional, measure of $U$ with respect to $V$.

The definitions  below   follow \cite{Bienvenu.Hoelzl.ea:12a}. Let~$\lambda$ denote Lebesgue measure. 
\begin{definition}
We define the (lower Lebesgue) density of a set $\sC \subseteq \R$ at a point~$z$ to be the quantity $$\varrho(\sC | z):=\liminf_{\gamma,\delta \rightarrow 0^+} \frac{\lambda([z-\gamma,z+\delta] \cap \sC)}{\gamma + \delta}.$$
\end{definition} 
\n (If we let $I= [z-\gamma, z+\delta]$, the expression above turns into the local measure $\leb_I(\sC)$.)

Intuitively, this measures the  fraction of space  filled by~$\sC$ around~$z$ if we ``zoom in'' arbitrarily close. Note that  $0 \le \varrho(\sC | z) \le 1$.
 We will  first discuss the Lebesgue density theorem.
\begin{theorem}[Lebesgue density theorem]
Let $\sC \subseteq \R$ be a measurable set. Then  $\varrho(\sC | z)=1$ for almost every $z\in \sC$.
\end{theorem}
It is interesting to compare this modern formulation  with the original in \cite[page 407]{Lebesgue:1910}:

\

  	  \scalebox{2.4}{\includegraphics[width=5cm,bb=0 0 380 47]{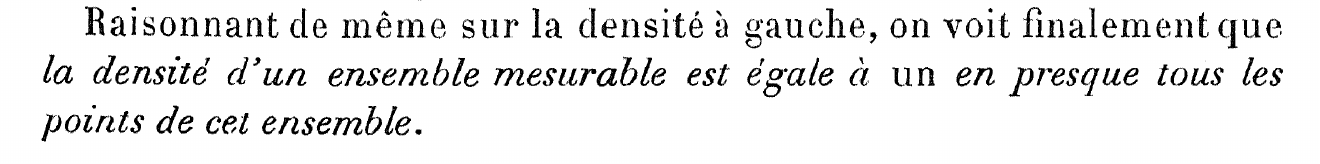}}
 
In 1910, mathematical writing was rather different from what it is today. The statement above is at the \emph{end} of a long argument. None of the statements in the 90-page monograph is labelled; so there is no cross referencing.

A (closed) basic dyadic interval has the form $[r \tp{-n}, (r+1) \tp{-n}]$ where $r \in \mathbb Z, n \in \mathbb N$.  The lower dyadic density of a set $\sC \subseteq \R$ at a point~$z$ is the variant one obtains  when only considering basic dyadic intervals containing~$z$: 
 $$\varrho_2(\sC | z):=\liminf_{z \in I \lland |I| \rightarrow 0} \frac{\lambda(I \cap \sC)}{|I|}, $$
where $I$ ranges over basic dyadic intervals containing $z$. Clearly $\varrho_2(\sC | z) \ge \varrho(\sC | z)$. 
Sometimes we use  \emph{open} basic dyadic intervals; for the definition above the distinction does not matter. 

Suppose that a real $z$ is not a dyadic rational. Let $0.Z$ be its binary expansion. Note that $\varrho_2(\sC|z)$ is the same as 

 $$\liminf_{\sss \prec Z  \lland \sssl  \rightarrow \infty} \frac{\lambda(  [\sss] \cap \sC )}{\tp{-\sssl}}, $$
  when we view $\sC$ as a subclass of $\cantor$. This is the density in Cantor space.

\begin{definition}[\cite{Bienvenu.Hoelzl.ea:12a}]
Consider $z\in[0,1]$.
\begin{itemize}
\item We say that $z$ is a \emph{density-one point}   if $\varrho(\sC | z)=1$ for every effectively closed class $\sC$ containing $z$.  

\item We say that $z$ is a \emph{positive density point} if $\varrho(\sC | z)>0$ for every effectively closed class $\sC$ containing $z$.
\end{itemize}
\end{definition}

By the Lebesgue density theorem and the fact that there are only countably many effectively closed classes, almost every real $z$ is a density-one point. 
Note that 
we can form similar  definitions   with dyadic density.
The distinction between positive and full density is typical for the setting of effective  analysis. In classical analysis, everything is settled by Lebesgue's theorem. In effective analysis, more randomness is required to ensure a real is a  full density-one point.  Day and Miller~\cite{Day.Miller:nd} have built a ML-random real which is a positive density point but not density-one point. They can in fact ensure this real is  $\DII$.

A closely related   notion, \emph{non-porosity},  originates in the work of Denjoy. See for instance \cite[5.8.124]{Bogachev.vol1:07} (but note the typo in the definition there). We say that a set $\sC\sub\mathbb R$ is \emph{porous at} $z$ via  the constant $\varepsilon >0$ if   there exists  arbitrarily small $\beta>0$  such that $(z-\beta, z+ \beta)$ contains an open interval of length $\varepsilon\beta $ that is disjoint from $\sC$. We say that $\sC$ is \emph{porous at} $z$ if it is porous  at $z$ via some~$\varepsilon>0$.

\begin{definition}[\cite{Bienvenu.Hoelzl.ea:12a}]
We call $z$ a \emph{porosity point} if some  effectively closed class to which it belongs is porous at $z$. Otherwise, $z$ is  a \emph{non-porosity point}.
\end{definition}

Clearly, if  $\+C$ is porous at $z$ then  $\varrho(\sC | z)<1$, so $z$ is not   a density-one point. Therefore, almost every point of 
$\+ C$ is a non-porosity point.

\subsection{Dyadic density 1 is equivalent to  full density 1 for ML-randoms reals}

\cite[Remark 3.4]{Bienvenu.Hoelzl.ea:12a}~show that    a  ML-random real which is a \emph{dyadic} positive density point already is a full positive density point; both notions coincide with difference randomness, or equivalently, being ML-random and Turing incomplete~ \cite{Franklin.Ng:10}. Mushfeq Khan and  Joseph Miller have recently proved the analog of this result for density-one points.

\begin{theorem} \label{thm:ddensity vs full density} Let $z$ be a ML-random dyadic density-one point. Then $z$ is a full density-one point.  \end{theorem}  

The actual statement Joe and Mushfeq proved is the following.

\begin{thm} \label{prop:denseporous} Suppose $z$ is a non-porosity point. Let $\+ P$ be a $\PPI$ class, $z\in \+ P$, and $\rho(\+ P\mid z) < 1$. Then already $\rho_2 (P \mid z) < 1$. \end{thm}
Thus, we have the same class $\+ P$ both times. This implies Theorem \ref{thm:ddensity vs full density} since $z$ is Turing incomplete, and hence by the result of Bienvenu et al.\ \cite{Bienvenu.Hoelzl.ea:12a} a non-porosity point.    In Remark~\ref{rem:right-c.e.} we will see that in fact $\rho(\+ P\mid z) =\rho_2 (P \mid z) $ for each non-porosity point $z$.
\begin{proof}[Proof of Thm.\ \ref{prop:denseporous}]
Let $\epsilon >0$ be such that $\rho(P\mid z) < 1-\epsilon$. Assume that $\rho_2 (P \mid z) =1$. Let $n^*$ be sufficiently large so that $\leb_L(\+ P) \ge 1 - \epsilon/4 $ for each basic dyadic interval of length $\le \tp{-n^*}$  containing $z$.

Consider now an arbitrary interval $I$ of length $\le \tp{-n^*}$  with $z \in I$ and $\leb_I(\+ P)< 1-\epsilon$. Let $n$ be such that $\tp{-n+1} > |I| \ge \tp{-n}$; thus, $n \ge n^*$.  We may  cover  $I$ with three consecutive basic dyadic intervals $A,B,C$ of length $\tp{-n}$. Say $z \in B$. Since $\+ P$  is relatively thin in $I$, but thick in $B$, this means that $\+ P$ must be thin in $A$ or $C$. This   leads to    large `holes'  arbitrarily close to $z$   in an appropriate $\PPI$ class $\+ Q$, which shows that $z$ is   a  porosity point. This class   $\+ Q$  consists of the basic dyadic  intervals where   $\+ P$ is thick:

\bc $\+ Q = [0,1] - \bigcup \{ L \colon \, \leb_L(\+ P) < 1- \delta\}$
\ec
where $\delta = \epsilon/4$ and $L$ ranges over \emph{open} basic dyadic intervals.  We obtain that $\+Q$ is porous at $z$ via porosity constant $1/3$.

\n
\emph{Technical detail:} 
We have

 $$\leb(\+ P \cap ( A \cup B \cup C))< 3 \cdot \tp{-n} - \epsilon |I| \le (3 - \epsilon) \tp{-n},$$

while $$\leb ( \+ P \cap B ) \ge (1-\delta ) \tp{-n}.$$ Therefore

$$\leb(\+ P \cap ( A  \cup C))<(2 - (\epsilon -\delta)) \tp{-n},$$
 and so  

\bc 
 $\leb(\+ P \cap A)<(1 - (\epsilon -\delta)/2) \tp{-n}$ 
or   $\leb(\+ P \cap C)<(1 - (\epsilon -\delta)/2) \tp{-n}.$ \ec
Thus, since $\frac 3 8 \epsilon > \delta$ one of $A$, $C$ will be removed from $\+ Q$.

The case that $z \in A$ or $z \in C$ is similar. \end{proof}

\subsection{Background from analysis, and two lemmas on comparing derivatives}
We need notation and  a few definitions, mostly taken  from   \cite{Brattka.Miller.ea:nd} or  \cite{Bienvenu.Hoelzl.ea:12a}. 
For a function~$f\colon\subseteq\R\to\R$, the \emph{slope} at a pair $a,b$ of distinct reals in its domain is
\[
S_f(a,b) = \frac{f(a)-f(b)}{a-b}.
\]
For an interval $A$ with endpoints $a,b$, we also write $S_f(A)$ instead of $S_f(a,b)$. For a string  $\sss$ by $[\sss]$ we denote the closed basic dyadic interval $[0.\sss, 0.\sss + \tp{-\sssl}]$.  The open basic dyadic interval is denoted $(\sss)$.  We write $S_f([\sss])$    with the expected meaning.

If $z$ is in an open neighborhood of the domain of~$f$, the \emph{upper} and \emph{lower derivatives} \label{def_upper_lower_deriv} of $f$ at $z$ are
\[
\ol D f(z)  =  \limsup_{h\ria 0} S_f(z, z+h) \quad  \textnormal{and}    \quad
\underline D f(z)  =  \liminf_{h\ria 0} S_f(z, z+h),
\]
where as usual, $h$ ranges over positive and negative values. The derivative $f'(z)$ exists if and only if these values are equal and finite.
We can also consider the upper and lower \emph{pseudo}-derivatives   defined by:   
\begin{align*}
\utilde Df(x) &= \liminf_{h \to 0^+} \, \{S_f(a,b)  \mid   \, a\le x \le b \lland\, 0 <  b-a\le h\}, \\
\widetilde Df(x) &= \limsup_{h \to 0^+} \,  \{S_f(a,b)     \mid    \, a\le x \le b \lland\, 0 <  b-a\le h\}.
\end{align*}
where $a,b$ range over rationals in $[0,1]$.  We only use them because in our arguments it is often convenient  to consider (rational) intervals containing $x$, rather than intervals with $x$ as an endpoint. Also,  we want to be able to discuss pseudo-differentiability  for partial functions that are defined on all rationals in $[0,1]$, such as in the last section of \cite{Brattka.Miller.ea:nd}.

Brattka et al.\ \cite[after Fact 2.4 ]{Brattka.Miller.ea:nd}  check that $\ul Df(z) \le \utilde Df(z) \le \widetilde Df(z) \le \ol Df(z)$ for any real~$z\in [0,1]$; in 
   \cite[Fact 7.2]{Brattka.Miller.ea:nd}  they  verify    that for continuous functions with domain $[0,1]$, the lower and upper pseudo-derivatives of $f$ coincide with the usual lower and upper derivatives. 
   
   They also  coincide if $f$ is nondecreasing: for instance, to show  $ \utilde Df(z) \le \ul Df(z)$, fix  an arbitrarily small $\epsilon >0$. Given $h > 0$,  choose rationals $a \le z $, $z+h \le b$ such that $(b-a) \le (1+\epsilon ) h$. Then $S_f(z, z+h) \le (1+\epsilon) S_f(a,b)$.

We will use the subscript $2$ to indicate that all the limit operations  are restricted to the case of  basic dyadic intervals containing   $z$. For instance, 

\[\widetilde D_2f(x) = \limsup_{|A| \to 0} \,  \{S_f(A)     \mid    \, x \in A \lland A \text{ is basic dyadic interval}\}. \]

\subsubsection{A pair of   analytical lemmas}
Similar to Theorem~\ref{thm:ddensity vs full density}, we show that  discrepancy of dyadic and full upper/lower derivatives at $z$ implies that some closed set is porous at $z$. 
\begin{lemma}\label{lem:classic diff porous}
	Suppose $f \colon \, [0,1] \to \RR$ is a nondecreasing function. Suppose for a real $z \in [0,1]$, with binary representation  $z = 0.Z$,  there is rational $p$     such that  
	\[\widetilde D_2 f(z) < p  < \widetilde Df(z).\]
	Let $\sss^* \prec Z$ be any string such that  $\forall \sss \, [ Z \succ \sss  \succeq \sss^*   \RA S_f([\sss]) \le p]$.  Then the closed set  
	\begin{equation} \+ C = [\sss^*] - \bigcup \{ (\sigma) \colon \, \sss \succeq \sss^* \lland   S_f([\sss]) >p\},\label{eqn: def C porous} \end{equation}
which contains  $z$,    	is porous at $z$.
\end{lemma}

  \begin{proof}    Suppose $k\in \NN $ is such that  $p(1+\tp{-k+1})<  \widetilde Df(z)$. We show that there exists arbitrarily large $n$ such that some basic dyadic interval  $[a, \dot a]$     of length $\tp{-n-k}$ is disjoint from $\+ C$, and contained in $[z- \tp{-n+2}, z + \tp{-n+2}]$.  In particular,   we can choose $\tp{-k-2}$  as  a porosity constant.

By choice of $k$  there is  an interval $I \ni z$ of  arbitrarily short  positive length such that $  p(1+\tp{-k+1})< S_f(I) $. Let $n$ be such that $\tp{-n+1} > |I| \ge \tp{-n}$. Let $a_0$ be greatest of the form $v \tp{-n-k}$, $v \in \ZZ$, such that $a_0 <  \min I$. 
 Let $a_v = a_0 + v \tp{-n-k}$. Let $r$ be least such that $a_r \ge \max I$. 

Since $f$ is nondecreasing and $a_r - a_0 \le |I| + \tp{-n-k+1} \le  (1+ \tp{-k+1}) |I|$, we have 
\[ S_f (I)	 \le S_f(a_0, a_r) (1+ \tp{-k+1}  ),\]
 and therefore $S_f(a_0,a_r)>p$. Then,  by the averaging property of  slopes at consecutive intervals of equal length, there is an $u<r$ such that  $$S_f(a_u,a_{u+1})>p.$$ Since $(a_u, a_{u+1}) = (\sss)$ for some string $\sss$, this gives the required `hole' in $\+ C$ which is near $z \in I$ and large on the scale of $I$: in the definition of porosity let $\beta = \tp{-n+2}$ and note that  we have   $[a_u, a_{u+1}]  \sub [z- \tp{-n+2}, z + \tp{-n+2}]$ because $z \in I$ and $|I| < \tp{-n+1}$. 
  \end{proof}

There also is a \textbf{dual lemma} for   lower derivatives. Note that it can \emph{not} simply  be obtained from the first by taking $-f$ because the function in the dual lemma  is still non\emph{de}creasing. In fact, now the shortish dyadic intervals we choose in the proof  are all \emph{contained in}~$I$. (So in fact we can get a porosity constant $\tp{-k-1}$.)

\begin{lemma}\label{lem:classic diff porous 2}
	Suppose $f \colon \, [0,1] \to \RR$ is a nondecreasing function. Suppose for a real $z \in [0,1]$, with binary representation  $z = 0.Z$,  there  a rational $q$     such that  
	\[\utilde D f(z) <  q < \utilde D_2f(z).\]
		Let $\sss^* \prec Z$ be any string such that  $\forall \sss \, [ Z \succ \sss  \succeq \sss^*   \RA S_f([\sss]) \ge q]$. Then the   closed set 
	\[ \+ C = [\sss^*] - \bigcup \{ (\sigma) \colon \, \sss \succeq \sss^* \lland   S_f([\sss]) < q\},\]
 which contains $z$, is porous at $z$. \end{lemma}
  \begin{proof}  
The argument is very similar to the previous one. We will show that    we can choose as  a porosity constant $\tp{-k-1}$ where $k\in \NN $ is such that  $\utilde D f(z) < q(1- \tp{-k+1})$.
There is an  interval $I \ni z$ of  arbitrarily short  positive length such that $S_f(I) < q(1- \tp{-k+1})$. As before, let $n$ be such that $\tp{-n+1} > |I| \ge \tp{-n}$. Let $a_0$ be least of the form $v \tp{-n-k}$, $v \in \ZZ$, such that $a_0 \ge  \min (I)$. 
 Let $a_v = a_0 + v \tp{-n-k}$. Let $r$ be greatest  such that $a_r \le \max (I)$. 

Since $f$ is nondecreasing and $a_r - a_0 \ge  |I| - \tp{-n-k+1} \ge  (1- \tp{-k+1}) |I|$, we have 
\[ S_f (I)	 \ge S_f(a_0, a_r) (1- \tp{-k+1}  ),\]
 and therefore $S_f(a_0,a_r)< q$. Then there is an $u<r$ such that  $$S_f(a_u,a_{u+1})< q.$$ As before,  this gives the required hole in $\+ C$ which is near $z \in I$. 
  \end{proof}

 	\subsubsection{Basic dyadic intervals shifted by $1/3$}
	For $m \in \NN$ let $\+D _m $ be the collection of intervals of the form $$[k \tp{-m}, (k+1)\tp{-m}]$$ where $k \in \ZZ$. Let $  \+ D'_m$ be the set of  intervals $(1/3)  +I $ where $I \in \+ D_m$. 
	 We use a  `geometric' fact from  Morayne and Solecki~\cite{Morayne.Solecki:89}:
	\begin{fact} \label{fact:geom} Let $m \ge 1$.  If  $I \in \+ D_m$ and $J \in \+ D'_m$, then the distance between an  endpoint of $I$ and an endpoint of $J$ is at least $1/(3 \cdot 2^m)$.
	\end{fact}
	To see this: assume  that  $k \tp{-m} - ( p \tp{-m} +1/3) < 1/(3 \cdot 2^m)$. This yields $(3k-3p-2^m)/ (3 \cdot2^m) < 1/(3 \cdot 2^m)$, and hence $3| 2^m$, a contradiction.

	In the following we need values of functions at endpoints of any such intervals. So  we think of nondecreasing functions $f \colon \, [0,1] \to \RR$ extended to all of $\RR$ via $f(x) = f(0)$ for $x< 0$ and $f(y) = f(1)$ for $y>1$.  Effectiveness properties, such as computable or interval-c.e. (defined below), are preserved by this because it suffices to compute values of the function in question at rationals.

\subsection{Differentiability of nondecreasing computable functions} \ 
We give a short proof of  the following.
\begin{theorem}[\cite{Brattka.Miller.ea:nd}, Thm.\ 4.1]\label{thm:CRd dyadic diff}
		Suppose $f \colon \, [0,1] \to \RR$ is a nondecreasing computable function. Let $z \in [0,1]$ be computably random. Then $f'(z)$ exists. 
\end{theorem}
\begin{proof}   We  may assume $z> 1/2$, else  we work with $f(x+1/2)$ instead of $f$.

 Recall that a  Cauchy name is a sequence $(p_i) \sN i$,  $p_i \in \QQ$, such that  $\fa k > i  \,  |p_i - p_k | \le \tp{-i}$.
 Consider the computable martingale \bc $M(\sss) = S_f(0. \sss, 0. \sss+ \tp{-\sssl})$. \ec  Computability of~$M$ means that $M(\sss) $ is given by a  uniformly in $\sss$ computable Cauchy name. We denote by  $M(\sss)_u$ the $u$-th term of this Cauchy name, so that $|M(\sss) - M(\sss)_u | \le \tp{-u}$.

Note that  $\lim_n M(Z\uhr n)$ exists and is finite for each computably random real~$Z$. This is a version of Doob martingale convergence; see, for instance \cite{Downey.Hirschfeldt:book}.
Returning  to the language of slopes, the convergence of $M$ on $Z$ means that  $\utilde D_2 f(z)= \widetilde D_2 f(z) < \infty$.

 Assume for a contradiction that  $f'(z)$ fails to exist.

First suppose that  $\widetilde D_2 f(z) <  \widetilde Df(z)$.
Choose rationals $r,p$ such that $\widetilde D_2 f(z) <r  < p  < \widetilde Df(z)$. Choose $u \in \NN$ so large that $\widetilde D_2 f(z) <r - \tp{-u}$ and $r+\tp{-u} <p$.
	As usual let $Z\in \cantor $ be such that $z = 0.Z$. Let  $n^*$ be sufficiently large so that   $  [S_f(A) \le r -\tp{-u}]$ for each basic dyadic interval $A$ containing $z$ and of length $\le \tp{-n^*}$. Choose $k$ with  $p(1+\tp{-k+1})<  \widetilde Df(z)$. Then  Lemma~\ref{lem:classic diff porous} applies via the   string  $\sss^*= Z \uhr {n^*}$.  
	
	We define  a  computable rational-valued martingales $L, L'$ such that $L$ succeeds on $Z$, or $L'$ succeeds on~$Y$ where $0.Y$ is the  binary expansion of  $z-1/3$. 
	
	\vsp

	\n \fbox{\emph{Defining $L$.}} It suffices to consider   strings $\sss \succeq \sss^*$. 
	Let $L(\sss^*)=1$. Suppose $\eta \succeq \sss^* $ and $L(\eta)$ has been defined. 
Check  if there is a string $\alpha$ of length $k+4$ such that $M(\eta \alpha)_u> r$.  (Note we have an algorithm for that because $f$ is computable.) 

If so, bet $0$ on $\eta \alpha$ (we know that $\eta \alpha \not \prec Z$, so this won't make us lose along~$Z$). In return, increase the capital by a factor of $\tp{k+4}/(\tp{k+4}-1)$ along all strings $\eta \widehat \alpha$ such that  $|\widehat \alpha| = k+4$ and $\widehat \alpha \neq \alpha$. Continue the strategy  with all strings $\eta \widehat \alpha$.

If no such $\alpha $ exists, don't bet, that is, let  $L(\eta 0) = L(\eta 1) = L(\eta)$. Continue  with the strings $\eta 0$ and $\eta 1$.

\vsp

	\n \fbox{\emph{Defining $L'$.}} 	Let $\rho^*  = Y\uhr{n^* +1}$. It suffices to consider   strings $\rho \succeq \rho^*$. 

	Let $L'(\rho^*)=1$. Suppose $\rho \succeq \rho^*$ and $L(\rho)$ has been defined. 
Check if there is a string $\beta$ of length $k+5$ such that $[\rho \beta]+1/3 \sub [\tau]$ for a  string $\tau$ of length $|\rho \beta| -1$, and $M(\tau)_u> r$.  

If so, bet $0$ on $\rho \beta$ (we know that $\rho \beta \not \prec Y$). In return, increase the capital by a factor of $\tp{k+5}/(\tp{k+5}-1)$ along all strings $\rho  \widehat \beta$ such that  $|\widehat \beta| = k+5$ and $\widehat \beta \neq \beta$. Continue the strategy with all strings $\rho \widehat \beta$.

If no such $\beta $ exists, don't bet, that is, let  $L(\rho 0) = L(\rho 1) = L(\rho)$. Continue with the strings $\rho 0$ and $\rho 1$.

We show that 
	$L$ succeeds on $Z$, or $L'$ succeeds on $Y$. Let $\+ C$ be the class from (\ref{eqn: def C porous}) in Lemma~\ref{lem:classic diff porous}.  
Consider $n \ge n^*+4$  and a hole $[a,\dot a] \cap \+ C = \ES$ where $[a,\dot a]$ is a basic dyadic interval of length $\tp{-n-k}$, and $[a,\dot a] \sub [z - \tp{-n+2}, z+ \tp{-n+2}]$.

   By Fact~\ref{fact:geom} we have
\begin{claim}
	One of the following is true.
	\bi \item[(i)]  $z, a, \dot a $ are   all  contained in a single interval $A$ taken from $\+  D_{n-4}$. 
	
	\item[(ii)]  $z,a, \dot a $   are  all  contained in a single   interval $A'$ taken from $\+ D'_{n-4}$. \ei
\end{claim}
	In case (i) let $A = [\eta]$, so that $\eta \prec Z$ (recall $Z \not \in \QQ$ so there is no problem with the end points). Let $[a,\dot a] = \eta \alpha$ where $|\alpha|= k+4$. We have $z \not \in [a,\dot a]$, and $L$ increases its   capital by a factor of   $\tp{k+4}/(\tp{k+4}-1)$ along all  strings $\eta \hat \alpha$ as above. 
	
	Now suppose case (ii) applies. Let $\rho$ be the string such that $A'= [\rho]+ 1/3$. There is $[b, \hat b]$ from $\+ D'_{n+k+1}$ with $[b, \dot b] \sub [a, \dot a]$. Since (ii) holds we have $[b, \dot b] = [\rho  \beta]$ for some string $\beta$ of length $k+5$. We have $z \not \in [b,\dot b]$ and $L'$ increases its   capital by a factor of  $\tp{k+5}/(\tp{k+5}-1)$ along all  strings $\rho \widehat \beta$ as above. 
	
	Suppose now that $L$ fails on $Z$. Then   for  all sufficiently long $\gamma \prec Y$ we can   find $\rho$ with $\gamma \preceq \rho \prec Z$ and $L'$ increases its capital by a fixed factor $>1$ on the next $k+5$ bits of $Y$. Also the capital of $L'$ along $Y$   never decreases, because there is no basic dyadic interval $[\tau] \ni z$ with $|\tau| \ge n^*$ and $S_f(\tau)_u \ge r$. So $L'$ succeeds on $Y$.

The case  $\utilde D f(z) <  \utilde D_2f(z)$  is  analogous, using Lemma \ref{lem:classic diff porous 2} instead of Lemma~\ref{lem:classic diff porous}.
\end{proof}

The method of the proof has an interesting consequence. See e.g.~\cite[7.6.2]{Nies:book} or  \cite{Downey.Hirschfeldt:book} for the definition of Church  (or computable) stochasticity. By \cite{Ambos.Mayordomo.ea:96},  also see \cite[6.4.11]{Downey.Hirschfeldt:book},  $X \in \cantor$ is Church stochastic iff no computable  martingale that uses only finitely many, positive rational betting factors can win on $X$. The martingales $L$, $L'$ constructed above are of this kind (in fact we have to modify them slightly in order to avoid betting 0). 
\begin{corollary} Suppose that $z$ is Church stochastic. Then for each   nondecreasing computable function $f \colon \, [0,1] \to \RR$, we have  $\widetilde D_2 f(z) =  \widetilde Df(z)$ and  $ \utilde D_2f(z) = \utilde Df(z)$. \end{corollary} 
This means that on the rather generous class of Church stochastic reals $z$, the lower/upper derivative of a nondecreasing computable $f$ is completely given by the slopes at basic  dyadic intervals containing $z$. In particular, the derivative  at $z$ equals the  dyadic derivative.

\subsection{Polynomial time randomness and differentiability}

Recall that we represent a real $x$ by a Cauchy name $(p_i) \sN i$. We have $p_i \in \QQ$, and $\fa k > i  |p_i - p_k | \le \tp{-i}$. For feasible analysis, we  use  a compact set of Cauchy names:  the signed digit representation of a real. Such Cauchy names, called \emph{special}, have the form $p_i = \sum_{k=0}^i b_k \tp{-k}$, where $b_k \in \{-1,0,1\}$. (Also, $b_0=0, b_1 =1$.)  So they are given by  paths through $\{-1,0,1\}^\omega$, something a resource bounded TM can process. We call the $b_k$ the \emph{symbols} of the special Cauchy name.

\begin{definition}  A function  $g \colon [0,1] \to \RR$ is  polynomial time  computable if there is a  polynomial time TM   turning every special Cauchy name for $x \in [0,1]$ into a  special Cauchy name for $g(x)$.  \end{definition}
 This means that the first $n$ symbols of $g(x)$ can be computed in time poly(n), thereby  using polynomially many symbols of the oracle tape holding~$x$.
Functions such as $e^x, \sin x$ are polynomial time computable, essentially because analysis gives us rapidly converging approximation sequences, such as $\sum x^n/n!$.

The argument given above can be adapted to polynomial time randomness. 
A martingale $M \colon \strcantor \to \RR$ is called polynomial time computable if from  string $\sss$  and $i \in \NN$ we can  in time  polynomial in $\sssl + i$   compute  the $i$-th component of a special Cauchy name for $M(\sss)$. In this case we can compute a polynomial time rational valued martingale dominating $M$ (Schnorr / Figueira-N). We say $Z$ is \emph{polynomial  time random} if no polynomial time martingale succeeds on $Z$. For definitions omitted here  see~\cite{Figueira.Nies:13}.

\begin{theorem}[]\label{thm: poly dyadic diff}
		Let $z \in [0,1]$. Then $z$ is polynomial time  random $\LR$
		
		\hfill   $f'(z)$ exists  for each nondecreasing polynomial time computable 
		
		\hfill    function  $f\colon  [0,1] \to \RR$. 
\end{theorem}
The implication $\RA$  and other  results  were independently proved by A.\ Kawamura, who directly adapted the proof of \cite{Brattka.Miller.ea:nd}, Thm.\ 4.1] to the polynomial time setting. 

\begin{proof}  \lapf Suppose $z$ is not polynomial time random. Then some polynomial time martingale $L$ succeeds on the binary expansion $Z$  of~$z$. By \cite[Lemma 3]{Figueira.Nies:13}, there is a polynomial time martingale $M$ with the savings property that succeeds on $Z$. Let $\mu_M$ be the corresponding measure given by $\mu_M([\sss])= \tp{-\sssl}M(\sss)$. Let $\cdf_M$ be the cumulative distribution function of $\mu_M$ given by $ \cdf_M(x) = \mu_M[0,x)$. By \cite[Lemma 3]{Figueira.Nies:13}, for each dyadic rational $p$, $\cdf_M(p)$ is a dyadic rational that can be computed from $p$ in polynomial time.   Since $ M$ has the savings property, by \cite[Prop.\ 5]{Figueira.Nies:13}, $\cdf_M$ satisfies the `almost Lipschitz condition': there is    $ \epsilon>0$ such that for every
$x,y\in[0,1]$, if $y-x\leq\epsilon$ then
$$
\cdf_M(y)-\cdf_M(x) = O(-(y-x)\cdot\log(y-x)).
$$
This implies that $f=\cdf_M$ is polynomial time computable: Suppose we are   given a special  Cauchy name $(p_i)\sN i$  for a  real $z$. We know that $|z- p_{n+ \log n}| = O(\tp{-n-log n})$. So by the pseudo Lipschitz condition, we have $|f(z)- f(p_{n + \log n})| = O(\tp{-n})$. So  a TM can determine  in polynomial time   from the first $n + \log n$ symbols of the special Cauchy name for $z$    the first $n$ symbols of a special Cauchy name for $f(z)$.

	\rapf
	Since $f$ is polynomial time computable,   all the martingales involved in the proof of  Theorem~\ref{thm:CRd dyadic diff} are  computable in polynomial time. The usual proof of Doob martingale convergence can be turned into a polynomial time construction, and hence  shows that any polynomial time martingale converges on every polynomial random real. Thus we have $\utilde D_2 f(z)= \widetilde D_2 f(z) < \infty$.   Furthermore, by the base invariance of polynomial time randomness~\cite[Thm.\ 4]{Figueira.Nies:13}, if $z$ is polynomially random then so is $z-1/3$.  So $\widetilde D_2 f(z) = \widetilde Df(z)$ and $\utilde D f(z) =  \utilde D_2f(z)$   by the argument given above. 
\end{proof} 
\subsection{Interval  c.e.\ functions} 

\subsubsection{Background} 
We quote from 
\cite{Bienvenu.Greenberg.ea:OWpreprint}. 
Let $g\colon [0,1] \ria \R$. For $0 \le x< y \le 1$ define the \emph{variation} of $g$ in $[x,y]$ by 
$$V(g,[x,y]) = \sup \left\{\sum_{i=1}^{n-1} \bigl| g(t_{i+1}) - g(t_i)\bigr| : x \le t_1 \le t_2 \le \ldots \le t_n \le y\right\}.$$ The function $g$ is of bounded variation if $V(g,[0,1])$ is finite. 
If $g$ is a continuous function of bounded variation then the function $f(x) = V(g, [0,x])$ is  also continuous. If $g$ is computable then the function $f(x) = V(g, [0,x])$ is lower semicomputable (but may fail to be computable). A further property of this ``variation function'' comes from the observation that $V(g,[x,y]) + V(g, [y,z]) = V (g, [x,z])$ for $x< y< z$ (see \cite[Prop.\ 5.2.2]{Bogachev.vol1:07}). 

Identifying the variations of computable functions, Freer, Kjos-Hanssen, Nies and Stephan \cite{Freer.Kjos.ea:nd} studied a class of monotone, continuous, lower semicomputable functions which they called \emph{interval-c.e.}

\begin{definition} \label{def:intervalce}  
A non-decreasing, lower semicontinuous function $f\colon [0,1]\to \R$ is \emph{interval-c.e.}\ if $f(0)=0$, and $f(y)-f(x)$ is a left-c.e.\ real, uniformly in rationals $x<y$. 
\end{definition} 

Thus, the variation function of each computable function of bounded variation is interval-c.e. Freer et al.\ \cite{Freer.Kjos.ea:nd}, together with Rute, showed that conversely, every continuous interval-c.e.\ function is the variation of a computable function.  (End quote.)

Note that the better term would be \emph{interval-left-c.e.} There is also  a dual concept, being \emph{interval right-c.e.}, where $f(y)-f(x)$ is a uniformly a right-c.e.\ real. For instance, the function $f(x) = \leb ([0,x] \cap \+ P)$ for an effectively closed class $\+ P$ is interval right-c.e. There is a curious break of symmetry that the variations of  computable functions are the continous  interval \emph{left}-c.e.\ functions vanishing at~$0$. This seems to say the left-c.e.\ version is the cooler one.

(We note that either  class is closed under the `double mirror' transformation: if $f$ is interval left-c.e. [right c.e.] then so is $\hat f(x)= 1- f(1-x)$. The slopes  $S_{\hat f}(x,y)= S_f(1-y,1-x)$.)

\subsubsection{Interval  (left)-c.e.\ functions: upper dyadic  equals upper full derivative for non-porosity points} 

\begin{proposition}\label{pro:interval c.e.}
Let $f \colon \, [0,1] \to \RR$ be interval-c.e. Then  $\widetilde D_2 f(z) = \widetilde Df(z)$ for each non-porosity point $z$.
\end{proposition}

\begin{proof} Assume $\widetilde D_2 f(z) < \widetilde Df(z)$.  Since $f$ is interval c.e., we can view $S_f(\sss)$ as a left-c.e.\ martingale. In particular, the class $\+ C$ defined in (\ref{eqn: def C porous}) in Lemma~\ref{lem:classic diff porous} is effectively closed. This class is porous at $z$ for a contradiction.
\end{proof}

\subsubsection{Dual fact for interval right-c.e.\ functions}
\begin{remark}
	\label{rem:right-c.e.}
	  If $f$ is interval \emph{right}-c.e.\ we can   apply the dual  Lemma~\ref{lem:classic diff porous 2} to conclude that,  $\utilde D f(z) = \utilde D_2f(z)$ for each non-porosity point $z$. For instance, let $f$ be the Lipschitz function given by $f(x) = \leb ([0,x] \cap \+ P)$ for an effectively closed class $\+ P$. Then we may conclude  that (lower) dyadic density of $\+ P$ at  a non-porosity point $x$  coincides with the (lower)  full density, thereby obtaining a strengthening of Proposition~\ref{prop:denseporous}. 
\end{remark}

\subsubsection{Interval  c.e.\ functions: dyadic  equals   full derivative for reals at  which all left-c.e.\ martingales converge}

Consider a  real $z \in [0,1] - \QQ$.  If a martingale $M$ converges to a finite value at the binary expansion of $z$, we write $M(z)$ for this finite value. We say that $z$  is  a \emph{convergence point for c.e.\ martingales} if $M(z)$ exists for each  c.e.\ martingale $M$.

 Convergence points for c.e.\ martingales coincide with the  ML-random (dyadic) density one points. This was obtained by 2012 work of  a group in Madison consisting of Uri Andrews, Mingzhong Cai, David Diamondstone, Steffen Lempp, and Joseph S.\ Miller. The implication 
\bc martingale convergence $\RA$ density  one \ec 
 was already pointed out in \cite{Bienvenu.Greenberg.ea:OWpreprint}. The hard implication is 
\bc dyadic density one  $\RA$ martingale convergence. \ec 
See Theorem~\ref{thm:Madison} below.

\vsp

%
\begin{theorem}\label{thm:interval left-c.e. MG and derivative} Let $f \colon \, [0,1] \to \RR$ be interval-c.e. Let $z$ be a convergence point for c.e.\ martingales. Then    $f'(z)$ exists.  
\end{theorem}
\begin{proof}   We  may assume $z> 1/2$, else  we work with $f(x+1/2)$ instead of $f$.  
	The real $z$ is a  a dyadic  density one point, hence a (full) density one point by the
	 Khan-Miller Theorem~\ref{thm:ddensity vs full density}. Then $z-1/3$ is also a ML-random density-one point, so  using the  work of the Madison group discussed in Section~\ref{s:Andrews}, $z-1/3$ is also a c.e.\ martingale convergence point. In particular, both  $z$ and $z-1/3$ are    non-porosity points.

	For a nondecreasing function $g\colon [0,1] \to \RR$  recall that $M_g $ is  the (dyadic) martingale associated with the slope $S_g$ evaluated at intervals  of the form $[  i \tp{-n},  (i+1) \tp{-n}]$. Thus, \bc $M_g(\sss) = S_g(0. \sss, 0. \sss+ \tp{-\sssl})$. \ec
	Let $M= M_f$. 
Note that $M$ converges on $z$ by hypothesis.  Thus  $\utilde D_2f(z)= \widetilde D_2f(z) =M(z)$. 

By Proposition~\ref{pro:interval c.e.} again, we have $\widetilde D_2f(z) = \widetilde Df(z)$. It remains to show that 
	\begin{equation} \label{eqn: lower dyadic} \utilde Df(z)= \utilde D_2f(z). \end{equation} Since $f$ is nondecreasing,  this will establish that $f'(z)$ exists.

Let $\widehat f(x) = f(x+1/3)$, and let $M' = M_{\widehat f}$. We now show that  $M'$ converges on $z-1/3$, and the limits   coincide.
\begin{claim}  $M(z) = M'(z-1/3)$. 
\end{claim} 
As pointed out above, $z-1/3$ is also a convergence point for c.e.\ martingales. So    $M'$ converges on $z-1/3$. 
If   $M(z)  < M'(z-1/3)$ then $\widetilde D_2 f(z) <  \widetilde Df(z)$.
However $z$ is a non-porosity point, so this contradicts Proposition~\ref{pro:interval c.e.}. 
	If   $M'(z-1/3)  < M(z) $  we argue similarly using that $z-1/3$ is a non-porosity point. This establishes the claim. Hooray!

To show (\ref{eqn: lower dyadic}), we extend the method in the proof of Lemma~\ref{lem:classic diff porous 2}, taking into account both dyadic intervals, and dyadic intervals shifted by $1/3$. Recall that  $\utilde D_2f(z) = M(z)$. Assume for a contradiction that (\ref{eqn: lower dyadic}) fails. Then we can choose rationals $p,q$ such that

	\[\utilde D f(z) < p < q < M(z) = M'(z-1/3).\]
	Let $k\in \NN$ be such that $p< q(1- \tp{-k+1})$. Let $u,v$ be rationals such that 
	\bc $ q< u <  M(z)  <v$ and $v-u\le \tp{-k-3}(u-q)$. \ec
	
	Let $n^* \in \NN$ be such that for each $n \ge n^*$ and any interval $A\in \+ D_n \cup  \+ D'_n$, we have $S_f(A) \ge u$.

	Let 
	\begin{eqnarray*} \+ E &=&  \{ X \in \cantor \colon \, \fa n \ge n^* M(X\uhr n)\le v \}\\
		\+ E' &=&  \{ W \in \cantor \colon \, \fa n \ge n^* M'(W\uhr n) \le v \} \end{eqnarray*}
		Since $f$ is interval c.e., these classes are $\PPI$. In Cantor space we can apply notions of porosity    via the usual transfer to $[0,1]$ given by the binary expansion.

		Let  $0.Z$ be as usual the binary expansion of $z$.   By the choice of $n^*$ we  have  $Z \in \+ E $. Let  $0.Y$ be the binary expansion of $z-1/3$. We have   $Y \in  \+ E'$.

We will show that $\+ E $ is porous at~$Z$,  or $\+ E'$ is porous at $Y$.

 Consider   an interval  $I \ni z$   of  positive length $\le \tp{-n^*-3}$ such that $S_f(I) \le p$. Let $n$ be such that $\tp{-n+1} > |I| \ge \tp{-n}$. Let $a_0$ [$b_0$] be least of the form $w \tp{-n-k}$  [$w \tp{-n-k} +1/3$], where $w \in \ZZ$, such that $a_0  [b_0] \ge  \min (I)$.  
	 Let $a_i = a_0 + i \tp{-n-k}$ and $b_j = b_0 + j \tp{-n-k}$. Let $r,s$ be greatest  such that $a_r \le \max (I)$ and $b_s \le \max(I)$. 

As before, 	since $f$ is nondecreasing and $a_r - a_0 \ge  |I| - \tp{-n-k+1} \ge  (1- \tp{-k+1}) |I|$,  we have 
$S_f (I)	 \ge S_f(a_0, a_r) (1- \tp{-k+1}  )$,
	 and therefore $S_f(a_0,a_r)< q$. Then there is an $i<r$ such that  $S_f(a_i,a_{i+1})< q$. Similarly, there is $j< s$ such that $S_f(b_j,b_{j+1})< q$.
	
	\begin{claim}
	One of the following is true.
	\bi \item[(i)]  $z, a_i,a_{i+1} $ are   all  contained in a single interval taken from $\+  D_{n-3}$. 
	
	\item[(ii)]  $z, b_j,b_{j+1} $   are  all  contained in a single   interval taken from $\+ D'_{n-3}$. \ei
	\end{claim}
For suppose that (i) fails. Then there   an  endpoint of an $A\in \+ D_{n-3}$ (that is, a number of the form $w\tp{-n+3}$ with $w\in \ZZ$) between $\min (z, a_i) $ and $\max (z, a_{i+1})$. Note that $\min (z, a_i) $ and $\max (z, a_{i+1})$ are  in $I$. By Fact~\ref{fact:geom} and  $|I| < \tp{-n+1}$,  there can be no endpoint of an interval $A' \in \+ D'_{n-3}$ in $I$. Then, since $b_j, b_{j+1} \in I $,  (ii) holds. This establishes  the claim.

Suppose $I$ is an interval as above and  $\tp{-n+1} > |I| \ge \tp{-n}$, where $n \ge n^*+3$. Let $\eta = Z \uhr {n-3}$ and $\eta' = Y \uhr {n-3}$. 
	
	If (i) holds for this $I$ then there is  a string $\alpha$ of length $k+3$ (where $[\eta \alpha]=[a_i, a_{i+1}]$) such that $M( \eta \alpha) < q$. So by the choice of $q< u< v$ and since $M(\eta) \ge u$  there is $\beta$  of length $k+3$  such that $M(\eta \beta)> v$. This yields  a   hole in $\+ E$,  large and near $Z$ on the scale of $I$, which is  required  for porosity of $\+ E$ at $Z$.

 		Similarly, if (ii) holds for this $I$, then there is a string $\alpha$ of length $k+3$ (where $[\eta' \alpha]=[b_j, b_{j+1}]$) such that $M(\eta' \alpha) < q$. So by the choice of $q< u< v$ and since $M'(\eta')\ge u$  there is a string  $\beta$  of length $k+3$  such that $M'(\eta' \beta)> v$.  This yields  a   hole large and near $Y$ on the scale of $I$  required  for porosity of $\+ E'$ at $Y$. 

 Thus, if case (i) applies for arbitrarily short intervals $I$, then $\+ E$ is porous at $Z$, whence $z$ is a porosity point. Otherwise (ii) applies   for intervals below a certain length. Then   $\+ E'$ is porous at $Y$, whence $z-1/3$ is a porosity point. 
\end{proof}

\subsubsection{Interval  c.e.\ functions: characterizing ML-randomness}  \

\n Nies and Stephan have shown that there is an interval c.e.\ function~$h$ whose points of differentiability coincides with the ML-randoms. The same is true for the convergence points of the left-c.e.\  martingale $S_h(\sss)$. All this is obtained from the following stronger statement, which also strengthens \cite[Cor.6.6]{Bienvenu.Greenberg.ea:nd}:

\begin{theorem}\label{thm:interval c.e. char ML-random} There is a  continuous  interval c.e.\ function $h$  such that $h'(x)$ exists for each ML-random real $x$, and $\utilde Dh(x) = \infty$ whenever $x$ is not ML-random.
\end{theorem}
\begin{proof} 
	Brattka  el at.\  \cite[Lemma 6.5]{Brattka.Miller.ea:nd} show that there is a computable function $f$ of bounded variation  (in fact,  absolutely  continuous) such that $f'(z)$ exists only for \ML{} random reals $z$. Let \bc $h(x) = V(f,[0,x])$. \ec
	To see that $h$ is as required, we have to look at the construction of $f$,	which is actually given in \cite[proof of Thm.\  6.1]{Brattka.Miller.ea:nd}, a result on weak 2-randomness. The function $f$ is a superposition of steeper and steeper sawtooth functions based on intervals $C_{m,i}$ of length  rapidly decreasing  in $m$,  which are enumerated into a universal ML test $\la \+ G_m \ra$.  If $x$ is ML-random then $x \not \in  \+ G_m$ for almost every $m$, and hence for each $i$ we have  $x \not \in C_{m,i}$. This means that $h$ is polygonal in a sufficiently small neighbourhood  of $x$, and $x$ is not a break point. So $h'(x)$ exists.
	
	On the other hand, if $x$ is not ML-random then the change in variation due to the infinite superposition of sawteeth above $x$ adds up, and so $\utilde Dh(x) = \infty$. (Save the amazon.) For detail see the hopefully forthcoming paper~\cite{Greenberg.Hoelzl.ea:nd}.
\end{proof}

	\section{Khan: A dyadic density-one point that is not  full density-one}
		(Submitted by Mushfeq Khan, with acknowledgements to Joe Miller for many helpful discussions.)
				
		It seems intuitively likely that being full density-one is a stronger property than being dyadic density-one (see   Section~\ref{s:diff and porous} for definitions). After all, in the case of the latter, we are severely limiting the types of intervals with which we can witness drops in density. In this section, we construct a dyadic density-one point which is not a full density-one point. 
		
		We use the symbol $\mu$ to refer exclusively to the standard Lebesgue measure on Cantor space. If $\sigma$ is a string, and $C$ a measurable set, the shorthand $\mu_\sigma(C)$ denotes the relative measure of $C$ in the cone above $\sigma$. The following lemma, which is a critical part of the argument, is a special case of the Kolmogorov inequality for martingales (see for example, ~\cite[7.1.9]{Nies:book}, and consider the martingale $S(\sigma) = \mu_\sigma(W)$).
		
	\begin{lemma}\label{vicinity_cantor_space}
		Suppose $W \subseteq 2^\omega$ is open. Then for any $\varepsilon$ such that $\mu(W) \le \varepsilon \le 1$, let $U_\varepsilon$ denote the set $\dset{X \in 2^\omega}{ \mu_\rho(W) \ge \varepsilon \textrm{ for some $\rho \prec X$}}$. We call $U_\varepsilon$ the \emph{$\varepsilon$-vicinity} of $W$. Then $\mu(U_\varepsilon) \le \mu(W)/\varepsilon$.
	\end{lemma}
	
	\begin{theorem}\label{weak_vs_full}
		There is a dyadic density-one point that is not a density-one point.
	\end{theorem}
	\begin{proof}
		We build the desired real $Y$ by computable approximation. At each stage $s$ of the construction, we have a sequence of finite strings $\sigma_{0, s} \prec \sigma_{1, s} \prec ... $ approximating $Y$. At the same time, we build a $\Sigma^0_1$ class $B$ whose complement witnesses the fact that $Y$ is not a density-one point. Let $W_e$ denote the upward closure of the $e$-th c.e. set. Each c.e.\ set represents a requirement that needs to be met by $Y$. In other words, for each $e$, if $Y$ is not in $[W_e]$, we require that $\lim_{\rho \prec Y} \mu_\rho([W_e]) = 0$. Priorities are assigned to c.e.\ sets in the usual manner, with $W_j$ having higher priority than $W_i$ for any $i > j$. We make use of the following shorthand: Let $C$ be a measurable set and $\tau$ and $\tau'$ two strings such that $\tau \prec \tau'$. If for every $\rho$ such that $\tau \preceq \rho \prec \tau'$, $\mu_\rho(C) < \alpha$, then we say that \emph{between $\tau$ and $\tau'$, $\mu(C) < \alpha$}.

		At any stage $s$, for each $k$, we will be working above $\sigma_{k, s}$ to define $\sigma_{k+1, s}$. We have two goals in mind: Firstly, for any $e < k$ such that $\sigma_{k, s}$ is not already a member of $W_e$, we must keep the measure of $W_e$ between $\sigma_{k, s}$ and $\sigma_{k+1, s}$ below a certain threshold. If the threshold is exceeded, say at a string $\rho$ between $\sigma_{k, s}$ and $\sigma_{k + 1, s}$, we shall reroute $\sigma_{k+1}$ above $\rho$ to enter $W_e$. Secondly, we must ensure that there is an interval $I \subseteq [\sigma_{k, s}]$ such that $[\sigma_{k + 1, s}] \subseteq I$ and $\mu_I(B) > 1/4$. Both goals must be satisfied while keeping $Y$ from entering $[B]$. Globally, we must maintain the fact that between $\sigma_{k, s}$ and $\sigma_{k + 1, s}$, the measure of $B$ remains \emph{strictly below} a threshold $\beta(k, s)$, which is updated each time we act above $\sigma_{k, s}$ by rerouting $\sigma_{k + 1}$.
		The construction begins by setting $\sigma_{0, 0}$ equal to the empty string.
		
\n \fbox{\emph{Process above $\sigma_{k, s}$.}} When we first start working above $\sigma_{k, s}$, say at stage $s_0$, we set $\beta(k, s_0) = \beta^*(k)$ (see below for how $\beta^*(k)$ is defined). If $k > 0$, then we start by choosing a $\nu \succ \sigma_{k, s_0}$ long enough so that between $\sigma_{k - 1, s_0}$ and $\sigma_{k, s_0}$, $\mu(B \cup [\nu]) < \beta_{k-1, s_0}$. We let $\sigma_{k+1, s_0} = \nu 1 0^j$ and enumerate the string $\nu 0 1^j$ into $B$, where $j$ is chosen large enough so that the measure of $[B]$ between $\sigma_{k, s_0}$ and $\sigma_{k+1, s_0}$ remains below $\beta^*(k)$. If $k = 0$, $\nu$ can be chosen to be the empty string.

		In a subsequent stage $s$, suppose that $C_0, ..., C_l$ are those among the first $k$ c.e.\ sets that $\sigma_{k, s}$ is not already a member of, in order of descending priority. Now if for some $\rho$ between $\sigma_{k, s}$ and $\sigma_{k+1, s}$ and some $j \le l$, $\mu_\rho([C_j])$ exceeds $\sqrt{\beta(k, s)}$ and no action has yet been taken for a higher priority $C_{j'}$, then we act by rerouting $\sigma_{k+1, s}$ above $\rho$. Let $\nu \succeq \rho$ be a string in $C_j$ long enough so that:
		\begin{enumerate}
			\item Between $\rho$ and $\nu$, $\mu([B]) < \sqrt{\beta(k, s)}$. 
			\item $[B] \cap [\nu] = \emptyset$.
			\item If $k > 0$, then between $\sigma_{k-1, s}$ and $\sigma_{k, s}$, $\mu(B \cup [\nu])$ must be strictly less than $\beta(k-1, s)$.
		\end{enumerate}
		Let $j$ be large enough so that between $\sigma_{k, s}$ and $\nu$, $\mu(B \cup [\nu 0 1^j])$ remains strictly below $\sqrt{\beta(k, s)}$. We set $\sigma_{k+1, s + 1} = \nu 1 0^j$ and enumerate $\nu 0 1^j$ into $B$. Finally, we set $\beta(k, s+1) = \sqrt{\beta(k, s)}$.

\n	\fbox{\emph{Choosing $\beta^*(k)$.}} We move $\sigma_{k+1, s+1}$ into $C_j$ when the following is seen to occur at some stage $s$: For some $\rho$ between $\sigma_{k, s}$ and $\sigma_{k+1, s}$, $\mu_\rho([C_j])$ exceeds the measure of the $\sqrt{\beta(k, s)}$-vicinity of $[B]$ above $\rho$, i.e., if $\mu_\rho([C_j]) > \beta(k, s)/\sqrt{\beta(k, s)} > \mu_\rho(B)/\sqrt{\beta(k, s)}$. If this does not occur, we wish to limit the measure of $C_j$ to $2^{-k}$ between $\sigma_{k, s}$ and $\sigma_{k+1, s}$. Each time we act above $\sigma_{k, s}$, the value of $\beta(k, s + 1)$ is magnified by a power of $1/2$, so we require that $\beta^*(k)$ satisfy \[(\beta^*(k))^{1/2^{k+1}} \le 2^{-k}.\]
		
\n \textbf{Verification.}
		\begin{claim}\label{weak_vs_full_meas_preservation}
			Unless we act immediately above $\sigma_{k, s}$, the measure of $B$ between $\sigma_{k, s}$ and $\sigma_{k+1, s}$ remains strictly below $\beta(k, s)$. 
		\end{claim}
		\begin{proof}
			Condition (2) above ensures that if $\sigma_{k, s}$ is redefined at stage $s$ due to an action above $\sigma_{l, s}$ for some $l < k$, then $\mu(B \cap [\sigma_{k, s}]) = 0$. If we act above $\sigma_{k+1, s}$, then condition (3) ensures that $\mu(B)$ remains below $\beta(k, s)$ between $\sigma_{k, s}$ and $\sigma_{k+1, s}$. Note that there is a string $\nu$ such that $\sigma_{k + 1, s} \prec \nu \prec \sigma_{k+2, s}$ and $\mu(B \cup [\nu]) < \beta(k, s)$ between $\sigma_{k, s}$ and $\sigma_{k+1, s}$. So if we act above $\sigma_{l, s}$ for some $l > k + 1$, then we add some measure to $B$, but this measure is contained entirely in $[\nu]$.
		\end{proof}
		\begin{claim}
			We can act above $\sigma_{k, s}$ while satisfying requirements (1) through (3) above.
		\end{claim}
		\begin{proof}
			By Claim~\ref{weak_vs_full_meas_preservation}, $\mu(B) < \beta(k, s)$ between $\sigma_{k, s}$ and $\sigma_{k+1, s}$. So if at stage $s$, for some $\rho$ between $\sigma_{k, s}$ and $\sigma_{k+1, s}$, $\mu(C_j)$ exceeds $\sqrt{\beta(k, s)}$ then by Lemma~\ref{vicinity_cantor_space} there is an $X \in C_j$ extending $\rho$ such that for every $\alpha$ such that $\rho \preceq \alpha \prec X$, $\mu_\alpha(B) < \sqrt{\beta(k, s)}$. Thus there are arbitrarily long strings extending $\rho$ satisfying condition (1). Conditions (2) and (3) are met by simply choosing a long enough such string.
		\end{proof}
		\begin{claim}
			For each $k \in \omega$, $\sigma_k = \lim_s \sigma_{k, s}$ exists, and $Y = \bigcup_{k} \sigma_k$ is total.
		\end{claim}
		\begin{proof}
			Assume that $\sigma_{k, s}$ has stabilized by stage $s$. Then $\sigma_{k+1}$ is redefined above $\sigma_{k, s}$ at most $k$ times.
		\end{proof}
		\begin{claim}
			$Y$ is a dyadic density-one point.
		\end{claim}
		\begin{proof}
			Suppose that $Y \notin [W_e]$. Let $k$ be large enough so that $k > e$ and for all $e' < e$, if $Y \in W_{e'}$, then $\sigma_k \in W_{e'}$. For any $k' > k$, let $s$ be large enough so that $\sigma_{k', s}$ has stabilized. By our choice of $k$, we never act above $\sigma_{k', s}$ for the sake of $W_{e'}$ for any $e' < e$, and by the assumption that $Y \notin [W_e]$, we never act for the sake of $W_e$. Let $t > s$ be such that $\sigma_{k'+1, t}$ has stabilized. For all $t' > t$, between $\sigma_{k', t'}$ and $\sigma_{k' + 1, t'}$, $\mu(W_e)$ does not exceed $\sqrt{\beta(k', t')}$, which is always bounded by $2^{-k}$. 
		\end{proof}
		\begin{claim}
			$Y$ is not a density-one point.
		\end{claim}
		\begin{proof}
			Let $\sigma_k$ and $\sigma_{k+1}$ be the final values of $\sigma_{k, s}$ and $\sigma_{k+1, s}$ respectively. Then by construction there is a string $\nu$ such that $\sigma_k \prec \nu \prec \sigma_{k+1} \prec Y$, and $\sigma_{k+1} = \nu 1 0^j$ for some $j$ and $\nu 0 1^j \in B$. Let $l = |\nu| + j + 1$ and let $I$ be the interval $(0.\nu 1 - 2^{-l}, 0.\nu 1 + 2^{-l})$. Since $Y$ is a dyadic density-one point, $Y$ is not a rational and so $Y \in [0.\nu 1, 0.\nu 1 + 2^{-l}) \subset I$, while the left half of $I$ belongs entirely to $B$. 
		\end{proof}
		This completes the proof of Theorem~\ref{weak_vs_full}. We note that the construction actually  ensures that $\cantor -{B}$ is  porous at $Y$.
	\end{proof}

\section{Nies: Upper density and partial computable randomness}

The \emph{upper}  (Cantor-space) density of a set $\sC \subseteq \cantor$ at a point~$Z$ is: 
 $$\ol \varrho_2(\sC | Z):=\limsup_{\sss \prec Z\lland |\sss| \rightarrow \infty} \frac{\lambda(I \cap \sC)}{|I|}, $$
where $I$ ranges over basic dyadic intervals containing $z$. Bienvenu et al.\ \cite[Prop.\ 5.4]{Bienvenu.Greenberg.ea:preprint} showed that for any effectively closed set $\+ P$ and   ML-random $Z \in \+ P$, we have $\ol \varrho_2(\+ P\mid Z) =1$; this implies of  course that the    upper density in $\RR$  also equals $1$. 

The following shows that ML-randomness was actually too strong an assumption. The right level seems to be  given by the ``ugly duckling'' notion of partial computable randomness. See \cite[Ch.\ 7]{Nies:book} for background. If the measure $\leb \+P$ is a  computable real, then in fact computable randomness of $Z$ suffices. In that case the full dyadic density is $1$.

\begin{proposition} \label{prop:PC random upper density} Let $\+ P \sub \cantor$ be effectively closed. Let  $Z \in \+ P$ be partial computably random. Then $\ol \varrho_2(\+ P \mid Z) =1$. \end{proposition} 

\begin{proof} Suppose there is $q < 1$ and $n^*$ such  that $\leb_\sss(\+ P) < q$ for each $\eta \prec Z$ with $|\eta| \ge n^*$. We will define a partial computable martingale $M$ that succeeds on $Z$. Let $M(\eta)=1 $ for all strings $\eta$ with $|\eta| \le n^*$. Now suppose that $M(\eta)$ has been defined for a string $\eta$ of length at least $n^*$, but $M$ is as yet undefined on extensions of $\eta$. Search for $t > |\eta|$ such that 
\[  \tp{-(t -|\eta|)}  \# \{\tau \mid \, |\tau| = t   \lland [\tau] \cap \+ P _t = \ES\} > 1- q. \]
If $t$ is found, bet all the capital existing at $\eta$ on the strings $\sss \succ \eta$ with $\sssl = t $ that are not $\tau$'s as above, thereby multiplying the capital by $1/q$. Now repeat with all such strings $\sss \succ \eta$  of length $t$. 

The formal definition of  $M$ is  as follows (supplied by Jing Zhang). For all $|\tau|\leq n^*$, $M(\tau)=1$. Next we define $M$ inductively on $2^{<\omega}$. Suppose $M$ has been defined on $\alpha$ and $M(\alpha)=\beta$, let $t\in \omega$ such that $t>|\alpha|$ and let $S=\{\tau\in 2^t: [\tau]\cap P_t\}$ and $r=|S|>2^{t-|\tau|}(1-q)$. For each $\sigma \in 2^t\backslash S$, define $M(\sigma)=\frac{1}{q}\alpha$, and let $\tau^*\in S$ be the leftmost element and define 
\[M(\tau^*)=2^{t-|\tau|}\alpha - \frac{1}{q}\alpha (2^{t-|\tau|}-r)\]

For any $\sigma\succcurlyeq \alpha$ and $|\sigma|<t$, define $M$ accordingly to make $M$ a martingale. 

Next is the verification. First we check that $\forall \tau\preccurlyeq Z$, $M(\tau)$ is defined. We verify this inductively. Suppose $\eta \preccurlyeq Z$ is already defined. Then by assumption, $\lambda_\eta(\bar{P})>(1-q)$. Therefore, there exists a stage $t\in \omega$ such that $\lambda_{\eta}(\bar{P}_t)>(1-q)$. Thus we have $\Sigma_{\tau\in 2^t, \eta \preccurlyeq \tau}\lambda_{\tau}(\bar{P}_t)=2^{t-|\eta|} \lambda_\eta(\bar{P}_t)>2^{t-|\eta|}(1-q)$;  here we use the fact that for   any measurable class $Q\subset 2^\omega$, the function $\sigma \mapsto \lambda_\sigma(Q)$  is a martingale. Therefore, we have found such a $t$ to define a proper extension of $\eta$. By induction, $M$ is defined on $Z$. It is easy to see $Z$ succeeds on $M$ since every time a new string is defined, the capital becomes $\frac{1}{q}>1$ times of the original capital. 

Note that $M$ succeeds on $Z$ because \emph{all}  strings $\sss \prec Z$  of length $\ge n^*$ qualify as possible $\eta$'s where $t$ exists. On the other hand, if $\eta$ is off $Z$ then  there may be no $t$, so $M$ can be partial. 
\end{proof}

\begin{question} Is there  a computably random $Z$ in some $\PPI$ class $\+ P$ so that $\ol \varrho_2(\+ P \mid Z) < 1$ ? \end{question}
\begin{proposition} Let $\+ P \sub \cantor$ be effectively closed with $\leb \+ P$ computable. Let  $Z \in \+ P$ be   computably random. Then $\varrho_2(\+ P \mid Z) =1$. \end{proposition}

\begin{proof} First we show $\ol \varrho_2(\+ P \mid Z) =1$.  The easy, but not quite accurate, argument would be that in the construction above,  before searching for $t$, we ask whether $\leb_\eta (\+ P) < q$; only then do we attempt to find $t$. 

This isn't quite right because ``$\leb_\eta (\+ P) < q$''  is merely $\SI 1$,  even though $\leb_\eta (\+ P) $ is a uniformly in $\eta$ computable real. To amend this, fix $q'< q$ such that in fact $\leb_\sss(\+ P) < q'$ for each $\eta \prec Z$ with $|\eta| \ge n^*$.  We ask simultaneously \bi \item[(1)] whether  $\leb_\eta (\+ P) > q'$; if the positive answer  to this $\SI 1$ question turns up first we don't bet on extensions of $\eta$
\item[(2)]  $\leb_\eta (\+ P) < q$; in this case we bet. \ei
One of the queries must yield an answer.

The computable martingale $\eta \to \leb_\eta(\+ P)$ cannot oscillate along  the computably random $Z$. Thus, the dyadic density $\rho_2(\+ P \mid Z)$ is 1.
\end{proof}

In fact, Schnorr randomness of $Z$ is sufficient as a hypothesis in the preceding proposition by deeper work of \cite{Pathak.Rojas.ea:12} and \cite{Freer.Kjos.ea:nd}. The characteristic function $1_P$ is $L_1$-computable because there is a  sequence $\seq{1_{P_{g(n)}}}\sN n$, where $g$ is a computable function such that $\leb (P_{g(n)}-P) \le \tp{-n}$. Now use e.g.  \cite[Theorem 3.15]{Pathak.Rojas.ea:12}.

\section{Density-one points and Madison tests (written by Nies)}

\label{s:Andrews}
\newcommand{\weight}{\mbox{\rm \textsf{wt}}}

\newcommand{\rk}{\mbox{\rm \textsf{rk}}}

The following is 2012 work of a group at Madison, consisting of U.\ Andrews,  M.\ Cai, D.\  Diamondstone, S.\ Lempp, and l.n.l.\ J.\ S.\ Miller. The writeup below, due to Nies, is based on discussions with Miller, and  Miller's  slides for his talks at the Buenos Aires Semester 2013. Technical details in the verifications have been added.  No proof by the Madison group has appeared so far (June 2014).

A martingale $L \colon \strcantor \to \RR^+_0$ is called \emph{left-c.e.} if $L(\sss)$ is a left-c.e.\ real uniformly in $\sss$. We focus on convergence of such a martingale along a real $Z$, which means that $\lim_n L(Z \uhr n)$ exists in $\RR$. Unlike the case of computable martingales, convergence requires more randomness than boundedness. For instance, let $\+ U = [0, \Om)$, and let  $L(\sss) = \leb (\+ U \mid [\sss])$ (as a shorthand we use $\leb_\sss(\+ U)$ for this conditional measure); then the left-c.e.\ martingale $L$ is bounded by 1 but  diverges on $\Omega$ because $\Om$ is Borel normal.
\begin{theorem}[Andrews, Cai, Diamondstone, Lempp and Miller, 2012] \label{thm:Madison} The following are equivalent for a ML-random real $z \in [0,1]$.

\bi 

\item[(i)] $z$ is a dyadic density-one point.   

\item[(ii)] Every left-c.e.\ martingale converges along $Z$, where $0.Z$ is the binary expansion of $z$. \ei 
\end{theorem}
Note that by  Theorem~\ref{thm:ddensity vs full density},  $z$ is a  full density-one point iff $z$ is a dyadic density-one point. A ML-random satisfying any of these equivalent conditions will be called \emph{density random}.
\begin{proof}

(ii) $\to$ (i) is   \cite[Cor 5.5]{Bienvenu.Greenberg.ea:preprint}. 

\n (ii) $\to$ (i) is   \cite[Corollary 5.5]{Bienvenu.Greenberg.ea:preprint}. 

\n (i) $\to$ (ii).  We can   work within Cantor space because the  dyadic density  of a class $\+ P \sub [0,1]$ at $z$ is the same as the density of $\+ P$ viewed as a subclass of   Cantor space  at $Z$.  We use the  technical  concept of ``Madison tests''. They are  also called density tests, even though they   actually seem to be  motivated by oscillation of martingales. As a first step, Lemma \ref{lem: density to Madison} shows that if $Z \in \cantor$ is a ML-random dyadic density-one point, then $Z$   passes all Madison tests. As a second step, Lemma~\ref{Madison to  MG convergence} shows  that if $Z$ passes all Madison tests, then every left-c.e.\ martingale converges along $Z$.

  We will now  introduce and motivate this technical test concept. We define the weight  of a set   $X \sub \cantor$  as $${\weight} ( X) = \sum_{\sss \in X} \tp{-\sssl}.$$ Let $\sss^\prec = \{ \tau \in \strcantor \colon \, \sss \prec \tau \}$.

\begin{definition}  A \emph{Madison test} is a computable sequence $\seq {U_s}\sN s$ of computable subsets of $\strcantor$  such that $U_0 = \ES$,      there is  a  constant $c$ such that for each  stage $s$ we have $\weight (U_s)  \le c$, and for  all strings $\sss, \tau$, 
\bi \item[(a)]  $\tau \in U_s - U_{s+1} \to \ex \sss \prec \tau \, [ \sss \in U_{s+1} - U_s]$
\item [(b)] $\weight (\sss^\prec \cap U_s) > \tp{-\sssl} \to \sss \in U_s$. 
\ei
Note that by (a),  $U(\sss) := \lim_s U_s(\sss)$ exists for each $\sss$; in fact, $U_s(\sss)$ changes at most $\tp{\sssl}$ times.

We say that $Z$ \emph{fails} $\seq {U_s}\sN s$ if $Z \uhr n \in U$ for infinitely many $n$; otherwise $Z$ \emph{passes} $\seq {U_s}\sN s$.
\end{definition}
 
We show    that $\weight(U_s)\leq \weight(U_{s+1})$, so that  $\weight(U) = \sup_s \weight(U_s) < \infty $  is a left-c.e.\ real.   Suppose  that $\sss$ is  minimal under the prefix relation such that  $\sigma\in U_{s+1}-U_s$.  By (b) and since $\sigma\not \in U_{s}$, we have  $\weight(\sigma^\prec \cap U_s)\leq 2^{-|\sigma|}$. So enumerating $\sss$ adds $\tp{-\sssl}$ to the weight, while the weight of strings  above $\sigma$  removed from $U_s$  is at most $2^{-|\sigma|}$.

\begin{remark} \label{rem:Doob Madison} {\rm The definition of a Madison test  is closely related to Dubins' inequality,  which limits the amount of oscillation a martingale can have; see, for instance, \cite[Exercise 2.14 on pg.\ 238]{Durrett:96}. Note that this inequality implies a version of  the better-known Doob  upcrossing inequality by taking the sum over all $k$. 
	
We only need to discuss these inequalities in the restricted setting of martingales on $\strcantor$.	 Consider a computable rational-valued martingale $B$; that is, $B(\sss)$ is a rational uniformly computed (as a single output)   from~$\sss$. Suppose that $c,d$ are rationals, $0< c< d$,  $B(\estring)< c$, and  $B$  oscillates between values less than $c$ and greater than $d$ along a bit sequence $Z$. An \emph{upcrossing} (for these values)  is a pair of strings $\sss \prec \tau$ such that $B(\sss)< c$, $B(\tau)>d$, and $B(\eta) \le d$ for each $\eta$ such that $\sss \preceq \eta \prec \tau$.  
By Dubins' inequality, for each $k$ we have \begin{equation} \label{eqn: upcr} \leb\{X \colon \, \ex k \, \text{upcrossings along} \,  X\} \le (c/d)^k.\end{equation}
(See \cite[Cor.\ b.7]{Bienvenu.Greenberg.ea:preprint} for a proof of this fact using the  notation of the present paper.)

Suppose  now that $2c< d$. We   define a Madison test that     $Z$ fails.  Strings never leave the computable  approximation of the  test, so (a) holds. 

 We put $\estring$ into $U_0$. If $\sss \in U_{s-1}$,  put into $U_s$ all strings $\eta$ such that $B(\tau) >d$ and $B(\eta) < c$ for some  $\tau \succ \sss$ chosen prefix minimal, and $\eta \succ \tau$ chosen prefix minimal. Let  $U = \bigcup U_s$ (which  is in fact computable). For each $\sigma$, by the upcrossing inequality (\ref{eqn: upcr}) localised to $[\sss]$, we have  $\weight (\sss^\prec \cap U) \le    \tp{-\sssl}  \sum_{k\ge 1} (c/d)^k   <  \tp{-\sssl}$, so (b) is satisfied vacuously.  }
  \end{remark}
  
  As already noted in  \cite{Bienvenu.Greenberg.ea:preprint}, if $B =\sup B_s$ is a left-c.e.\ martingale where the $B_s$ are uniformly computable martingales,  then an upcrossing  apparent at stage $s$ can later disappear because $B(\sss)$ increases.  In this case, as we will see in the proof of Lemma~\ref{Madison to  MG convergence},  the full power of the conditions (a) and (b) is needed to obtain a Madison test from the oscillating behaviour of $B$. 

  We make the first step of the argument outlined above.
\begin{lemma} \label{lem: density to Madison} Let $Z$ be a ML-random dyadic density-one point. Then $Z$ passes each Madison test. \end{lemma}

\begin{proof} Suppose that  a ML-random bit sequence $Z$ fails a Madison test $\seq {U_s}\sN s$. We will build a ML-test $\seq {\+ S^k} \sN k$ such that $\fa \sss \in U \, [ \leb_\sss(\+ S^k) \ge \tp{-k}]$, and  therefore $$\ul \varrho(\cantor - \+ S^k \mid Z) \le 1- \tp{-k}.$$ Since $Z$ is ML-random we have $Z \not \in \+ S^k$ for some $k$. So $Z$ is not a dyadic density-one point, as witnessed by  the $\PPI$ class $\cantor - \+ S^k$. 

To define the $\+ S^k$ we construct, for each $k, t \in \omega$ and each string $\sss \in U_t$,  clopen sets $\+ A^k_{\sss,t} \sub [\sss]$ given by  strong  indices for finite sets of strings computed from $k, \sss, t$, such that  $\leb (\+ A^k_{\sss,t} )= \tp{-\sssl -k}$ for each $\sss \in U_t$.   We  will let $\+ S^k$ be the union of these  sets over all  $\sss$ and $t$. The  clopen sets  for $k$ and a final string $\sss \in U$ will be disjoint from the    $\PPI$ class $\+ S^k$. Condition  (b) on Madison tests ensures that  during the construction,  a  string $\sss$ can  inherit the clopen sets belonging to   its extensions $\tau$,  without risking that  the $\PPI$ class becomes empty above $\sss$.

\vsp

\n \emph{Construction of clopen sets $\+ A^k_{\sss,t} \sub [\sss]$ for $\sss \in U_t$.}

\n  No    sets need to be defined at stage $0$ because  $U_0 = \ES$. 
Suppose at stage $t+1$, we have    $\sss \in U_{t+1} - U_t$. 
 For each $\tau \succ \sss$ such that  $\tau \in U_t - U_{t+1}$, put $\+ A^k_{\tau, t}$ into an  auxiliary clopen set  $ \wt {\+  A}^k_{\sss, t+1}$. Since $\sss \not \in U_t$, by  condition  (b) on Madison tests,  we have  $\weight (\sss^\prec \cap U_t) \le \tp{-\sssl}$, and so inductively \bc $\leb (\wt {\+ A}^k_{\sss,t+1}) \le  \tp{- \sssl -k}$. \ec Now,  to obtain $ \+ A^k_{\sss , t+1}$ we simply add mass from $[\sss]$ to   $\wt  {\+ A}^k_{\sss , t+1}$ in order to ensure equality as required. 

Let $$\+ S^k_t = \bigcup_{\sss \in U_t} \+ A^k_{\sss,t}.$$ Then $\+ S^k_t \sub \+ S^k_{t+1}$ by condition  (a) on Madison tests. Clearly \bc $\leb \+ S^k_t \le  \tp {-k} \weight (U_t)  \le \tp{-k}$. \ec So $\+ S^k = \bigcup_t \+ S ^k_t$ determines  a ML-test.  Since $Z$ is ML-random, we have  $Z \not \in \+ S^k$ for some $k$. If $\sss \in U$ then by construction  $\leb \+ A^k_{\sss,s} = \tp{-\sssl -k}$ for almost all $s$. Thus  $ \leb_\sss(\+ S^k) \ge \tp{-k}$  as required.  \end{proof}

 We begin  the second step of the argument with an intermediate fact. 
\begin{lemma} 
Suppose that $Z$ passes each Madison test. Then $Z$ is computably random.
\end{lemma}
\begin{proof}  Rather than giving a direct proof, we will rely on Remark~\ref{rem:Doob Madison}. Suppose $Z$ is not computably random. The proof of \cite[Thm.\ 4.2]{Freer.Kjos.ea:nd} turns success of a rational-valued computable martingale $M$ with the savings property into oscillation of another such  martingale $B$. Slightly adapting the (arbitrary)  bounds for the oscillation given there, we may assume that  $B$ is as in Remark~\ref{rem:Doob Madison} for $a=2, b=5$:  if $M$ succeeds along $Z$,  then there are infinitely many  upcrossings $\tau \prec \eta \prec Z$,  $B(\tau) < 2$ and $B(\eta) >5$. Therefore $Z$ fails some Madison test. 
\end{proof}

\begin{lemma} \label{Madison to  MG convergence} Suppose that $Z$ passes each Madison test. Then every left-c.e.\ martingale $L$ converges along $Z$. \end{lemma}

\begin{proof}
The $L$ be a left-c.e.\ martingale. Then  $L(\sss) = \sup_s L_s(\sss)$ where $\seq {L_s}$ is a uniformly computable sequence of martingales and $L_0 = 0$ and $L_s(\sss) \le L_{s+1}(\sss)$ for each $\sss$ and~$s$.
Since $Z$ is computably random,   $\lim_n L_s(Z \uhr n)$ exists  for each $s$.
If $L$ diverges along $Z$,  there is $\varepsilon < L(\estring) $ such that  
\[\limsup_n L(Z \uhr n) - \liminf_n L(Z\uhr n) > \varepsilon.\]
Based on this fact  we  define a Madison test that    $Z$ fails. Along with the $U_s$ we define a uniformly computable labelling  function $\gamma_s \colon \, U_s \to \{0, \ldots,s\}$. 

\vsp

\leftskip 0.5cm
\n {\it  Let $U_0 = \ES$. For $s>0$ we put  the empty string  $\estring $ into $U_s$ and let $\gamma_s(\estring) = 0$. If already $\sss \in U_s$ with $\gamma_s(\sss) = t$, then we  also put into $U_s$ all  strings $\tau \succ \sss$ that are minimal under the prefix ordering  with $L_s(\tau) - L_t(\tau) > \varepsilon$. Let $\gamma_s(\tau) $ be the least $r$ with $L_r(\tau) - L_t(\tau) > \varepsilon$.   }

\leftskip 0cm

\vsp

\n Note that $\gamma_s(\tau)$ records    the greatest stage $r \le s $ at which $\tau$ entered~$U_r$. Intuitively, this construction  attempts to find   upcrossings  between the values $ \liminf_n L(Z\uhr n) < \limsup_n L(Z\uhr n)$. Clearly $\lim_n L_t(Z \uhr n \le  \liminf_n L(Z\uhr n)$. If a string $\tau $ as above  is sufficiently long,  then in fact $L_t(\tau) <  \liminf_n L(Z\uhr n)$ so we have an upcrossing.

We  verify that  $\seq {U_s}$ is a Madison test. For condition (a), suppose that  $\tau \in U_s - U_{s+1}$. Let $\sss_0 \prec \sss_1 \prec \ldots \prec \sss_n = \tau$ be the prefixes of $\tau $ in $U_s$. We can  choose a least  $i< n$  such that $\sss_{i+1}$ is no longer the minimal extension of $\sss_i$ at stage $s+1$. Thus there is $\eta$ with $\sss_i \prec \eta \prec \sss_{i+1}$ and $L_{s+1}(\eta) - L_{\gamma_s(\sss_i)}(\eta) > \varepsilon$. Then $\eta \in U_{s+1}$ and $\eta \prec \tau$,  as required.

To verify condition (b) requires more work. 
We fix $s$ and write  $M_t(\eta)$ for $ L_s(\eta ) - L_t(\eta)$.

\begin{claim}  \label{cl:tech} For each $\eta \in U_s$, where $\gamma_s( \eta) = r$, we have
\[\tp{- |\eta|}M_r(\eta) \ge \varepsilon \cdot  \weight (U_s \cap \eta^\prec). \]
\end{claim}
\n In particular, letting  $\eta = \estring$, we obtain that $\weight (U_s)$ is bounded by  a constant $c=1 + L(\estring)/\varepsilon$ as required.
\n 

For $\sss \in U_s$ and $k \in \NN$ let  $U_s^{\sss,k}$  be the strings strictly above  $\sss$ and at a distance  to  $\sss$ of at most $k$, that is, the set of strings $\tau$ such that   there is $\sss = \sss_0 \prec \ldots \prec \sss_m =\tau$ on $U_s$ with $m \le k$ and $\sigma_{i+1}$ a child of $\sss_i$ for each $i<m$.   To establish the claim, we show by induction on $k$ that 
\[\tp{- |\eta|}M_r(\eta) \ge \varepsilon \cdot  \weight (U_s^{\eta, k}). \]
If $k=0$ then $U_s^{\eta,k}$ is empty so   the right hand side is $0$. Now suppose that  $ k>0$. Let $F$  be the  set of immediate successors of $\eta$ on $U_s$. Let $r_\tau = \gamma_s(\tau)$.
By the inductive hypothesis, we have for each  $\tau \in F$%
\begin{eqnarray}   \label{eqn:taus} \tp{-|\tau|} M_r( \tau) & = & \tp{-|\tau|} [(L_{r_\tau}(\tau)- L_r(\tau)) + M_{r_\tau}(\tau)] \\
								& \ge  & \tp{-|\tau|} \cdot \varepsilon + \varepsilon \cdot \weight (U_s ^{ \tau, k-1}). \nonumber   \end{eqnarray}
								Then, taking the sum over all $\tau \in F$, 
 $$ 	\tp{-|\eta|} M_r(\eta) \ge \sum_{\tau \in F} \tp{-|\tau|}M_r(\tau)	\ge  \varepsilon \cdot \weight (U_s ^{\eta,k}). $$
 The first inequality is Kolmogorov's inequality for martingales, using that the $\tau$ form an antichain. For the second inequality we have used (\ref{eqn:taus}) and  that $U_s^{ \eta, k} =F \cup \bigcup_{\tau \in F} U_s ^{ \tau, k-1}$. This completes the induction and  shows the claim. 
 
 Now, to obtain (b), suppose that  $\weight (U_s \cap \sss^\prec)  > \tp{- |\sss|}$. We use   Claim~\ref{cl:tech} to show that $\sss \in U_s$. Assume otherwise. Let $\eta \prec \sss$ be   in $U_s$ with $|\eta|$ maximal, and let $r = \gamma_s(\eta)$. As before, let  $F$   be the prefix minimal extensions of $\sss$ in $U_s$, and $r_\tau = \gamma_s(\tau)$.   Then $L_{r_\tau}(\tau)- L_r(\tau)> \varepsilon$ for $\tau \in F$. Since  $\tau \in U_s$,  we can apply the claim to $\tau$,  so (\ref{eqn:taus}) is valid.  
 
 Arguing as before, but with $\sss$ instead of $\eta$, we have 
 $$ 	\tp{-|\sss|} M_r(\sss) \ge \sum_{\tau \in F} \tp{-|\tau|}M_r(\tau)	\ge  \varepsilon \cdot \weight (U_s \cap \sss^\prec) $$
  (that part of the argument did not use that $\eta \in U_s$).
  Since  $\weight (U_s \cap \sss^\prec)  > \tp{- |\sss|}$, this implies that $M_r(\sss) > \epsilon$. Hence  some $\sss'$  with	$\eta \prec \sss' \preceq \sss$ is in $U_s$, contrary to the maximality of $\eta$.	
  
  This concludes the verification that $\seq {U_s}$ is a Madison test. 		As explained  already, for each $r$ there are infinitely many $n$ with $L(Z\uhr n) - L_r(Z \uhr n) > \varepsilon$. This shows that $Z$ fails this test: suppose inductively that we have $\sss \prec Z$ such that there is a least  $r$ with $\sss \in U_t$ for all $t\ge r$ (so that $\gamma_t(\sss) = r$ for all such $t$). Choose $n > \sssl$ for this $r$. Then from some stage on $\tau = Z \uhr n$ is a viable extension of $\sss$, so $\tau$, or some prefix of it that is  longer than $\sss$, is in~$U$.  	
\end{proof} 

\n This concludes our proof of Thm.\ \ref{thm:Madison}. \end{proof}

 The Oberwolfach  group (Bienvenu, Greenberg, Kucera, Nies, and Turetsky) \cite[Cor.\ 5.5]{Bienvenu.Greenberg.ea:preprint} showed that every OW-random is density random.  The Madison group provided a direct  proof of this fact.   A   left-c.e.\ bounded test is a nested sequence  $\seq{\+ V_n}$ of uniformly  $\Sigma^0_1$ classes such that for some  computable sequence of rationals $\seq{\beta_n}$ and  $\beta = \sup_n \beta_n$ we have $\leb(\+ V_n)\le \beta-\beta_n$ for all $n$. $Z$ fails this test if $Z \in \bigcap_n \+ V_n$.  The OW group introduced  this  test notion and used it for one possible characterisation of OW randomness. The Madison group used these tests (formerly called Auckland tests) directly.
  
 \begin{prop} \label{prop: OW MSN} Every OW random $Z$  is density random. \end{prop}  
 
 \begin{proof} Given left-c.e.\ martingale  $M$ we want to show that $M$ converges along $Z$. Let $M= \sup D_m$ where $D_m$ is a  computable rational valued martingale uniformly in $m$. Let $\beta = M(\la \ra) $ and $\beta_m = D_m( \la \ra)$ so that  $\beta = \sup_m \beta_m$. Let $L_m = M- D_m$ be the ``rest'' martingale at stage $m$.

 Assume that $M$ does not converge along $Z$. Multiplying $M$ by a sufficiently large integer we may then assume that \bc  $1 < \limsup_k  M(Z\uhr k) - \liminf_k M(Z\uhr k)$. \ec Define the left-c.e.\ bounded test by
 
 \[ \+ V_m= \{ Y \colon \, \ex k \, L_m(Y\uhr k) >1\}. \]
 Then by the usual Kolmogorov inequality, we have $\leb \+ V_m \le L_m( \la \ra) = \beta -  \beta_m$.  
 
To show   $Z$ fails $(\+ V_m)$:  $Z$ is computably random, so $l_m = \lim_k D_m(Z \uhr k)$ exists for each $m$. Furthermore, $l_m \le \inf_k M(Z \uhr k)$.  Thus $\ex k \, L_m(Y\uhr k) >1$; namely, the divergence is only due to the rest martingale at stage $m$.    \end{proof}
 We note that this proof fails in the higher setting of  randomness notions. See Section~\ref{s: higher density}.

\section{Nies: Density and higher randomness}

\label{s: higher density}
\newcommand{\QI}{\Pi^1_1}
\newcommand{\VI}{\Sigma^1_1}
\newcommand{\OMC}{\omega_1^{CK}}

By Nies (August).
The work of the Madison group described in Section~\ref{s:Andrews} can be lifted to the domain of higher randomness. Interestingly, density one now can be equivalently required for any $\VI$ class containing the real, not necessarily closed.

We use the following fact due  to  Greenberg. It is a higher analog of the original weaker version of  Prop.\ \ref{prop:PC random upper density}.

\begin{proposition}[N.\  Greenberg, 2013] \label{prop:higher ML random upper density} Let $\+ C \sub \cantor$ be $\VI$. Let  $Z \in \+ C$ be $\QI$-ML-random. Then $\ol \varrho_2(\+ C \mid Z) =1$. \end{proposition} 
\begin{proof} If $\ol \varrho_2(\+ C \mid Z) < 1$ then there is    a positive rational  $q<1$ and $n^*$ such that for   all $n\ge n^*$ we have $\leb_{Z\uhr n}(\+ C)< q$.  Choose a rational $r$ with $q< r<1$. We define $\QI$-anti chains in $\strcantor$ $U_n$, uniformly in $n$. Let $U_0 = \{\la Z \uhr {n^*} \ra\}$. Suppose $U_n$ has been defined.  For each $\sss \in U_n$, at a stage $\alpha$ such that $\leb_\sss(\+ C_\alpha) < q$, we obtain effectively  a hyper-arithmetical  antichain $V$   of extensions of $\sss$ such that $\+ C _\alpha \cap [\sss] \sub \Opcl V$ and $\leb_\sss(\Opcl V) < r$. Put $V$ into $U_{n+1}$. 

Clearly $\leb U_n \le r^n$ for each $n$. Also, $Z \in \bigcap_n U_n$, so $Z$ is not $\QI$-ML-random.  \end{proof} 

A martingale $M\colon \strcantor \to \RR$ is called left-$\QI$ if $M(\sss)$ is a left-$\QI$ real uniformly in $\sss$. 
\begin{thm} Let $Z$ be  $\QI$-ML-random. The following are equivalent.

\bi \item[(i)] $\rho(\+ C \mid Z)=1 $ for each $\Sigma^1_1$-class $\+ C$ containing $Z$.
\item[(ii)] $\rho(\+ C \mid Z)=1 $ for each \emph{closed} $\Sigma^1_1$-class $\+ C$ containing $Z$.
\item[(iii)] each left-$\QI$ martingale converges on $Z$ to a finite value. \ei
\end{thm}

\begin{proof}  (iii) $\to$ (i): The measure of a $\VI$ set is left-$\VI$ in a uniform way (see e.g.\ \cite[Ch.\ 9]{Nies:book}). Therefore $M(\sss)=  1- \leb_\sss(\+ C)$ is a left-$\QI$ martingale. Since $M$ converges along $Z$, and since by Prop.\  \ref{prop:higher ML random upper density} $\liminf_n M(Z\uhr n) = 0$,  it converges along $Z$ to $0$.  This shows that $\rho(\+ C \mid Z)=1 $.

\n (ii) $\to$ (iii).   We follow the proof of the Madison Theorem~\ref{thm:Madison}  given above. All stages $s$ are now interpreted as computable ordinals.  Computable functions/  constructions, are now  functions   $\OMC \to L_{\OMC}$ with $\Sigma_1$ graph/  assignments of recursive ordinals to instructions.

\begin{definition}  A \emph{$\QI$-Madison test} is a  $\Sigma_1$ over $L_{\omega_1^{CK}}$ function  $\seq {U_s}_{ s< \omega_1^{CK}}$ mapping ordinals to (hyperarithmetical)  subsets of $\strcantor$  such that $U_0 = \ES$,    for each  stage $s$ we have $\weight (U_s)  \le c$ for some constant $c$, and for  all strings $\sss, \tau$, 
\bi \item[(a)]  $\tau \in U_s - U_{s+1} \to \ex \sss \prec \tau \, [ \sss \in U_{s+1} - U_s]$
\item [(b)] $\weight (\sss^\prec \cap U_s) > \tp{-\sssl} \to \sss \in U_s$. 
\ei 
Also $U_\gamma(\sigma) = \lim_{\alpha < \gamma} U_\alpha(\sss)$ for each limit ordinal $\gamma$. 
\end{definition}

The following well-known fact can be proved similar to \cite[1.9.19]{Nies:book}.
\begin{lemma} \label{lem:extend open} Let $\+ A \sub \cantor$ be a hyperarithmetical open. Given  a rational   $q$ with $q >  \leb A$, we can effectively determine from $\+ A, q$  a hyperarithmetical open $\+ S \supseteq \+ A$ with $\leb \+ S = q$. \end{lemma}

\begin{lemma} \label{lem: density to higher Madison} Let $Z$ be a $\QI$ ML-random such that $\rho(\+ C \mid Z)=1 $ for each \emph{closed} $\Sigma^1_1$-class $\+ C$ containing $Z$. Then $Z$ passes each $\QI$-Madison test. \end{lemma}
The proof follows  the proof of the analogous Lemma~\ref{lem: density to Madison}.  The sets $\+ A^k_{\sss,s}$  are now hyperarithmetical open sets computed from $k,\sss, s$.  Suppose $\sss \in U_{s+1} - U_s$. The set $\wt {\+ A}^k_{\sss,s}$  is defined as before. To effectively obtain $ \+ A^k_{\sss , s+1}$, we apply Lemma~\ref{lem:extend open} to  add mass from $[\sss]$ to   $\wt  {\+ A}^k_{\sss , s+1}$ in order to ensure $\leb (  {\+ A}^k_{\sss,s+1}) = \tp{- \sssl -k}$ as required. 

As before let   $\+ S^k_t = \bigcup_{\sss \in U_t} \+ A^k_{\sss,t}$. Then $\+ S^k_t \sub \+ S^k_{t+1}$ by property (a) of $\QI$ Madison tests. Clearly $\leb \+ S^k_t \le  \tp {-k} \weight (U_t)  \le \tp{-k}$. So $\+ S^k = \bigcup_{t< \OMC} \+ S ^k_t$ determines  a $\QI$ ML-test. 

By construction 
$\ol \rho(\cantor - \+ S^k \mid Z) \le 1- \tp{-k}$. Since $Z$ is ML-random we have $Z \not \in \+ S^k$ for some $k$. So $\ol \rho(\+ C \mid Z) < 1 $ for the   {closed} $\Sigma^1_1$-class $\+ C = \cantor - \+ S_k$ containing $Z$.

The analog of Lemma~\ref{Madison to  MG convergence} also holds.
\begin{lemma} Suppose that $Z$ passes each $\QI$-Madison test. Then every left-$\QI$  martingale $L$ converges along $Z$. \end{lemma}
The proof of \ref{Madison to  MG convergence} was already set up so that this works. The uniformly hyp  labelling functions $\gamma_s$ now map $U_s$ to $\OMC$. Note that the antichains $F$ can now be infinite.
\end{proof}

A $\QI$ ML-random satisfying any of the three conditions above will be called {$\QI$-density random}. We note the following implications, none of which are known to be proper.

\bc higher weak 2 random $\RA$ $\QI$ OW-random $\RA$ $\QI$ density random. \ec

The first implication is due to Bienvenu, Greenberg and Monin.
The second is the higher analog of Proposition \ref{prop: OW MSN}:

 \begin{prop} \label{prop: higher OW MSN} Every $\QI$ OW random $Z$  is  $\QI$ density random. \end{prop}  
 However,  we need to go back to the original proof  \cite[Cor.\ 5.5]{Bienvenu.Greenberg.ea:preprint}. The reason is that the left-c.e.\ bounded tests don't make sense in the higher setting; Oberwolfach tests, in contrast, can be suitably adapted. The $n$-th test component is obtained by counting $n$ oscillations.

%

\section{Miyabe: Being a Lebesgue point for each integral tests}

Input by Kenshi Miyabe.

\begin{theorem}\label{th:density-one-converge}
The following are equivalent for a ML-random real $z \in [0,1]$.
\begin{enumerate}
\item $z$ is a density-one point.
\item Every left-c.e.\ martingale converges on $z$.
\end{enumerate}
\end{theorem}

A ML-random satisfying one of these conditions will be called 
\emph{density random}.
This theorem follows from the following theorems.

\begin{theorem}[Mushfeq Khan and Joseph Miller]\label{th:ML-dyadic-full}
Let $z$ be a ML-random dyadic density-one point.
Then $z$ is a full density-one point.
\end{theorem}

\begin{theorem}[Bienvenu et al.\ \cite{Bienvenu.Greenberg.ea:OWpreprint}]
If every left-c.e.\ martingale converges on $z$,
then $z$ is a dyadic density-one point.
\end{theorem}

\begin{theorem}[Andrews, Cai, Diamondstone, Lempp and Miller, 2012]
If $z$ is a ML-random dyadic density-one point,
then every left-c.e.\ martingale converges on $z$.
\end{theorem}

Here, we give a characteriziation of density randomness
via the Lebesgue differentiation theorem.

\begin{theorem}\label{th:density-lebesgue}
The following are equivalent for $z\in[0,1]$:
\begin{enumerate}
\item $z$ is density random.
\item $z$ is a dyadic Lebesgue point for each integral test.
\item $z$ is a Lebesgue point for each integral test.
\end{enumerate}
\end{theorem}

Recall that an \emph{integral test} on $[0,1]$ with the Lebesgue measure
is an integrable lower semicomputable function
$f:[0,1]\to\overline{\mathbb{R}}^+$.

Note that one direction is easy.

\begin{proof}[Proof of (ii) $\Rightarrow$ (i) of Theorem \ref{th:density-lebesgue}]
Suppose that $z$ is a Lebesgue point for each integral test.
Then $f(z)$ is finite for each integral test $f$,
whence $z$ is ML-random.

Let $C$ be a $\Pi^0_1$ class containing $z$.
We define a function $f:[0,1]\to\overline{\mathbb{R}}^+$ by
\[f(x)=\begin{cases}1&\mbox{ if }x\not\in C\\
0&\mbox{ if }x\in C.\end{cases}\]
Then, $f$ is an integral test.
Since $z$ is a Lebesuge point for $f$,
$C$ has density-one at $z$.
\end{proof}

For the converse, we first show the following lemma.

\begin{lemma}\label{lem:weak-l}
If an ML-random set $z$ is a dyadic weak Lebesgue point for an integral test $f$,
then $z$ is a dyadic Lebesgue point for $f$.
\end{lemma}

\begin{proof}
As a notation, for a function $f:\subseteq [0,1]\to\mathbb{R}$ and $z\in[0,1]$,
let
\[D(f,\sigma)=\frac{\int_{[\sigma]}f\ d\mu}{2^{-n}}.\]
Then, $z$ is a dyadic Lebesgue point
iff $\lim_n D(f,z\uh n)=f(z)$.
If $f$ is a integral test, then $D(f,-)$ is a left-c.e.\ martingale.

Suppose that $z$ is not a dyadic Lebesgue point
for an integral test $f$
and $z$ is a dyadic weak Lebesgue point for $f$.
Then $\lim_n D(f,z\uh n)=:r$ exists and $f(z)\ne r$.

Let
\[f=\sup_s f_s\]
where $\{f_s\}$ is a computable sequence of rational step functions.
Then, there is a computable order $u$ such that
$D(f_s,\sigma)=D(f_s,\sigma0)=D(f_s,\sigma1)$
for each $\sigma$ satisfying $|\sigma|\ge u(s)$.
Unless $z$ is a dyadic rational,
we have
\[\lim_n D(f_s,z\uh n)=f_s(z).\]

Suppose that $r<f(z)$.
Since $\lim_s f_s(z)=f(z)$,
there is $t$ such that
\[r<f_t(z)\le f(z).\]
Then
\[r<f_t(z)=\lim_n D(f_t,z\uh n)\le\lim_n D(f,z\uh n).\]
This is a contradiction.

Suppose that $r>f(z)$.
Let $q$ be a rational such that $f(z)<q<r$.
We build a new integral test $g$ such that $g(z)=\infty$.

We prepare auxiliary uniformly c.e.\ sets $\{S_n\}$
where $S_n\subseteq2^{<\omega}\times\omega$ for each $n$.
Let $S_0=\{(\lambda,0)\}$ where $\lambda$ is the empty string.
For each $n\ge1$ and $(\sigma,s)\in S_{n-1}$,
computably enumerate $(\tau,t)$ into $S_n^\sigma$ so that
\begin{itemize}
\item $\sigma\prec\tau$,
\item $|\tau|\ge u(s)$,
\item $D(f_t,\tau)>q$,
\item $\{\tau\in2^{<\omega}\ :\ (\tau,t)\in S_n^\sigma\}$ is prefix-free,
\end{itemize}
We can further assume that
\[\bigcup\{[\tau]\ :\ \sigma\prec\tau,\ |\tau|\ge u(s),\ D(f,\tau)>q\}=\bigcup\{[\tau]\ :\ (\tau,t)\in S_n^\sigma\}.\]
Let $S_n=\bigcup_{(\sigma,s)\in S_{n-1}}S_n^\sigma$.

For each $(\tau,t)\in S_n$,
let
\[g_\tau=(q-D(f_s,\sigma))\mathbf{1}_{[\tau]}\]
where $(\sigma,s)\in S_{n-1}$ and $\sigma\prec\tau$.
We define $g$ by
\[g=\sum_n\sum_{(\tau,t)\in S_n}g_\tau.\]
Note that
\[\int g_\tau\ d\mu\le(D(f_t,\tau)-D(f_s,\sigma))2^{-|\tau|}
=\int_{[\tau]}(f_t-f_s)d\mu,\]
thus $\int g\ d\mu\le \int f\ d\mu<\infty$.
Hence, $g$ is an integral test.

Since $\lim_n D(f,z\uh n)=r>q$,
there exists $(\tau_n,t_n)\in S_n$ such that $\tau_n\prec z$ for each $n$.
Then,
\[g(z)=\sum_n (q-D(f_s,\sigma))
\ge\sum_n(q-f(z))=\infty.\]
\end{proof}

\begin{proof}[Proof of (i) $\Rightarrow$ (ii) of Theorem \ref{th:density-one-converge}]
Suppose that $z$ is density random.
Let $f$ be an integral test.
Then $D(f,-)$ is a left-c.e.\ martingale.
By Theorem \ref{th:density-one-converge},
$\lim_n D(f,z\uh n)$ exists,
whence $z$ is a dyadic weak Lebesgue point for $f$.
By Lemma \ref{lem:weak-l},
$z$ is a dyadic Lebesgue point for $f$.
\end{proof}

To drop ``dyadic'', we recall the following results.

\begin{proposition}
Let $f:[0,1]\to\mathbb{R}$ be interval-c.e.
Then $\widetilde{D}_2f(z)=\widetilde{D}f(z)$
and $\utilde{D}_2 f(z)=\utilde{D} f(z)$ for each non-porosity point $z$.
\end{proposition}

\begin{lemma}[ \cite{Brattka.Miller.ea:nd}; after Fact 2.4, Fact 7.2]
For each real $z$, 
\[\underline{D}f(z)\le\utilde{D}f(z)\le \widetilde{D}f(z)\le\overline{D}f(z).\]
If $f$ is continuous, then
\[\utilde{D}f(z)=\underline{D}f(z)\mbox{ and }
\widetilde{D}f(z)=\overline{D}f(z).\]
\end{lemma}

\begin{lemma}[Lemma 3.8 in \cite{Bienvenu.Hoelzl.ea:12a}]
Let $C$ be a $\Pi^0_1$ class.
If $z\in C$ is difference random, then $C$ is not porous at $z$.
\end{lemma}

\begin{proof}[Proof of (ii) $\iff$ (iii) of Theorem \ref{th:density-one-converge}]
Note that (iii) $\Rightarrow$ (ii) holds by definition.

We prove (ii) $\Rightarrow$ (iii).
Suppose that $z$ is a dyadic Lebesgue point for each integral test.
Then, $z$ is density random, whence difference random,
thus a non-porosity point.

Let $f$ be an integral test.
Then, $F(x)=\int_{[0,x]}f\ d\mu$ is interval-c.e.\ and continuous.
Hence,
\[\limsup_{Q\to z}\frac{\int_Q f\ d\mu}{\mu(Q)}
=\overline{D}F(z)
=\widetilde{D}F(z)
=\widetilde{D}_2F(z)
=\limsup_{n\to\infty}\frac{\int_{[z\uh n]}f\ d\mu}{\mu(Q)}
=f(z).\]
Similarly, we have
$\liminf_{Q\to z}\frac{\int_Q f\ d\mu}{\mu(Q)}=f(z)$,
whence $z$ is a Lebesgue point for $f$.
\end{proof}

Actually, only one integral test characterizes density randomness.

\begin{lemma}
Let $f,g$ be integral tests.
If an ML-random set $x$ is a dyadic weak Lebesgue point for $f+g$,
then $x$ is a dyadic weak Lebesgue point for $f$.
\end{lemma}

\begin{proof}
Suppose that $x$ is a dyadic weak Lebesug point for $f+g$
and $x$ is not a dyadic Lebesgue point for $f$.
Let
\[r=\lim_n D(f+g,x\uh n).\]
Then, there are rationals $p,q$ $(p<q)$ such that
$D_n(f,x\uh n)>q$ for infinitely many $n$
and $D_n(f,x\uh n)<p$ for infinitely many $n$.
Notice that $q\le r$.
By replacing $q$ with $\frac{p+q}{2}$, we can assume that $q<r$.
Let $\epsilon=\frac{q-p}{3}>0$.
Then, there is a natural number $N$ such that, for each $n>N$, we have
\[|D(f+g,x\uh n)-r|<\epsilon.\]
Hence,
\[r-\epsilon<D(f,x\uh n)+D(g,x\uh n)<r+\epsilon.\]
If $D(f,x\uh n)>q$, then
\[D(g,x\uh n)<r+\epsilon-q.\]
If $D(f,x\uh n)<p$, then
\[D(g,x\uh n)>r-\epsilon-p.\]
If $D(g,x\uh n)>r-\epsilon-p$, then
\[D(f,x\uh n)<r+\epsilon-r+\epsilon+p=2\epsilon+p<q.\]

We consider the following betting strategy.
First use the strategy $f$ until $D(f,x\uh n)>q$.
When found, stop betting until $D(g,x\uh n)>r-\epsilon-q$.
At the stage $n$, use the strategy
\[\frac{q}{2\epsilon+p}f.\]
Then, $x$ is not ML-random.
\end{proof}

\begin{theorem}
Let $f$ be a Solovay-complete integral test.
Then $x$ is a Lebesuge point for $f$ if and only if $x$ is density random.
\end{theorem}

\begin{proof}
The ``if'' direction follows from Theorem \ref{th:density-lebesgue}.

Suppose $x$ is not density random.
We can assume that $x$ is ML-random, because, otherwise,
$f(x)=\infty$ and $x$ is not a dyadic Lebesgue point for $f$.
Then there is an integral test $g$ such that
$x$ is not a dyadic Lebesgue point for $g$.
Since $f$ is Solovay-complete,
there are a rational $q$ and an integral test $h$ such that
\[f=\frac{g}{q}+h.\]
Notice that $x$ is not a dyadic Lebesgue point for $\dfrac{g}{q}$.
By Lemma \ref{lem:weak-l},
$x$ is not a dyadic weak Lebesgue point for $\dfrac{g}{q}$.
By the lemmas above, $x$ is not a dyadic weak Lebesgue point for $f$.
Thus, $x$ is not a Lebesgue point for~$f$.
\end{proof}

 \newpage

\part{Similarity relations  for Polish metric spaces}

\n In October 2013, Andr\'e Nies gave a talk as part of  the Universality and Homogeneity Trimester at the Hausdorff Institute for Mathematics (HIM)  in Bonn. The summary follows.

 We are given a class of structures. We always mean concrete presentations of structures (rather than ``up to isomorphism''). We address the following

\n {\bf leading questions} for this class:
\bi \item[(a)] Which similarity relations are there on the class?
\item[(b)] How complex  are these similarity relations?

\item[(c)] If structures $X,Y$ in the class are similar, how complex, relative to $X,Y$, is the means for showing this? For instance, if $X \cong Y$, can one compute an isomorphism from the structures?
\ei

In the model theoretic setting, we could be given the countable models of a first-order theory. In this setting, some  answers  to these questions are: 

\bi \item[(a)] isomorphism $\cong$, elementary equivalence $\equiv$, elementary equivalence $\equiv_\aaa$ for $L_{\omega_1, \omega}$ sentences of rank $< \alpha$.  

\item [(b)]   Isomorphism of countable graphs, linear orders, countable Boolean algebras is $\le_B$ complete for orbit equivalence relations of continuous $S_\infty$ actions (where $\le_B$ is Borel reducibility, and  $S_\infty$ is the Polish group of permutations of $\omega$).

\item  [(c)]  Suppose the  similarity   is $\cong$. For certain natural classes, this question has been  answered in computable model theory. That area introduced  the  notion of being \emph{relatively computably categorical}, where presentations of $X,Y$ together uniformly  compute an isomorphism  if there is one at all. For instance, a dense linear order  is r.c.c.  There are variants, such as being \emph{uniformly computably categorical}, where one computes an isomorphism from   computable indices  for the  structures.
\ei

We will be mainly considering the {\bf metric} setting.  We are given  a class of Polish metric spaces. To answer (a): The following similarities,  which will be defined formally below, have been studied.

\bc Isometry $\cong_i$, homeomorphism $\cong_h$,  \ec 

\bc Gromov-Hausdorff distance $0$, Lipschitz equivalence.  \ec

The  former two are discussed in detail in \cite[Ch.\ 14]{Gao:09}. 
The latter two  are due to  Gromov; see his  book  \cite[Ch.3]{Gromov:07}  (the first edition dates from 1998). After some preliminary facts, we will answer (b) and (c) for the metric setting. 
We also consider Polish metric spaces  with some additional structure, such as Banach spaces,  or spaces with a probability measure  on the Borel sets.

\section{Representing Polish metric spaces}

We adopt the global view. Single  structures are thought of as points in a ``hyperspace''. To endow this hyperspace with its own structure,  it matters how we represent a single structure. For metric spaces,  two ways are common.

\bi \item[(1)]  Let $\UM$  denote the Urysohn space. Let $F(\UM)$ denotes its Effros algebra, which is  a $\sss$-algebra where the points are  closed subsets of $\UM$. Each Polish metric space is isometric to an element of $F(\UM)$. See  Gao \cite[Ch.\ 14]{Gao:09}.

\item[(2)] A point $V = \seq {v_{i.k}}\sN{i,k} \in \mathbb R^{\NN \times \NN}$ is a   {\it distance matrix}  if $V$ is a pseudo-metric on $\NN$. Let $M_V$ denote the  completion of the corresponding pseudo-metric space. This means that  in $M_V$ we have a distinguished dense sequence  of points $\seq {p_i}$ and present the space by giving their distances. We  merely ask  that $V$ is a pseudo-metric in order to ensure that the set $\+ M$ of distance matrices is closed in $R^{\NN \times \NN}$. 
\ei

Both representations are in a sense equivalent as pointed out for instance in~\cite[Ch.\ 14]{Gao:09}. However, the second one is better for studying  the complexity of the space. For instance, a computable metric space  $(M, d, \seq {p_i})$  is given by a distance matrix $w$ such that $w_{i,k} = d(p_i, p_k)$  is a computable real uniformly in $i,k$. 

A Polish group action is a continuous action $G \times X \to X$ where $G$ is a Polish group and $X$ a Polish space. We write $G \curvearrowright X$ to say that $G$ acts on $X$ continuously. The corresponding  \emph{orbit equivalence relation} is $E^X_G = \{\la x, y \ra \colon \, \ex g \, [ g x = y]\}$.

  \section{Polish metric spaces and  the classical Scott analysis.}  A metric space $(M,d)$ can be turned into a structure in the language with binary relations $S_q$ for $ q \in \QQ^+$, where $S_q(a,b)$ holds  if $d(a,b) < q$.

\begin{definition}
	Let $M$ be an $\mathcal{L}$-structure.
	We define inductively what it means for
	finite tuples $\bar{a},\bar{b}$  from $M$   of the same length 
	 
	to be $\alpha$-equivalent, denoted by
	$\bar{a}\equiv_\alpha\bar{b}$.
	\begin{itemize}
		\item $\bar{a}\equiv_0\bar{b}$ if and
		only if the quantifier-free types of
		the tuples are the same.
		\item For a limit ordinal $\alpha$,
		$\bar{a}\equiv_\alpha\bar{b}$ if and
		only if $\bar{a}\equiv_\beta\bar{b}$
		for all $\beta<\alpha$.
		\item $\bar{a}\equiv_{\alpha+1}\bar{b}$
		if and only if both of the following hold:
		\begin{itemize}
			\item For all $x\in M$, there is some
			$y\in M$ such that 
			$\bar{a}x\equiv_\alpha\bar{b}y$
			\item For all $y\in M$, there is some
			$x\in M$ such that 
			$\bar{a}x\equiv_\alpha\bar{b}y$
		\end{itemize}
	\end{itemize}
\end{definition}

The \emph{Scott rank} $\mathrm{sr}(M)$ of a 
structure $M$ is defined as the smallest 
$\alpha$ such that $\equiv_\alpha$
implies $\equiv_{\alpha+1}$ for all tuples 
of that structure. We remark that always
$\mathrm{sr}(M)<|M|^+$.

\begin{fact} A Polish space has Scott rank $0$ iff it is ultrahomogeneous.  \end{fact}

Friedman, K\"orwien and Nies (2012) have shown  that for each $\alpha < \omega_1$, there is a   countable Polish  ultrametric space $M$  such that  $\mathrm{sr}(M) = \alpha  \times \omega$.

\begin{question} \

\n (a) Does every Polish metric space have  countable Scott rank? 

\n
(b) Can  it in fact be described within the class of Polish metric spaces by  an  $L_{\omega_1, \omega} $     sentence?  \end{question}

\n Note (Feb 2014).  Question (a) has been answered in the affirmative by Michal Ducha, a postdoc from  Warsaw (student of J. Zapletal)  who participated in  the 
HIM program.

\section{Isometry $\cong_i$}

In 1998 Anatoly Vershik~\cite{Vershik:98}   asked about the complexity of isometry $\cong_i$ on Polish metric spaces, and in particular if one can assign invariants. The answer was a resounding no. By   the following result,  $\cong_i$ is Borel equivalent to  $E^{F(\UM)}_ { \text{Iso}(\UM)}$, the  universal  orbit equivalence relation given by the action of the isometry group of $\UM$ on the Effros algebra of $\UM$.

\begin{theorem}[Gao-Kechris 2000; Clemens;  see  \cite{Gao:09}, Ch.\ 14] \   \bi \item[(1)] $\cong_i  \  \le_B \  E^{F(\UM)}_ { \text{Iso}(\UM)}$.

\item[(2)] For every Polish group action $G \curvearrowright X$ we have   $E^X_G \  \le_B \  \cong_i$.  \ei \end{theorem} 

Let $\+ K$ be the class of compact metric spaces. Note that this is $\Pi^0_3$ with respect to the distance matrix representation of Polish metric spaces, because compactness  is equivalent to being  totally bounded. Isometry of compact spaces is much simpler than in the general case: the points in some fixed  Polish space can serve as invariants.

\begin{theorem}[Essentially Gromov \cite{Gromov:07}, Thm 3.27.5] 
 \[ \cong_i \cap (\+ K \times \+ K )  \le_B \text{id}_\RR. \] \end{theorem}

\begin{proof} Gromov shows that the sequence of sets of $n \times n$  distance matrices that occur in a compact space $X$ constitute a complete set of invariants. Each such matrix is a point in a compact set $K_n(X) \sub \RR^{n^2}$.  The sequence of such compact sets can be represented by a single point in a Polish space, say $\RR$.  \end{proof}

\n \emph{Computable versions.}  The distance matrices $ V= \seq {v_{i.k}}\sN{i,k} \in \mathbb R^{\NN \times \NN}$ form an effectively closed set. They can in fact   be coded as the infinite branches of    a   $\PPI$  tree  $\sub \strcantor$. Such a branch provides yes/no answers to queries  of the form $|v_{i,k} - q| < \epsilon$ for  $i,k \in \NN$, $q \in \QQ^+_0$, and $\epsilon \in \QQ^+$.   

Let $V_e$ denote the $e$-th  partial computable distance matrix. The domain of this partial computable function  grows as long as the data are consistent with being a distance matrix; if they are seen to be not (a $\SI 1$ event) it stops, so that the function is only defined on an initial segment of $\NN$.   Being total is $\PI 2$. 

Let $M_e$ denote the computable metric space given by the $e$-th (total) distance matrix $V_e$.   The following can be  proved by computably reducing  the isomorphism problem for computable graphs by Fokina et al. \cite{Fokina.Friedman.etal:10}.   

\begin{prop} $\{ \la e,k \ra \colon \, M_e \cong_i M_k \}$ is complete   for $\Sigma^1_1$ equivalence relations on $\omega$  with respect to computable reductions. \end{prop}

\begin{prop}[Melnikov and Nies \cite{Melnikov.Nies:13}] The set $C$ of indices for compact computable metric spaces is $\PI 3$.  Isometry is $\PI 2$ within that set, that is, of the form $E \cap C \times C$ where $E$ is a  $\PI 2$ relation. \end{prop}

 \section{II. Having Gromov-Hausdorff distance $0$.}  The following is ongoing work  of  Itai Ben Yaacov, Nies,  and Todor Tsankov. One  thinks of two  metric spaces $X,Y$  as   isometric within error $\epsilon$ if they can be  isometrically embedded into  a third metric space $Z$ in such a way that the usual Hausdorff distance of the two  images is at most $\epsilon$.  ``$X,Y$ isometric within error $0$''   clearly  means that the completions of $X,Y$  are isometric.  The Gromov- Hausdorff distance of $X,Y$ is defined by 
 
 \bc $d_{GH} (X,Y) = \inf \{ \epsilon \colon\, X,Y \, \text{are isometric within error} \, \epsilon\}$. \ec  (Also see Subsection~\ref{ss:bi-Katetov} for an equivalent definition.)
 
 For instance, if we let $X = \{0, 1 \}$ and $Y = \{ 1/4, 3/4 \}$, then \bc $d_{GH} (X,Y)  = 1/4$.    \ec
 
 So, are there  examples of non-isometric spaces $X,Y$ with GH-distance $0$? If so, neither $X$ nor $Y$  can be compact (Gromov). Also there is no positive lower bound on the distance of distinct points, otherwise a near isometry with error less than that bound will be an isometry.  During  the HIM talk,  Nies mentioned an example: let $\mathbb E$ be the unit sphere of the Gurarij space.
 Let $v \in \mathbb E$ be smooth, and $w$ be non-smooth. Let $X = Y =  \mathbb E \cup  \{a,b\}$, with $d_X(a,b)= d_Y(a,b) = 3$.  We set $d_X(v,a) = d_X(v, b) = 3$, and $d_Y(w,a) = d_Y(w, b) = 3$. Any isometry would have to map $v$ to $w$, which is impossible. However, by general properties of the Gurarij space, $d_{GH}(X,Y) = 0$. 
 
\subsection{Fact and more examples for GH-distance}  After Nies' HIM talk, Matatiahou Rubin and Philipp Schlicht constructed further,   simpler examples.  Let $B_X$ denote the unit ball of a Banach space~$X$. 

\begin{prop} There are nonisometric Banach spaces $X,Y$ 

with $d_{GH}(B_X, B_Y)=0$. 
\end{prop} 

To prove this  let $D= \seq {p_i}\sN i$ be a dense sequence of distinct elements in $(1,2)$, say. Let $U_p$ be the 2-dimensional $\RR$ vector space with $\ell_p$ norm. Let $E_D$ be the $c_0$-sum of the spaces $U_{p_i}$. That is, null sequences, with norm   the maximum  of the individual $\ell_{p_i}$ norms. If we have two dense sequences with different sets of members,  the unit balls of the  spaces are at distance 0, but not isometric.

 Let  $\mathbb{G}$ denote  the Gurarij space. By a   (continuous) model-theoretic argument, related to $\aleph_0$-categoricity, one can show that if $X$ is a Banach space and $d_{GH}(B_X,B_{\mathbb{G}})=0$,   then $X$ is isometric to $\mathbb{G}$.

\begin{remark}[Melleray-Schlicht]  \mbox{}

\begin{enumerate} 
\item Any two separable Banach spaces $X,Y$ with $d_{GH}(X,Y)<\infty$ are isometric. 
\item Isometry on Polish spaces reduces to $E_{GH}$. 
\end{enumerate} 
\end{remark} 

\begin{proof} For Banach spaces $X,Y$, $(X,Y)\in E_{GH}$ and $(X,Y)\in E^{\infty}_{GH}$ are equivalent by rescaling (i.e. rescaling the metric space into which we embed the spaces). It follows from the Main Theorem in a paper by  Omladic and Semrl~\cite{Omladic.Semrl:95} that it is sufficient to prove that there is an $\epsilon$-isometry $T\colon X\rightarrow Y$ for some $\epsilon$. For any two perfect Polish spaces with $d_{GH}(X,Y)=0$, we can construct a bijective $\epsilon$-isometry by a straightforward back-and forth argument. 

The second claim follows from a paper of Melleray~\cite{Melleray:07}, where he shows that isometry of Polish spaces reduces to isometry of Banach spaces. 
\end{proof}

The following result of Schlicht and Rubin shows that there is a single $E_{GH}$ class such that the isometry equivalence relation inside is Borel bi-reducible with identity on $\cantor$. In particular, there are continuum many non-isometric spaces with discrete topology that are  mutually  at GH distance 0.

We equip $[0,\epsilon]\times (\omega+1)\times\mathbb{R}$ with the metric $d$ defined by 

$d((x,i,y),(x',i',y'))=1$ if $(x,i)\neq (x',i')$ and 

$d((x,i,y),(x',i',y'))=|y-y'|$ if $(x,i)=(x',i')$. 

\begin{definition} Suppose that $f\colon [0,\epsilon]\rightarrow \omega+1$ is a function. 
\begin{enumerate} 
\item Let $X_f=\{(x,i,0),(x,i,x)\in [0,\epsilon]\times (\omega+1)\times\mathbb{R}\mid i\leq f(x)\}$
with the metric from $[0,\epsilon]\times (\omega+1)\times\mathbb{R}$. 
\item Let $\mathrm{supp}(f)=\{x\in [0,\epsilon]\mid f(x)\neq 0\}$ denote the \emph{support} of $f$. 
\item Let $\mathrm{bound}(f)=\{(x,i)\mid x\in \mathrm{supp}(f),\ i\leq f(x)\}$. 
\end{enumerate} 
\end{definition} 

If $|\mathrm{supp}(f)|=\omega$, then $X_f$ is a discrete countable metric space with distance set $\mathrm{supp}(f)\cup\{0,1\}$. 

\begin{proposition} Suppose that $\epsilon\leq\frac{1}{2}$. Suppose that $f_0\colon [0,\epsilon]\rightarrow \omega+1$ is a function such that $\mathrm{supp}(f_0)$ is a countable dense subset of $[0,\epsilon]$. 
Then  $\mathrm{id}_{{}^{\omega}2}$ is Borel reducible to $\mathrm{Iso}\upharpoonright [X_{f_0}]_{GH}$. 
\end{proposition} 

\begin{proof} Note that for arbitrary functions $f,g\colon [0,\epsilon]\rightarrow \omega+1$, $X_f$, $X_g$ are isometric if and only if $f=g$. 

\begin{claim} Suppose that $f,g\colon [0,\epsilon]\rightarrow \omega+1$ are functions such that $\mathrm{supp}(f)$, $\mathrm{supp}(g)$ are countable dense subsets of $[0,\epsilon]$. Then $d_{GH}(X_f,X_g)=0$. 
\end{claim} 

\begin{proof} 
Note that for every $\delta>0$, there is a bijection $h\colon \mathrm{bound}(f) \rightarrow \mathrm{bound}(g)$ such that $|x-h(x,i)_0|<\delta$ for all $(x,i)\in \mathrm{bound}(f)$. 
Let $h\times\mathrm{id}\colon X_f\rightarrow \mathrm{bound}(g)\times \mathbb{R}$, $(h\times \mathrm{id})(x,i,y)=(h(x,i),y)$. Then $h\times \mathrm{id}$ is distance preserving and $d_H((h\times\mathrm{id})[X_f],X_g)\leq\delta$. Hence $d_{GH}(X_f,X_g)\leq\delta$. 
\end{proof} 

Let $D_q=\{0,q\}$ for $q>0$. 
Suppose that $(q_n,i_n)_{n\in\omega}$ is an enumeration of $\mathrm{bound}(f_0)$ without repetitions. Suppose that $X$ is a complete metric space with $d_{GH}(X,X_{f_0})=0$. Suppose that $0<\delta<1$. Since $d_{GH}(X,X_{f_0})<\frac{\delta}{3}$, $X$ is of the form $X=\bigsqcup_{n\in\omega} X_n^{\delta}$ with
\begin{enumerate} 
\item \label{condition: small GH distance} $d_{GH}(X_n^{\delta},D_{q_n})< \delta$ and 
\item \label{condition: large distance between pieces} $|d(x,y)-1|<\delta$ if $x\in X_m^{\delta}$, $y\in X_n^{\delta}$, and $m\neq n$. 
\end{enumerate} 
Let $X_n=X_n^{\frac{1}{2}}$. Conditions \ref{condition: small GH distance} and \ref{condition: large distance between pieces} imply that for all $\delta<\frac{1}{2}$ and all $n$, there is some $m$ with $X_m^{\delta}=X_n$. 
Hence for each $n$ there is a sequence $(n_i)_{i\in\omega}$ in $\omega$ with $d_{GH}(X_n,D_{q_{n_i}})<\frac{1}{2^i}$. It follows that $1\leq |X_n|\leq 2$. Let $p_n=d(x,y)$ if $X_n=\{x,y\}$. 
Let $A=\{p_n\mid n\in\omega\}$. 

\begin{claim} $d(x,y)=1$ for all $x\in X_m$ and $y\in X_n$ with $m\neq n$. 
\end{claim} 

\begin{proof} This follows from Condition \ref{condition: large distance between pieces} and since for all $\delta<\frac{1}{2}$ and all $k$, there is some $l$ with $X_l^{\delta}=X_k$. 
\end{proof} 

\begin{claim} $A\subseteq [0,\epsilon]$. 
\end{claim} 

\begin{proof} Suppose that $X_n=\{x,y\}$ and $\eta=d(x,y)-\epsilon>0$. Suppose that $X_n=X_m^{\eta}$. This contradicts the fact that $d_{GH}(X_m^{\eta},D_{q_m})< \eta$ by Condition~(\ref{condition: small GH distance}). 
\end{proof} 

\begin{claim} $A$ is dense in $(0,\epsilon)$. 
\end{claim} 

\begin{proof} Suppose that $U\subseteq (0,\epsilon)$ is nonempty and open with $U\cap A=\emptyset$. Suppose that $(q_n-\delta,q_n+\delta)\subseteq U$. This contradicts the fact that $d_{GH}(X_n^{\frac{\delta}{2}},D_{q_n})< \frac{\delta}{2}$ by Condition (\ref{condition: small GH distance}). 
\end{proof} 

Let $f\colon [0,\epsilon]\rightarrow \omega+1$, $f(x)=0$ if $x\notin A$, $f(0)=i$ if $|\{n\in \omega\mid |X_n|=1\}|=i$, and $f(z)=i$ if $|\{n\in \omega\mid \exists x,y\in X_n\ d(x,y)=z\}|=i$ for $z\in (0,\epsilon]$. 
Then $X_f$, $X$ are isometric. 
\end{proof}

\subsection{Bi-Katetov functions} \label{ss:bi-Katetov}One can describe being isometric within error $\epsilon$  without referring to a third space.  
  A~\emph{bi-Katetov function} $f \colon X \times Y \to \R$ is defined as
\[
f(x, y) = d_Z(i(x), j(y)),
\]
where $i, j$ are embeddings into some metric space as above. Equivalently, $f$ is $1$-Lipschitz in both variables and
\begin{align*}
  d_A(x, w) &\leq f(x, y) + f(w, y) \\
  d_B(y, z) &\leq f(x, y) + f(x, z) 
\end{align*}

A bi-Katetov function $f$  can be seen as an approximate isometry. Its  error $q_f$ is given by 
\[ q_f= \max(\sup_x \inf_y f(x, y), \sup_y \inf_x f(x, y)).\] 
By definition this equals the Hausdorff distance of the isometric images above.

For instance, if there is an actual onto isometry $\theta: X \to Y$, we can let $f(x,y) = d_Y(\theta(x), y)$ and obtain the least possible  error $0$.  Conversely, as mentioned above,  if the spaces are complete and  the error is $0$ then  there is an onto  isometry.

Clearly  we  have 	
\[  d_{GH}(X, Y) = \inf_{f  }  q_f,
\]
where $f$ runs through all  the bi-Katetov functions on $X \times Y$.

\begin{remark} \label{rem: extension} f $A \sub X$ and $B \sub Y$, then any bi-Katetov function defined on $A \times B$ extends to one $f'$ defined on $X \times Y$. One uses   amalgamation:
\[
f'(x, y) = \inf_{a\in A,  b \in B} d_X(x, a) + f(a, b) + d_Y(b, y).
\]
\end{remark}

\subsection{Continuous Scott analysis.}

We define approximations to $d_{GH}$  from below by induction on   countable ordinals.

\noindent  Suppose $\bar a = \seq {a_i}_{i < n}$ and $\bar b =  \seq {b_i}_{i < n}$ are enumerated finite metric spaces.  Following   Uspenskii~\cite{Uspenskii:08} define

\[
r_{0, n} (\bar a, \bar b) = \inf_{f \ \text{is bi-Katetov on} \,   \bar a \times \bar b  }   \max_{i< n} f(a_i, b_i).
\]
  Uspenskii gives an  explicit expression for this   in \cite[Proposition 7.1]{Uspenskii:08}:  
  
  \begin{equation}  \label{eq:Uspenskii r0}
r_{0, n} (\bar a, \bar b) = \varepsilon /2  \, \text{ where } \,  \varepsilon  = \max_{i,k < n} |  d(a_i, a_k) - d(b_i, b_k)|.
\end{equation}
(In fact, Uspenskii builds a bi-Katetov function such that  $f(a_i, b_i) = \varepsilon/2 $ for each~$i$.)

 \begin{definition}  Suppose $A$ and $B$ are metric spaces and  $\bar a \in A^n, \bar b \in B^n$. Define by induction on ordinals $\alpha$:
\begin{align*}
  r_{0, n}^{A, B}(\bar a, \bar b) &= r_{0, n}(\bar a, \bar b) \\
  r_{\alpha+1, n}^{A, B}(\bar a, \bar b) &= \max \big(
    \sup_{x \in A} \inf_{y \in B} r_{\alpha, n+1}^{A, B}(\bar a x, \bar b y),
    \sup_{y \in B} \inf_{x \in A} r_{\alpha, n+1}^{A, B}(\bar a x, \bar b y)
    \big) \\
  r_{\alpha, n}^{A, B}(\bar a, \bar b) &= \sup_{\beta < \alpha} r_{\beta, n}^{A, B}(\bar a, \bar b), \quad \text{for } \alpha  \ \text{a  limit ordinal}.
\end{align*}
\end{definition}

%
%

Given a metric space $(X, d)$ and $n \ge 1$, we equip $X^n$ with the ``maximum'' metric $d (\bar u, \bar  v) = \max_{i< n} d(u_i, v_i)$. The following are not hard to check.

 \begin{lemma}  \label{lem: basic props of r's} Fix  separable metric spaces $A, B$ of   finite diameter.
\begin{enumerate}
\item \label{i:1} For each $\alpha$  and each $n$, the functions $r_{\alpha, n}^{A, B}(\bar a, \bar b)$ are 1-Lipschitz in $\bar a$ and $\bar b$.
\item \label{i:2} The functions $r_{\alpha, n}^{A, B}(\bar a, \bar b)$ are  nondecreasing in $\alpha$.
  
\item \label{i:3}There is $\alpha < \omega_1$ after which all  the $r_{\alpha, n}^{A, B}$ stabilize.   \end{enumerate}
\end{lemma}

\begin{theorem}[Ben Yaacov, Nies, Tsankov  2013]   \label{prop:GH approx} Let $A,B$ be separable metric spaces of finite diameter.  Let $\alpha^*$ be such that $r_{\alpha^*+1, n}^{A, B} = r_{\alpha^* , n}^{A, B}$ for each $n$. Then  
  \[
   r_{\alpha^*, 0}^{A, B} = d_{GH}(A,B).
  \]

\end{theorem} 
 
 \begin{proof}
Since $A,B$ are fixed we suppress them in our notations.  Variables $a,a_i$ etc range over $A$, and $b_i$ etc.\ range over $B$. For tuples $\bar a, \bar b$ of length $k$, let   

\[ \delta_k ( \bar a ,  \bar    b)  = \inf_f \{ \max (q_f, \max_{i<k} f(a_i,b_i))\}, \]
where $f$ ranges over bi-Katetov functions on $A \times B$. 
We show that for each $n$ and tuples $\bar a, \bar b$ of length $n$, 
\[ r_{\alpha^*, n} (\bar a, \bar b) = \delta_n ( \bar a ,  \bar    b).\]
For $n=0$ this establishes the theorem.

Firstly, we show by induction on ordinals  $\aaa$  that  
\[ r_{\alpha, n} (\bar a, \bar b)     \le      \delta_n ( \bar a ,  \bar    b)).\]
 The cases  $\aaa = 0$ and $\aaa $ limit ordinal are  immediate.  For the successor case, suppose that   $ \delta_n ( \bar a ,  \bar    b) < s$ via a bi-Katetov function $f$ on $A \times B$. For each  $x \in A$ we can pick $y \in B$ such that $f(x,y) < s$.  Then $\delta_{n+1} (\bar a x, \bar b y) < s$ via the same $f$. Inductively we have $r_{\alpha, n+1} (\bar ax , \bar by)  < s$. Similarly, for each $y \in B$ we can pick $x \in A$ such that $r_{\alpha, n+1} (\bar ax , \bar by)  < s$. This shows that $ r_{\alpha+1, n} (\bar a, \bar b)     \le s$.

Secondly, we verify that 
\[   \delta_n ( \bar a ,  \bar    b))  \le r_{\alpha^*, n} (\bar a, \bar b)         \] 
Let $ r_{\alpha^*, n} (\bar a, \bar b) < t$. We combine  a  back-and-forth argument  with  the compactness of the space of bi-Katetov functions in order to build a bi-Katetov function $f$ with $q_f \le t$ and  $\max_{i < n} f(a_i, b_i ) \le t$.

To do so we extend $\bar a, \bar b$ to dense sequences in $A,B$ respectively.  Let $D \sub A, E \sub B$ be countable dense sets.  Let $\bar u ^k$ denote a tuple of length $k$; in particular, we can write $\bar  a = \bar  a^n$ and $\bar  b = \bar b^n$.  
  We ensure that 
\bc $ r_{\alpha^*, k} (\bar a^k, \bar b^k) < t$ for each $k \ge n$. \ec
Suppose $\bar a ^k , \bar b^k$ have been defined. If $k$ is even, let $a_k$ be the next element in $D$.  Using $r_{\alpha^*+1, k} (\bar a^k, \bar b^k)  = r_{\alpha^*, k} (\bar a^k, \bar b^k)$  we can choose $b_k$ so that  $ r_{\alpha^*, k+1} (\bar a^{k+1}, \bar b^{k+1}) < t$. Similarly, if $k$ is odd, let $b_k$ be the next element in $E$ and choose $a_k$ as required. 

By Lemma~\ref{lem: basic props of r's}(2) we have  $ r_{0, k} (\bar a^k, \bar b^k) < t$ for each $k\ge n$ via some bi-Katetov function $\widetilde f_k$ defined on $\{ a _0, \ldots, a_{k-1}\}  \times \{ b_0, \ldots, b_{k-1} \}$. By Remark~\ref{rem: extension} we  can extend this to a  bi-Katetov function $  f_k$ defined on  $A \times B$. By the compactness of the space of bi-Katetov functions on $A \times B$, viewed as elements of $\mathbb R^{D \times E}$,  there is a subsequence $k_0 < k_1 < \ldots$ such that $\seq {f_{k_u}}$ converges pointwise to a bi-Katetov function $f$. Since  bi-Katetov functions are 1-Lipschitz in both arguments, this implies $\lim_u f_{k_u} (a_p, b_p)  = f(a_p, b_p)$ for  each $p$. Therefore $f(a_p,b_p) \le t$. This implies $q_f \le t$  as required. 
\end{proof}

\begin{definition} The  \emph{continuous Scott rank} of $A$  is the   least $\alpha$ for which
\[
r_{\alpha, n}^{A, A}(\bar a_1, \bar a_2) = r_{\alpha+1, n}^{A, A}(\bar a_1, \bar a_2), \quad \text{for all } n, \bar a_1, \bar a_2 \in A^n.
\]
 \end{definition}

One can define an equivalence relation $E_{GH}$ on the set of  distance matrices $\mathcal M$ by
\[
A {E_{GH}} B \iff d_{GH}(A, B) = 0.
\]
Using the continuous Scott analysis we can show:
\begin{thm} Each equivalence class of $E_{GH}$   is Borel. \end{thm}

\begin{proof} 
By induction each $r_{\alpha, n}$ is a Borel function $\+ M \times \+ M \times \NN^n \times \NN^n \to [0, 1]$.  Next  one needs to prove the following. Fix $A_0 \in \+ M$ and let $\alpha_0 = \rank A_0$.
\begin{itemize}
\item for each $\alpha$, the set $\{B \in \+ M : \rank(B) = \alpha\}$ is Borel;
\item $B {E_{GH}} A_0 \iff \rank(B) = \alpha_0 \lland  r_{\alpha, 0}^{A_0, B} = 0$.
\end{itemize}
\end{proof}

\begin{question} Is  the function $d_{GH}(A_0, \cdot)$   Borel on $\mathcal M$? \end{question}

\section{III. Homeomorphism $\cong_h$}

We collect some results, most of which are proved in \cite[Ch.\ 14]{Gao:09}. For  general Polish metric spaces, $\cong_h$ is merely known to be $\bf \Sigma^1_2$. Homeomorphism of compact metric spaces $X,Y$  is analytic, because   homeomorphisms are uniformly   continuous. In fact, by  the Banach-Stone theorem, we have

\bc $X \cong_h Y \LR  \+ C(X) \cong_i \+ C(Y)$; \ec
so by the aforementioned results of Gao and Kechris  on isometry~\cite{Gao.Kechris:03}, $\cong_h$ on compact metric spaces is Borel reducible to an orbit equivalence relation. 
(A similar argument works for locally compact metric spaces, using $ \+ C_0(X)$, the $C^*$ algebra of continuous functions vanishing at $\infty$; however, for  a Polish metric space, to be  locally compact is known to be properly $\mathbf \Pi^1_1$.)

Camerlo and Gao \cite{Camerlo.Gao:01} proved that graph isomorphism is Borel reducible to homeomorphism of totally disconnected compact metric spaces (i.e., separable Stone spaces).   One notes that countable compact metric spaces $X$ won't work here, because $X$  is  scattered and hence given by the Cantor-Bendixson rank $\alpha$,  together with  the size of the last set $X^{(\alpha)}$.

The main question remains open. 
\begin{question} Determine the complexity with respect to $\le_B$ of $\cong_h$ for compact metric spaces. \end{question} 

In   contrast, in the computable case the complexity is known to be  as large as possible. 
\begin{thm} Homeomorphism of compact computable metric spaces is complete   for $\Sigma^1_1$ equivalence relations on $\omega$  with respect to computable reductions. \end{thm} 

\begin{proof}  Friedman et al.\  \cite{Fokina.Friedman.etal:10} showed this for isomorphism of computable graphs. It can be verified that the   construction  Camerlo and Gao \cite{Camerlo.Gao:01} use for providing their Borel reduction  is effective. Hence, if the given graph is computable, then uniformly in its index they build a compact  computable  metric space. \end{proof}

\section{The complexity of particular isometries} Let us return to the leading questions posed initially. It appears that Questions (b) and (c) are closely connected: 

\vsp

\n {\it It is easy to detect that $X$ is similar to $Y$  $\LR$ 

\hfill  we can determine from $X, Y$ a means via which the similarity holds.}

\vsp 

\n We will provide some evidence for this thesis, first for compact metric spaces, and then for metric measure spaces studied by Gromov~\cite{Gromov:07} and Vershik. For a function $g$, let $g'$ be the halting problem relative to the graph of $g$. 

\begin{theorem}[Melnikov, Nies \cite{Melnikov.Nies:13}]   Let $X,Y$ be compact   metric spaces.  Let $A$ be an   oracle Turing equivalent to the Turing  jump of (the presentation of) $X$ together with  $Y$.  

\bi \item[(a)]  If $X \cong_i Y$ then there is an isometry $g$ such that $g' \leT A''$. 
\item[(b)]  there are isometric compact computable metric spaces $X,Y$ with no isometry $g \leT \Halt$. 
\ei \end{theorem} 
For (a) note that is suffices to build $g$ an isometric embedding. By compactness we can view embeddings as branches  on a subtree $T \sub \strbaire$  with an $A'$ computable bound on the level size. Now apply the low basis theorem relative to $A'$ in order to obtain $g$.  

A metric measure (m.m.) space has the form $T= (X,\mu, d)$ where $(X,d)$ is Polish, $\mu$ a Borel probability measure. We may assume that $\mu U > 0$ for any non-empty open $U$; otherwise, replace $X$ by the least conull closed set.   

\begin{theorem}[Gromov (1997), see \cite{Gromov:07}] Measure-preserving isometry of m.m.\ spaces is smooth. \end{theorem}
Gromov used as invariants the sequence of distributions  $D_n$ of the distance matrix of $n$ randomly chosen points. He used moments to show that $\seq {D_n} \sN n$ is a complete invariant for  the m.m.\ space $T$.  Note that in the subsequence lemma, there is a typo. It should say $\liminf$ there.

 In  1996  Anatoly Vershik \cite{Vershik:98}     gave a  proof as well; also see  the survey~\cite{Vershik:04}. He describes $T$   by the single  invariant $D_T$, the distribution of the distance matrix of a randomly chosen infinite sequence $(x_i)$. More formally,  $D_T$ is the push forward  measure of $d(x_i, x_k)$ on the space $\+ M\sub \RR^{\omega\times \omega}$ of distance matrices. He used a form of the law of large numbers to reconstruct $T$ from $D_T$.  The 2006 paper by Cameron and Vershik is also relevant here. 
 
 We can  give an effective analysis of Vershik's proof. Let $\+ O \sub \omega$ be some $\Pi^1_1$ complete set. 
 \begin{theorem} Suppose $T_1, T_2$ are computable m.m.\ spaces (that is, the measure of Boolean combinations of open  balls is uniformly computable).  Then there is a measure-preserving isometry $\Theta$ such that $\Theta \leT \+ O$.  \end{theorem}

 \begin{proof}   Recall that $\+ M$ is the Polish space of distance matrices.  Following Vershik, we have canonical maps $F_k \colon \, T_k^\omega \to \+ M$, $ \ol x \to \seq {d(x_i, x_k)}\sN {i,k}$. Let $D_k = F_k \mu_k ^ \omega$ be the push forward measure on $\+ M$. This is the distribution of the distance matrix for a randomly picked sequence of points in $T_k$.
 \end{proof} 
 The main source of complexity is that one has to pick an element $r$  in a non-empty $\Sigma^1_1$ class of distance matrices. By Gandy basis theorem, there is such an $r$ with $\mathcal O^r \le_h \mathcal  O$. 
Ongoing work of Melnikov and Nies reduces this complexity to $\Delta^0_3$.

 \part{Other topics}

\section{Yu: A note on the  Greenberg-Montalban-Slaman Theorem}

Greenberg, Montalban and Slaman prove the following theorem.
\begin{theorem}[Greenberg, Montalban and Slaman \cite{GMS11}]\label{theorem: gms 11}
Assume that $\omega_1$ is inaccessible in $L$. For any countable structure $\mathcal{M}$, if the set $A=\{x\mid \exists \mathcal{N}\in L[x](\mathcal{N}\cong \mathcal{M})\}$ contains all the nonconstructible reals,  then $A=2^{\omega}$.
\end{theorem}

We prove that, under the weaker assumption that  $\omega_1^L<\omega_1$, Theorem \ref{theorem: gms 11} remains true for any $\Sigma^1_2$-equivalence relation.

\begin{proof}

For any equivalence relation $E$, reduction $\leq_r$ over $2^{\omega}$ and real $x\in 2^{\omega}$, let $$\spe_{E,r}(x)=\{y\mid \exists z\leq_r y(E(z,x))\}$$ be the $(E,r)$-spectrum of $x$.

Let $E$ be a $\Sigma^1_2$-relation and $x$ be a real so that $\spe_{E,L}(x)\supseteq\{z\mid z\not\in L\}$. Since $E$ is $\Sigma^1_2$, there must be some $\Pi^1_1$-relation $R_0\subseteq (2^{\omega})^3$ so that $$\forall y\forall z (E(y,z)\leftrightarrow \exists s R_0(y,z,s)).$$ By the Shoenfield absoluteness, $$\forall y\forall z (E(y,z)\leftrightarrow \exists s \in L_{\omega_1^{L[y\oplus z]}}[y\oplus z]R_0(y,z,s)).$$ In particular, $$\forall y (E(y,x)\leftrightarrow \exists s \in L_{\omega_1^{L[y\oplus x]}}[y\oplus x]R_0(y,x,s)).$$

Note that, by the assumption, the set $\spe_{E,L}(x)$ is $\Sigma^1_2(x)$ and conull.

 Since $\omega_1^L<\omega_1$, there are conull many $L$-random reals. So the set $\{y\mid \omega_1^L=\omega_1^{L[y]}\}$ is conull. We may also assume that $\omega_1^{L[x]}=\omega_1^L$. Then the set  $\{y\mid \omega_1^{L[x\oplus y]}=\omega_1^L\}$ is also conull.

Then  $z\in \spe_{E,L}(x)\leftrightarrow$
$$  \exists t  \exists y \exists s  (t\mbox{ codes a well ordering } \wedge y\in L_{|t|}[z]\wedge  s \in L_{|t|}[y\oplus x]\wedge R_0(y,x,s)).$$

For any real $t$ coding a well ordering, let $$z\in R_{1,t}\leftrightarrow  \exists y\in L_{|t|}[z] \exists s\in L_{|t|}[y\oplus x]( R_0(y,x,s)).$$ Then $R_{1,t}\subseteq \spe_{E,L}(x)$ is a $\Pi^1_1(t\oplus x)$-set and so measurable. Moreover, if $z$ is $L[x]$-random, then  $z\in  \spe_{E,L}(x)$ if and only if $z\in R_{1,t}$ for some real $t\in L$ coding a well ordering. Since $\mu(\spe_{E,L}(x))=1$ and $\omega_1^L<\omega_1$, there must be some $t\in L$ coding a well ordering so that $\mu(R_{1,t})>0$. Fix such a real $t_0\in L$. Then there must be some formula $\varphi$ in the set theory language so that  the set $$R_{1,t_0,\varphi}=\{z\mid \exists y \exists s\in L_{|t_0|}[y\oplus x](\forall n(n\in y\leftrightarrow L_{|t_0|}[z]\models \varphi(n))\wedge  R_0(y,x,s))\}$$ has positive measure. Then there must be some $\sigma\in 2^{<\omega}$ so that $$\mu(R_{1,t_0,\varphi}\cap [\sigma])>\frac{7}{8}\cdot 2^{-|\sigma|}.$$

Were $x$  constructible, then $R_{1,t_0}\cap [\sigma]$ would contain a constructible real. Now we try to get rid of the parameter $x$.

Let \begin{multline*}S=\{r\mid \mu( \{z\succ \sigma \mid \exists y \exists s\in L_{|t_0|}[y\oplus r]\\ (\forall n(n\in y\leftrightarrow L_{|t_0|}[z]\models \varphi(n))\wedge  R_0(y,r,s))\})>\frac{3}{4}\cdot 2^{-|\sigma|}\}.\end{multline*}

Then $S$ is a $\Pi^1_1(t_0)$-set and every real in $S$ is $E$-equivalent to $x$. Since $x\in S$, we have that $S$ is not empty. Thus there must be some $t_0$-constructible, and so constructible, real in $S$.

This completes the proof.
\end{proof}

By a similar method, one also can prove:
\begin{proposition}\label{proposition: on pi11}
For any $\Pi^1_1$-equivalence relation $E$ and real $x$, if  $\spe_{E,h}(x)\supseteq\{z\mid z\not\in \Delta^1_1\}$, then  $\spe_{E,h}(x)= 2^{\omega}$.
\end{proposition}

Let $MA$ be  Martin's axiom, By a similar method, one also can prove
\begin{theorem}\label{theorem: boldface}
Assume that $MA \lland 2^{\aleph_0}>\aleph_1+\omega_1^L=\omega_1$. Then for any real $x_0, $ and   $\Sigma^1_2(x_0)$-relation $E$, if \ $\spe_{E,L}(x)\supseteq\{z\in 2^{\omega}\mid z\not\in L[x_0]\}$, then $\spe_{E,L}(x)= 2^{\omega}$.
\end{theorem}

So the large cardinal assumption in \cite{GMS11} is unnecessary.

 \newpage

Note: \emph{Normality of a real relative to  non-integral bases, and uniform distribution} has moved  to the 2014 Blog.

\section{Turetsky: $K^X \ge_T X$}
   Proved by Miller and Turetsky, and then vastly simplified by Bienvenu. Let $K$ denote prefix free descriptive string complexity.

\begin{proposition}
For any real $X$, $K^X \ge_T X$.
\end{proposition}

\begin{proof}
$X$ is $X$-trivial.  That is, $K^X(X\!\!\upharpoonright_n) \le K^X(n) + c$.  Note that $K^X$ can compute the tree $\{\sigma \in 2^{<\omega} : K^X(\sigma) \le K^X(|\sigma|) + c\}$.  This tree has finitely many infinite paths, and $X$ is one of them. As an isolated path, $K^X$ can compute $X$.
\end{proof}

\section[Notes on a theorem of Hirschfeldt, Kuyper and Schupp]{Nies: Notes on a theorem of Hirschfeldt, Kuyper and Schupp regarding coarse computation and $K$-triviality}

 Recall that  we write $X \le_{ibT} Y$ if $X \leT  Y$ with use function bounded by the identity.  When building prefix free machines, we use the terminology of  \cite[Section 2.3]{Nies:book}   such as Machine Existence Theorem (also called the Kraft-Chaitin Theorem), bounded request set etc.

\newcommand{\cc}{\mathbf{c}}

Hirschfeldt, Kuyper and Schupp (2013)  proved the following in slightly different language.
\begin{thm} \label{thm:Kuc-g} Let  $Y$ be  a $\DII$ set of  positive effective Hausdorff dimension. There is a cost function $\cc$ such that $A \models \cc$ implies $A \le_{ibT}  D $ for any set $D$ with $\ol \rho (D  \triangle Y) =0$. 

Moreover,  if $Y$ is $\omega$-c.e.,  then $\cc$ can be chosen to be  benign.\end{thm}

 \begin{proof}  
 The  proof given here   extends  a  similar result in \cite{Greenberg.Nies:11}, and also \cite[Thm 5.5]{Nies:nd}. 

 By the hypothesis on $Y$ there is a positive rational $\delta$ such that  
\[ 3\delta < \liminf_n K(Y \uhr n)/n.\]

 Let $\seq {Y_s}$ be a computable approximation of  $Y$.  To help with  building a reduction procedure for   $A \le_{ibT} D$, via the Machine Existence Theorem  we give  prefix-free  descriptions of initial segments $Y_s\uhr e$. On input $x$,  if   at a   stage  $s>x$, $e$ is least such that  $Y(e)$ has changed  between stages $x$ and $s$, then we  still  hope  that    $Y_s \uhr e$  is   the final version of $Y \uhr e$.  So whenever $A(x)$ changes at such a stage $s$, we  give a description of  $Y_s\uhr e$   of length $\lfloor \delta e \rfloor$.     We will define an appropriate  cost function  $\cc$  so that a set  $A $  that obeys~$\cc$     changes little enough that   we can provide all the descriptions  needed.

  To ensure that   $A \le_{ibT} D$,  we define a computation  $ \Gamma(D \uhr x) $  with output $A(x)$ at  the least  stage  $t\ge x $ such that $Y_t \triangle D\uhr e$  has sufficiently few $1$'s     for each $e \le x$.  Then  $A(x) $ cannot change at any stage  $s> t$ (for almost all $x$), for otherwise $Y_s \uhr e$ would receive  a description of length $\lfloor \delta e \rfloor$,  where $e$ is least such that 
  $Y(e)$ has changed  between  $x$ and  $s$.   
   
We give the details.  
Let $H$ denote the binary Bernoulli entropy. Choose  a rational $\beta > 0$ such that $H(\beta )\le  \delta$. This  implies that no more than  $\tp{\delta n}$ strings $v$ of length $n$   have  at most     $ \beta  n$ many $1$'s  (see Wikipedia page on binomial coefficients).  Therefore, for each such string $v$, we have 
\begin{equation} \label{eqn:count 1s}  K(v) \le \delta n + 2 \log n + O(1). \end{equation}

Next   we give a definition of  a   cost
 function $\cc$.   Let   $ \cc(x,s) = 0 $ for each $x\ge s$.  If  $x<s$,  and~$e< x$ is
 least such that
$Y_{s-1} ( e) \neq Y_s  ( e)$,     let 
 \begin{equation} \label{eqn:defn of cY}  \cc(x,s) =
\max( \cc(x,s-1), \tp{-\lfloor \delta e \rfloor} ). \end{equation} 
 We show that $A \models \cc$ implies $A \le_{ibT} D$ for any set $A$. We may suppose that $\cc{\seq{A_s}}  \le 1$. Enumerate a bounded request set~$L$ as follows. When $A_{s-1}(x)\neq A_s(x)$ and $e$ is least such that $e=x$ or  $Y_{t-1} ( e) \neq Y_t  ( e)$ for some $t \in [x,s)$, put the  request $\la  \lfloor \delta e \rfloor +1 , Y_s \uhr e \ra$ into $L$. Then $L$ is indeed a bounded request set.

We  show $A \le_{ibT} D$. 
 Choose $e_0$  with $4 \log (e_0) \le \delta e$ and  for each $e \ge e_0$   the number of $1$'s in $(Y \triangle D) \uhr e$ is at most $\beta e/2$. 
Choose $s_0\ge e_0$ such that $Y_s \uhr{e_0}$ is stable for each  $s \ge s_0$.

Given an input $x\ge s_0$, using~$D$ as an oracle,  compute $t >x$ such that 
\[  \fa e. e_0 \le e \le x  [   (Y_t \triangle D)  \uhr x  \, \text{has at most} \, \beta e/2 \, \text{many}\, 1\text{'s}].\]
We claim that  $A(x) = A_t(x)$. Otherwise  $A_{s}(x) \neq A_{s-1}(x)$ for some $s > t$. Let $e \le x$ be the  largest number such that $Y_r\uhr e = Y _t \uhr e$ for all~$r$, $t < r \le s$.  If $e <  x$  then   $Y(e)$ changes in the  interval $(t,s]$ of stages.  Hence, by the choice of $t\ge s_0$, we cause $K(Y_s \uhr e)  <   \lfloor \delta e  + O(1)$. 
Since $e \ge e_0$, the string $(Y_s \triangle Y) \uhr e$ has at most $\lfloor \beta e \rfloor $ many $1$'s.  Thus,  by (\ref{eqn:count 1s}),
\[ K(Y \uhr e) \le  K(Y_s \uhr e) + K((Y_s \triangle Y)   \uhr e  )   +O(1)\le 2 \delta e + 4 \log e + O(1). \]
This contradicts the definition of $\delta$ for $x$ large enough. 
 \end{proof}

\bibliographystyle{plain}
\def\cprime{$'$}

%
%

\end{document}